\documentclass[11pt, oneside]{article} 

\usepackage[utf8]{inputenc}
\usepackage[T1]{fontenc}
\usepackage{xcolor} 
\usepackage{color} 

\usepackage{amsmath}
\usepackage{amssymb} 
\usepackage{mathrsfs} 
\usepackage{bm} 
\usepackage{bbm} 

\usepackage{amsthm} 
\theoremstyle{plain}
\newtheorem{theorem}{Theorem}[section] 
\newtheorem{corollary}[theorem]{Corollary} 
\newtheorem{definition}[theorem]{Definition}
\newtheorem{condition}{Condition}
\newtheorem{example}[theorem]{Example}
\newtheorem{lemma}[theorem]{Lemma}
\newtheorem{proposition}[theorem]{Proposition}
\newtheorem{remark}[theorem]{Remark}

\usepackage{algorithm2e}
\SetKwInput{KwInputs}{Inputs}
\usepackage[noend]{algpseudocode}
\makeatletter
\def\BState{\State\hskip-\ALG@thistlm}
\makeatother

\RequirePackage[numbers]{natbib}
\RequirePackage[colorlinks,citecolor=blue,urlcolor=blue]{hyperref}
\usepackage{graphicx} 
\usepackage{subcaption}

\usepackage{tikz} 
\usetikzlibrary{positioning ,shadows,matrix}

\usepackage[english]{babel} 

\usepackage{hyperref} 
\hypersetup{
    pdfpagemode={UseOutlines},
    bookmarksopen,
    pdfstartview={FitH},
    colorlinks,
    linkcolor={blue}, 
    citecolor={blue}, 
    urlcolor={blue} 
}

\usepackage[letterpaper,margin=1in]{geometry} 

\usepackage{enumitem} 

\def\iid{\overset{\textnormal{iid}}{\sim}} 

\makeatletter  
\@ifundefined{@currsizeindex} 
  {\RequirePackage{relsize}\let\dolarger\relsize} 
  {\def\dolarger#1{\larger[#1]}} 
\newcommand*\@@bigtimes[2]{\vphantom{\prod} 
  \vcenter{\hbox{\dolarger{4}$\m@th#1\mkern-2mu\times\mkern-2mu$}}} 
\newcommand*\bigtimes{\mathop{\mathpalette\@@bigtimes\relax}\displaylimits} 
\makeatother

\def\iid{\overset{\textnormal{iid}}{\sim}} 

\def\N{\mathbb{N}}\def\R{\mathbb{R}}\def\Z{\mathbb{Z}}

\def\Acal{\mathcal{A}}\def\Bcal{\mathcal{B}}\def\Ccal{\mathcal{C}}\def\Ecal{\mathcal{E}}\def\Fcal{\mathcal{F}}\def\Hcal{\mathcal{H}}\def\Lcal{\mathcal{L}}\def\Ncal{\mathcal{N}}\def\Pcal{\mathcal{P}}\def\Rcal{\mathcal{R}}\def\Scal{\mathcal{S}}\def\Vcal{\mathcal{V}}\def\Wcal{\mathcal{W}}\def\Xcal{\mathcal{X}}\def\Zcal{\mathcal{Z}}

\def\vol{\textnormal{vol}}

\def\eps{\varepsilon}
\newcommand{\epsltrue}{{\eps_{l}^{0}}}
\newcommand{\Oomega}{{\Omega_\eps(\kappa)}}
\newcommand{\aeps}{{A(\eps)}}
\newcommand{\nes}{N_\eps+\alpha_\eps\alpha_n(\eps)}
\newcommand{\nnes}{N_\eps+N_\eps^c+\alpha_\eps}
\newcommand{\pne}{p_N(\eps)}
\newcommand{\peps}{{P(\eps)}}

\numberwithin{equation}{section}

\title{\bf Nonparametric Bayesian intensity estimation for covariate-driven inhomogeneous point processes}
\author{Matteo Giordano, Alisa Kirichenko and Judith Rousseau \\ \\ University of Turin, University of Warwick and University of Oxford}
\date{} 

\begin{document}

\maketitle

\abstract{
This work studies nonparametric Bayesian estimation of the intensity function of an inhomogeneous Poisson point process in the important case where the intensity depends on covariates, based on the observation of a single realisation of the point pattern over a large area. It is shown how the presence of covariates allows to borrow information from far away locations in the observation window, enabling consistent inference in the growing domain asymptotics. In particular, optimal posterior contraction rates under both global and point-wise loss functions are derived. The rates in global loss are obtained under conditions on the prior distribution resembling those in the well established theory of Bayesian nonparametrics, combined with concentration inequalities for functionals of stationary processes to control certain random covariate-dependent loss functions appearing in the analysis. The local rates are derived with an ad-hoc study that builds on recent advances in the theory of Pólya tree priors, extended to the present multivariate setting with a novel construction that makes use of the random geometry induced by the covariates.
}

\bigskip


\bigskip


\tableofcontents

\section{Introduction} \label{sec:Intro}

 A central problem in the statistical analysis of spatial point patterns is to infer the relationship between the point distribution and a collection of covariates of interest. Among the numerous application areas are: the environmental sciences (e.g.~the influence of meteorological conditions on the occurrence of wildfires, \cite{BGMMM20}), geology (e.g.~the prediction of mineral deposit locations from terrain features, \cite{BCST12}), forestry (e.g.~the dependence of biodiversity on the interaction between different plant species, \cite{IMW09}), ecology (e.g.~the preference of plants for specific habitats, \cite{G08}), and epidemiology (e.g.~the raised disease incidence caused by harmful environmental factors, \cite{D90}). Further applications and a general treatment of the theory of point processes can be found in the monographs \cite{DVJ88,MW04,D14}.

	Consider data arising as the realisation of an inhomogeneous point process $N$ over a finite observation window $\Wcal\subset\R^D$, $D\in\N$. The key object determining the occurrence of points is the (first-order) intensity function, namely a map $\lambda:\Wcal\to[0,\infty)$ with the property that, denoting by $N(B)$ the random number of points within any subset $B\subset \Wcal$, 
$$
	E[N(B)]=\int_B \lambda(x)dx.
$$
Hereafter, $N$ will be taken to be of Poisson type. Additionally, assume that a multi-dimensional covariate $Z(x)\in \Zcal\subseteq\R^d$, $d\in\N$, is observed for all $x\in \Wcal$. The relationship between the point pattern and the covariate is then modelled by postulating that
\begin{equation}
\label{Eq:RhoOfZ}
	\lambda(x)=\rho(Z(x)), \qquad x\in \Wcal,
\end{equation}
for some function $\rho :\Zcal\to[0,\infty)$. For example, in the environmental application of \cite{BGMMM20}, the points represent the locations of wildfires across Canada, and two covariates $Z^{(1)}(x)$ and $Z^{(2)}(x)$ are employed reflecting the average temperature and precipitation measurements at each location $x\in \Wcal$, respectively. When the covariate field $Z:=(Z(x), \ x\in \Wcal)$ is itself modelled as being  random, the resulting point process $N$ with intensity \eqref{Eq:RhoOfZ} is `doubly stochastic', and defines an instance of Cox process, \cite{C55}.

	Under this framework, the problem of intensity estimation entails the recovery of the unknown function $\rho$ in \eqref{Eq:RhoOfZ} from an observed realisation of the point process $N$ and the covariate $Z$. Here, we shall consider the nonparametric Bayesian approach to such task. This consists in  assigning to $\rho$ prior distributions $\Pi $ in function spaces and then forming, via Bayes' theorem, the corresponding posteriors $\Pi(\cdot|N,Z)$, which represent the updated beliefs about $\rho$ given the data and furnish point estimates and credible sets for uncertainty quantification; see \cite[Chapter 1]{GvdV17} for an overview. Our main goal is to provide theoretical guarantees, in the form of asymptotic concentration results for the posteriors around the `ground truth' $\rho_0$, under the frequentist assumption that the observed point pattern has been generated according to an intensity given by \eqref{Eq:RhoOfZ} with $\rho = \rho_0$.

%
%
%

\subsection{Related literature}

The spatial statistic literature has largely focused on the parametric approach to covariate-based intensity estimation, both in the frequentist, e.g.~\cite{B78,D90,W07}, and Bayesian literature, e.g.~\cite{RMC09,YL11,ISR12}; see also \cite{MW04,D14} and references therein. For example, the log-Gaussian Cox model, in which \eqref{Eq:RhoOfZ} takes the form $\lambda(x)=\exp(\beta^TZ(x))$ for some $\beta\in\R^d$ and $Z$ a multivariate Gaussian random field, is often used.

	The existing nonparametric frequentist  approaches to intensity estimation are typically built on kernel methods, \cite{K98,BD89,CvL18}. In the present framework with covariates, \cite{G08} proved asymptotic consistency of a covariate-based kernel estimator in the `increasing domain asymptotics' (i.e.~as $\vol(\Wcal)\to\infty$), under the assumption that $Z$ is a stationary and ergodic random field. In their result, the incorporation of ergodic covariates allows to combine the information carried by (potentially distant) locations in the observation window with the same covariate values, thus overcoming the non-vanishing variance issue, and the resulting lack of consistency at the boundary, that afflicts non-covariate-based kernel estimators, cf.~\cite[Section 1]{G08}. Notably, the assumptions on $Z$ maintained in \cite{G08} also circumvent the conditions underpinning the consistency results in \cite{BCST12,BGMMM20}. These are based on the analysis of the induced point pattern in covariate space and require $Z$ to have non-vanishing gradient, a condition that is challenging to verify in the presence of a random multi-dimensional covariate; see Section \ref{Subsec:ErgCov} for further discussion.

	To our knowledge, nonparametric Bayesian procedures for intensity estimation have so far been confined to models without covariates. An early methodological contribution based on weighted gamma process priors  is by \cite{L82}. Computational aspects of posterior inference with Gaussian process priors, combined with various link functions, were investigated in \cite{MSW98,AMM09,PM13} among the others. Several classes of nonparametric priors including gamma, extended gamma and beta processes were employed in \cite{KG97}. Kernel mixture priors were considered in \cite{KS07}, while procedures with spline-based priors and piecewise constant priors were devised in \cite{DMGK01} and \cite{HA98}, respectively.

	The study of the frequentist asymptotic properties of posterior distributions for inhomogeneous Poisson processes has been initiated only more recently, following seminal developments in Bayesian nonparametrics in the early 2000s, \cite{GGvdV00}. Minimax-optimal posterior contraction rates (in $L^2$-distance) for H\"older-smooth intensities were obtained by \cite{BSvZ15} using spline priors with uniform coefficients, assuming that repeated observations of the point pattern over a fixed domain are available. In similar settings, posterior contraction results for Gaussian process priors were proved by \cite{KvZ15,GS13}, and later by \cite{GL19,NM19}. Approaches based on piecewise-constant priors were studied in \cite{GvdMSS20}. Finally, \cite{DRRS17} investigated general Aalen models, developing a novel $L^1$-testing theory founded upon the connection between intensity models and density estimation. The general result was then employed to obtain optimal $L^1$-posterior contraction rates under smoothness or shape constraints. 

%
%
%

\subsection{Our contributions}

	Here, we shall provide the first analysis of the frequentist asymptotic properties of posterior distributions for covariate-driven inhomogeneous Poisson point processes. As in \cite{G08}, we work in the growing domain asymptotics, which is natural for spatial statistics, cf.~\cite[Section 2]{G08}, assuming that $\vol(\Wcal)\to\infty$ and that a single realisation of the processes $N$ and $Z$ is observed over $\Wcal$. With the goal of providing a comprehensive investigation and appealing to a wider range of applications, we shall study the problem under various representative assumptions for the covariates and employ several classes of priors of interest. Our results and proof strategies also provide a template to extend the analysis beyond the considered settings. Below is an overview of our main contributions.

\begin{itemize}
\item In the first part of the paper, we study posterior contraction in global loss functions. We prove a general result (Theorem \ref{Theo:GenRandCov}) in an `empirical' (i.e.~covariate-dependent) $L^1$-distance, holding under a minimal stationarity assumption on $Z$ and abstract prior conditions that resemble the well-understood ones for density estimation \citep{GGvdV00}. The theorem is based on the established testing approach \citep{GvdV17}, building on \cite{DRRS17} to construct suitable tests for the empirical $L^1$-distance.

\item We illustrate the general theory for the empirical $L^1$-distance considering both bounded and unbounded covariate spaces. In the former case, we obtain posterior contraction rates for Gaussian and hierarchical Gaussian priors (Propositions \ref{Theo:GPRandLoss} and \ref{th:Gauss:adapt:lambda}). For unbounded covariates, we employ location mixture of Gaussians priors (Proposition \ref{Theo:MixRandLoss}).

\item In Section \ref{Subsec:ErgCov}, we refine the above analysis to derive posterior contraction rates in standard (non-random) $L^1$-distances under additional ergodicity assumptions on the covariates. These are common in spatial statistics, cf.~\cite[Sec.~10.2]{DVJ88}, \cite[Sec.~3]{G08} and \cite[Sec.~2.3]{C15}, and allow to combine the previous results for the empirical $L^1$-distance with concentration inequalities that stem from recent work of \cite{DG20AHL,DG20ALEA}. In our main result, Theorem \ref{Theo:GaussWav}, we show that for the important case where $Z$ is a stationary and ergodic Gaussian random field (i.e.~when $N$ is a Gaussian-Cox process), truncated Gaussian wavelet priors achieve minimax-optimal $L^1$-posterior contraction rates. Outside of the frequentist consistency results of \cite{G08} and of \cite{BCST12,BGMMM20}, where  convergence rates are not investigated, we are not aware of other comparable studies. We further consider piecewise constant covariates arising as Poisson random tessellations, as well as categorical covariates, for which we obtain parametric rates of convergence.

\item In the second part of the paper, we turn to the study of local posterior contraction rates. We design procedures in the spirit of the P\'olya tree priors for probability density functions (e.g.~\cite[Chapter 3]{GvdV17}), whose asymptotic properties of concentration, adaptation and uncertainty quantification have recently been studied in \cite{C17,CM21}. See also \cite{M17}. Here, we provide a novel construction tailored to intensity functions defined on multi-dimensional covariate spaces, carefully accommodating the random geometry induced by the covariates in the tree-generating partition and the marginal prior distributions. We then prove that the resulting posteriors adapts to the local regularity of the ground truth, with adaptive optimal point-wise contraction rates under a local smoothness assumption. Drawing a connection with the first part of the paper, we further show that the P\'olya tree prior also achieves optimal $L^1$-contraction rates under global regularity conditions.
\end{itemize}

	The rest of the paper is organised as follows. Section \ref{Sec:StatProbl} provides preliminaries, notation, and a precise description of the statistical problem. Section \ref{Sec:GlobLoss} presents the results in global loss functions. The local and global analysis of the Pólya tree priors is developed in  Section \ref{sec:localrates}. In Section \ref{sec:Summary}, we provide a summary of the paper, further discussion and an outlook on several related research questions. All the proofs, alongside auxiliary results and further background material are contained in the Supplement, \cite{Supplement}.

\section{Covariate-driven Poisson processes and Bayesian inference}
\label{Sec:StatProbl}

%
%
%

\subsection{Preliminaries and notation}
\label{Subsec:Prelim}

Throughout, $\Wcal\equiv \Wcal_n\subset\R^D$, $D\in \N$, is a  nonempty compact set, which we will often refer to as the `observation window'. The subscript $n$ serves as an indication that we will consider asymptotic properties as  $\vol( \Wcal_n)\to\infty$.

	Given a measure space $(\Xcal,\mathfrak{X},\mu)$ and a normed vector space $(\Vcal,\|\cdot\|_\Vcal)$, we will denote by $L^p(\Xcal,\mu;\Vcal)$, $1\le p\le \infty$, the usual Lebesgue spaces of $\Vcal$-valued functions defined on $\Xcal$ with integrable $p^\textnormal{th}$-power, equipped with norm
$$
	\|f\|^p_{L^p(\Xcal,\mu;\Vcal)}:=\int_{\Xcal}\|f(x)\|_\Vcal^p d\mu(x), 
	\qquad 1\le p<\infty,
$$
replaced by the essential supremum of $\|f\|_\Vcal$ over $\Xcal$ if $p=\infty$. When $\Vcal=\R$, we will use the shorthand notation $L^p(\Xcal,\mu)\equiv L^p(\Xcal,\mu;\R)$; further, if $\Xcal\subseteq\R^m$ for some $m\in\N$ and $\mu$ equals the Lebesgue measure $dx$ on $\R^m$, we will write $L^p(\Xcal;\Vcal)\equiv L^p(\Xcal,dx;\Vcal)$ and $L^p(\Xcal)\equiv L^p(\Xcal,dx;\R)$.

	For $\Xcal\subseteq\R^m$, let $C(\Xcal)$ be the space of continuous real-valued functions defined on $\Xcal$, equipped with the supremum norm $\|\cdot\|_\infty$. For $\alpha>0$, let $C^\alpha(\Xcal)$ be the usual H\"older space of $\lfloor \alpha \rfloor$-times continuously differentiable functions on $\Xcal$, with $(\alpha - \lfloor \alpha \rfloor)$-H\"older continuous $\lfloor \alpha \rfloor^{\textnormal{th}}$ derivative, equipped with norm $\|\cdot\|_{C^\alpha}$. For $\alpha\in\N$, let $H^\alpha(\Xcal)$ be the usual Sobolev space of functions with square-integrable $ \alpha^{\textnormal{th}}$ derivative, with norm $\|\cdot\|_{H^\alpha}$. For non-integer $\alpha>0$, $H^\alpha(\Xcal)$ can be defined via interpolation, e.g.~\cite{LM72}. When no confusion may arise, we will omit the dependence of the function spaces on the underlying domain, writing for example $C^\alpha$ for $C^\alpha(\Xcal)$.

	We will use the symbols $\lesssim, \ \gtrsim$ and $\simeq$ for one- and two-sided inequalities holding up to multiplicative constants. The minimum and maximum between $a,b\in\R$ will be denoted by $a\land b$ and $a\vee b$. For two real sequences $(a_n)_{n\ge1}$, $(b_n)_{n\ge1}$, we will write $a_n=o(b_n)$ if $a_n/b_n\to0$ as $n\to\infty$, and $a_n=O(b_n)$ if $a_n/b_n\lesssim 1$ for all sufficiently large $n$. The $\varepsilon$-covering number of a set $\Theta$ with respect to a semi-metric $\Delta$, denoted by $\Ncal(\varepsilon;\Theta,\Delta)$, is the minimal number of balls of $\Delta$-radius $\varepsilon>0$ needed to cover $\Theta$. 

%
%
%

\subsection{The observation model}
\label{Subsec:Obs}

On the ambient space $\R^D, \ D\in\N$, consider covariates given by a (jointly measurable) random field $Z := (Z(x), \ x\in\R^D)$ with values in a (measurable) subset $\Zcal\subseteq\R^d$, $d\in\N$, and an increasing sequence of compact observation windows $(\Wcal_n)_{n\ge1}$ satisfying $\Wcal_n\subseteq \Wcal_{n+1}\subset \R^D$. Throughout, we will assume that $Z^{(n)}$ has almost surely bounded sample paths and that it can be viewed as a Borel measurable map in the space $L^\infty(\Wcal_n;\Zcal)$, with law denoted by $P_{Z^{(n)}}$. We also maintain that $Z$ is stationary and denote by $\nu $ its invariant distribution; thus, $Z(x)\sim \nu $ for each $x\in\R^D$, cf.~Remark \ref{Rem:Stationarity}.

	We assume that we observe, for some $n\in\N$, a realisation of the covariates on $\Wcal_n$, denoted by $Z^{(n)}:=(Z(x), \ x\in \Wcal_n)$, and of a random point process on $\Wcal_n$ arising, conditionally given $Z^{(n)}$, as an inhomogeneous Poisson point process with first-order intensity function
\begin{equation*}
	\lambda_\rho^{(n)}(x) := \rho(Z(x)), \qquad x\in\Wcal_n,
\end{equation*}
 for some unknown (measurable and) bounded  function $\rho : \Zcal \to [0,\infty)$. Formally, we may represent the point pattern as
\begin{equation}
\label{Eq:PointProc}
	N^{(n)} \overset{d}{=} (X_1,\dots, X_{N_n}), 
	\qquad N_n|Z^{(n)}\sim\textnormal{Po}(\Lambda^{(n)}_\rho),
	\qquad
	X_i|Z^{(n)} \iid \frac{\lambda^{(n)}_\rho(x)dx}{\Lambda^{(n)}_\rho},
\end{equation}
with $\Lambda^{(n)}_\rho := \int_{\Wcal_n}\lambda^{(n)}_\rho(x) dx$. In other words, $N^{(n)}$ is a Cox process, \cite{C55}, directed by the random measure $\lambda^{(n)}_\rho(x)dx$.

	The joint law $P^{(n)}_{\rho}$ of the data vector $D^{(n)}:=(N^{(n)},Z^{(n)})$ arising as above is absolutely continuous with respect to the law $P^{(n)}_1$ corresponding to the standard Poisson case, with likelihood given by
\begin{equation}
\label{Eq:Likelihood}
	L_n(\rho):=\frac{dP^{(n)}_\rho}{dP^{(n)}_1}(D^{(n)}) \propto 
	\exp\left\{\int_{\Wcal_n} \log(\rho(Z(x))) d N^{(n)}(x) 
	-\int_{\Wcal_n}\rho(Z(x))dx\right\};
\end{equation}
see e.g.~\cite[Theorem 1.3]{K98}. We will write $E_\rho^{(n)}$ for the expectation with respect to $P_\rho^{(n)}$.

\begin{remark}[Continuous observations of the covariates]
The availability of the covariate value $Z(x)$ at each location $x\in \Wcal_n$ is a standard methodological assumption in the spatial statistics literature, e.g.~\cite{G08,BCST12,BGMMM20}, and is often realistic, cf.~the geological application in \cite[Section 8]{BCST12}. In other practical scenarios, the covariates may need to be interpolated from discrete measurements $Z(x_1),\dots,Z(x_K)$ over a finite grid $x_1,\dots,x_K\in\Wcal_n$. Provided that the grid is sufficiently fine, the numerical approximation error is then typically disregarded. Below, among various concrete instances, we will also consider specific discrete covariate fields, arising as piecewise constant processes associated to Poisson random tessellations; see Section \ref{Subsec:PoissRates}.
\end{remark}

\begin{remark}[Stationarity of the covariates]\label{Rem:Stationarity}
Stationarity is an often-used (testable, \cite{BS17}) assumption for spatially correlated data, e.g.~\cite{S99,R00,C15}. In the present setting, it entails the mild requirement that the statistics of the covariates remain homogeneous across the observation window. Intuitively, this is necessary for any `accumulation of information' to even be possible in the growing domain asymptotics, where only a single realisation of the point pattern and the covariate is available. Some related discussion can be found in Section \ref{Sec:Infill}.
\end{remark}

\begin{remark}[Deterministic covariates]
In the present paper, we are concerned with random covariates. However, some of the results to follow hold, conditionally given $Z^{(n)}$, under minimal assumptions, and can be extended with minor modifications to the case where $Z:\R^D\to\R^d$ is a fixed deterministic field. This will be explicitly pointed out where relevant (cf.~the discussion after Theorem \ref{Theo:GenRandCov} and before Theorem \ref{Theo:LocalRates}).
\end{remark}

%
%
%

\subsection{Nonparametric Bayesian inference on $\rho$}
\label{Subsec:BayesApproach}

We are interested in the problem of estimating the unknown covariate-based intensity function $\rho:\Zcal\to[0,\infty)$ based on data $D^{(n)}=(N^{(n)},Z^{(n)})$ from model \eqref{Eq:PointProc}. We consider the nonparametric Bayesian approach, assigning to $\rho$ prior distributions $\Pi $ on measurable collections $\Rcal\subset L^\infty(\Zcal)$ of non-negative functions defined on $\Zcal$. By Bayes'~formula (e.g.~\cite[p.7]{GvdV17}), the posterior distribution of $\rho|D^{(n)}$ is given by
$$
	\Pi(A|D^{(n)}) = \frac{\int_A L_n(\rho)d\Pi(\rho)}
	{\int_{\Rcal} L_n(\rho')d\Pi(\rho')},
	\qquad A\subseteq\Rcal \ \textnormal{measurable},
$$
with $L_n$ the likelihood in \eqref{Eq:Likelihood}. Nonparametric Bayesian intensity estimation has so far been investigated, methodologically and theoretically, only in models without covariates; see in \cite{L82,MSW98,AMM09,BSvZ15,KvZ15} and references therein.

	Our primary focus is on the asymptotic posterior concentration properties, assuming data $D^{(n)}\sim P^{(n)}_{\rho_0 }$ generated by some fixed true $\rho_0 \in\Rcal$, and studying under what conditions $\Pi(\cdot|D^{(n)})$ concentrates around $\rho_0 $ in the infinitely informative data limit. As the amount of information is determined by the volume of $\Wcal_n$, we will work in the `growing domain' regime wherein $\vol(\Wcal_n):=\int_{\Wcal_n}dx\to\infty$. Without loss of generality and for notational convenience, we will write $\vol(\Wcal_n)=n$; in other words the subscript $n$ has no particular meaning other than indicating that $\vol(\Wcal_n) \to\infty$, and all the asymptotic (in $n$) results below should be thought of as being in terms of the quantity $\vol(\Wcal_n)$.

%
%
%
%
%

\section{Posterior contraction rates in global loss}
\label{Sec:GlobLoss}

In this section we present our concentration results with respect to global ($L^1$-type) loss functions. A precise quantification of the speed of convergence is provided in the form of posterior contraction rates, namely positive sequences $\epsilon_n\to0$ such that, for $M>0$ large enough,
$$
	E^{(n)}_{\rho_0}\left[\Pi(\rho : \Delta(\rho,\rho_0) > M\epsilon_n |D^{(n)} )\right]\to0
$$
as $n\to\infty$, where $\Delta$ is a semi-metric between intensities in $\Rcal$.

%
%
%

\subsection{Posterior contraction rates in a covariate-dependent $L^1$-distance}
\label{Subsec:GenTheo}

We first derive a general result based on an `empirical' (covariate-dependent) loss function, holding under standard assumptions on the prior distribution resembling the well-understood ones for density estimation (\cite{GGvdV00}). This will constitute a key building block towards the more refined posterior contraction rates obtained in Sections \ref{Subsec:GaussRates} and \ref{Subsec:PoissRates} below.

	Recalling that $Z$ is stationary with invariant distribution $\nu $, for non-negative valued functions $\rho,\rho_0\in L^1(\Zcal,\nu)$, set $M_\rho:=\int_\Zcal\rho(z)d\nu(z)$, define the probability density function $\bar\rho(z):=\rho(z) / M_\rho, $ $z\in\Zcal$, 
and let $M_{\rho_0}$ and $\bar\rho_0$ be similarly constructed. For positive numbers $(\epsilon_n)_{n\ge1}$, define the neighbourhoods
\begin{equation*}
\begin{split}
	B_{n,0}(\rho_0) &:= \left\{ \rho\in\Rcal: \textnormal{KL}_\nu 
	(\bar \rho_0, \bar \rho) \le \epsilon_n^2 ,\ |M_\rho-M_{\rho_0}|
	\le \epsilon_n \right\};\\
	B_{n,2}(\rho_0) &:= B_{n,0}(\rho_0)
	\cap \left\{ \rho\in\Rcal : \int_\Zcal \bar \rho_0(z)
	\log^2\left( \frac{ \bar\rho_0(z) }{\bar \rho(z) } \right) d\nu(z)\le\epsilon_n^2  
	\right\},
\end{split}
\end{equation*}
where $\textnormal{KL}_\nu(\bar\rho_0, \bar\rho) := \int_\Zcal \bar\rho_0(z) \log (\bar\rho_0(z)/\bar\rho(z))d\nu(z) $ is the Kullback-Leibler divergence between the probability density functions $\bar\rho_0$ and $\bar\rho$.

\begin{theorem}\label{Theo:GenRandCov}
Let $\rho_0\in L^\infty(\Zcal)$ be non-negative valued. Consider data $D^{(n)}=(N^{(n)},Z^{(n)})\sim P^{(n)}_{\rho_0}$ from the observation model \eqref{Eq:PointProc} with $\rho=\rho_0$ and $Z$ a stationary, almost surely locally bounded, random field with invariant measure $\nu $. Assume that the prior $\Pi $ satisfies for some positive sequence $\epsilon_n\to0$ such that $n\epsilon_n^2\to\infty$,
\begin{equation}
\label{Eq:SmallBall}
	\Pi( B_{n,2}(\rho_0) ) \ge e^{-C_1 n\epsilon_n^2},
\end{equation}
for some $C_1>0$. Further assume that there exist measurable sets $\Rcal_n\subseteq\Rcal$ such that
\begin{equation}
\label{Eq:Sieves}
	\Pi(\Rcal_n^c)\le e^{-C_2 n\epsilon_n^2}, 
	\qquad C_2 := 2+2\|\rho_0\|_{L^\infty(\Zcal)}+C_1,
\end{equation}
and
\begin{equation}
\label{Eq:MetricEntropy}
	\log \Ncal(\epsilon_n;\Rcal_n,\|\cdot\|_{L^\infty(\Zcal)})\le C_3 n\epsilon_n^2,
\end{equation}
for some $C_3>0$. Then, for all sufficiently large $M>0$, as $n\to\infty$,
\begin{align}
\label{Eq:GenRandCov}
	E_{\rho_0}^{(n)}
	\Bigg[\Pi\Big(\rho \in\Rcal_n: \frac{1}{n} \|\lambda_\rho^{(n)} - \lambda_{\rho_0}^{(n)}\|_{L^1(\Wcal_n)} &> M\epsilon_n
	\Big| D^{(n)}\Big)\Bigg]
	=O(1/(n\epsilon_n^2)).
\end{align}
\end{theorem}

	Note that the boundedness assumption on $\rho_0$ implies $\rho_0\in L^1(\Zcal,\nu)$. Theorem \ref{Theo:GenRandCov} establishes sufficient conditions on $\Pi$ to obtain posterior contraction in the covariate-dependent $L^1$-metric
\begin{equation}
\label{Eq:RandLoss}
	\frac{1}{n} \|\lambda_\rho^{(n)} - \lambda_{\rho_0}^{(n)}\|_{L^1(\Wcal_n)}
	=\frac{1}{n} \int_{\Wcal_n}|\rho(Z(x))-\rho_0(Z(x))|dx.
\end{equation}
Recalling the convention $\vol(\Wcal_n)=n$, this entails a decay rate for the average distance between samples $\rho\sim\Pi(\cdot|D^{(n)})$ and the ground truth $\rho_0$, over the covariate surface $Z^{(n)}=(Z(x),\ x\in\Wcal_n)$. In Section \ref{Subsec:GaussRates} and \ref{Subsec:PoissRates} below, we will show how, for concrete instances of random covariate fields, Theorem \ref{Theo:GenRandCov} can be leveraged to derive optimal posterior contraction rates in standard $L^1$-metrics.

\begin{remark}[Small ball probability lower bound]\label{Rem:GenSmallBall}
Inspection of the proof of Theorem \ref{Theo:GenRandCov} (Section \ref{Sec:ProofGenRandCov} in the Supplement) shows that if $B_{n,2}(\rho_0) $ is replaced by $B_{n,0}(\rho_0)$ in Assumption  \eqref{Eq:SmallBall} then, by substituting $C_2$ in \eqref{Eq:Sieves} with an slowly increasing sequence $C_n\to\infty$, Theorem \ref{Theo:GenRandCov} remains valid with the constant $M>0$ in \eqref{Eq:GenRandCov} replaced by any sequence $M_n\to\infty$. Furthermore, a sufficient condition for \eqref{Eq:SmallBall} to hold is to have a similar lower bound for the sup-norm neighbourhood $\{ \rho\in\Rcal:\|\rho - \rho_0\|_{L^\infty(\Zcal)} \leq \epsilon_n \}$. This also allows to prove a version of Theorem \ref{Theo:GenRandCov} for deterministic or non-stationary designs, but we refrain to pursue such extensions here.
\end{remark}

\begin{remark}[Empirical loss function]\label{Rem:GenSmallBall}
The empirical loss function \eqref{Eq:RandLoss} naturally appears in our analysis via the testing approach to posterior contraction rates, \cite{GN16,GvdV17}, which we pursue by constructing tests with exponentially decaying Type-II error probabilities for alternatives separated in the empirical $L^1$-metric, building on ideas of \cite{DRRS17}, cf.~Lemma \ref{Lem:Tests} in the Supplement. Given such tests, the proof of Theorem \ref{Theo:GenRandCov} follows by standard arguments under the prior conditions \eqref{Eq:SmallBall} - \eqref{Eq:MetricEntropy}, similar to those typically used in the density estimation literature, \cite{GGvdV00}.
\end{remark}

	In the next sections, we present applications of Theorem \ref{Theo:GenRandCov} to two different families of prior distributions corresponding to different scenarios for the covariate process $Z$, namely when: (i) $Z$ takes values in a bounded set $\Zcal$; and (ii) $\Zcal$ is unbounded. Further see Section \ref{sec:finiteZ} for the case of finite covariate spaces.

%
%
%

\subsubsection{Bounded covariate spaces; Gaussian process priors}
\label{Subsec:RandGPRates}

We first consider the scenario where $\Zcal$ is compact. For example, this is the case if the random covariate field is given by a transformation $Z(x):=\Phi(\tilde Z(x))$ of a stationary process $\tilde Z:=(\tilde Z(x), \ x\in\R^D)$ with values in $\R^d$ via a bijective (differentiable) function $\Phi:\R^d\to\Zcal$; cf.~Section \ref{Subsec:GaussRates} below. Without loss of generality, take $\Zcal=[0,1]^d$. An example of transformation $\Phi$ is given by
\begin{equation}
\label{Eq:BoundedPhi}
	\Phi(\tilde z):=(\phi(\tilde z_1),\dots,\phi(\tilde z_d)),
	\qquad \tilde z\equiv (\tilde z_1,\dots,\tilde z_d)\in\R^d,
\end{equation}
where $\phi:\R\to[0,1]$ is a smooth cumulative distribution function. We then assign to $\rho:[0,1]^d\to[0,\infty)$ a  prior distribution $\Pi $ constructed as the law of the random function
\begin{equation}
\label{Eq:GPPrior}
	\rho_W(z) := \eta(W( z)), \qquad z\in[0,1]^d,
\end{equation}
where $W:=(W(z), \ z\in[0,1]^d)$ is a centred Gaussian process with almost surely bounded sample paths and $\eta:\R\to(0,\infty)$ is a fixed, strictly increasing and bijective link function.

	In applying Theorem \ref{Theo:GenRandCov} with the above prior, we allow for a large class of Gaussian processes $W$, whose law $\Pi_W$ we require to satisfy the following mild smoothness condition. See e.g.~\cite[Chapter 2]{GN16} or \cite[Chapter 11]{GvdV17} for background information on Gaussian processes and measures.

\begin{condition}\label{Cond:GPCondition}
	For $\alpha>0$, let $\Pi_W$ be a centred Gaussian Borel probability
measure on the Banach space $C([0,1]^d)$, with reproducing kernel Hilbert space (RKHS) $\Hcal_W$ continuously embedded into the Sobolev space $H^{\alpha+d/2}([0,1]^d)$.
\end{condition}

\begin{proposition}\label{Theo:GPRandLoss}
Assume that $\rho_0 =\rho_{w_0}$ for some $w_0\in C([0,1]^d)$ and $\eta:\R\to(0,\infty)$ a fixed, smooth, strictly increasing, uniformly Lipschitz and bijective function. Consider data $D^{(n)}\sim P^{(n)}_{\rho_0}$ from the observation model \eqref{Eq:PointProc} with $\rho=\rho_0$ and $Z$ a stationary random field with values in $[0,1]^d$. Let the prior $\Pi $ be given by \eqref{Eq:GPPrior} with $W$ a Gaussian process on $[0,1]^d$ satisfying Condition \ref{Cond:GPCondition} for some $\alpha>0$ and RKHS $\Hcal_W$. For positive numbers $(\epsilon_n)_{n\ge1}$ such that $\epsilon_n\to0$ and $\epsilon_n \gtrsim n^{-
\alpha/(2\alpha+d)}$, assume that there exists a sequence $(w_{0,n})_{n\ge1}\subset\Hcal_W$ satisfying
\begin{equation*}
	\|w_0 - w_{0,n}\|_\infty\lesssim \epsilon_n;
	\qquad \|w_{0,n}\|^2_{\Hcal_W}\lesssim n\epsilon_n^2.
\end{equation*}
Then, for all sufficiently large $M>0$, as $n\to\infty$,
$$
	E_{\rho_0}^{(n)}
	\Bigg[\Pi\Big(\rho : 
	\frac{1}{n} \|\lambda_\rho^{(n)} - \lambda_{\rho_0}^{(n)}\|_{L^1(\Wcal_n)} > M\epsilon_n
	\Big| D^{(n)}\Big)\Bigg]
	\to 0.
$$
\end{proposition}

\begin{remark}[Gaussian priors in the intensity estimation literature]
Gaussian priors, transformed via positive link functions, are a popular methodological choice for nonparametric Bayesian intensity estimation, \cite{MSW98,AMM09,PM13}. These priors have been shown to yield minimax-optimal posterior contraction rates in non-covariate based models, \cite{KvZ15,GS13,GL19,NM19}. For instance, the exponential link $\eta(\cdot)=\exp(\cdot)$ is used in the celebrated log-Gaussian Cox model, \cite{MSW98}, while more restrictive Lipschitz and (bounded) logistic-type instances were used in \cite{GS13} and \cite{AMM09,KvZ15}, respectively.
\end{remark}

	The proof of Proposition \ref{Theo:GPRandLoss} is in the supplementary Section \ref{Sec:ProofGPRandLoss}. Examples satisfying Condition \ref{Cond:GPCondition} are stationary Gaussian processes with polynomially-tailed spectral measures (e.g.~the popular Matérn processes, cf.~\cite[Section 11.4.4]{GvdV17}), as well as series priors defined on basis functions spanning the Sobolev scale, such as the wavelets employed below.

\begin{example}[Truncated Gaussian wavelet series priors]\label{Ex:GPWav}
Given an orthonormal tensor product wa-velet basis $(\psi_{lk},\ l\ge1, \ k=1,\dots, 2^{ld})$ of $L^2([0,1]^d)$, formed by $S$-regular (with $S\in\N$ sufficiently large) compactly supported and boundary corrected Daubechies wavelets (see e.g.~\cite[Chapter 4.3]{GN16} for details), consider the Gaussian  expansion
\begin{align}
\label{Eq:GaussWavPrior}
	W(z) 
	:= \sum_{l=1}^{L}\sum_{k=1}^{2^{ld}}2^{-l(\alpha+\frac{d}{2})} 
	g_{lk}\psi_{lk}(z), 
	\qquad z\in[0,1]^d,
	\qquad g_{lk}\iid N(0,1),
\end{align}
with $\alpha>0$ and $L\equiv L_n\in\N$ chosen so that $2^{L_n}\gtrsim n^\frac{1}{2\alpha +d}$. The law $\Pi_W$ of $W$ satisfies Condition \ref{Cond:GPCondition} with RKHS equal to a wavelet projection space and RKHS norm equivalent to $\|\cdot\|_{H^{\alpha+d/2}}$. If further $\rho_0=\rho_{w_0}$ for some $w_0\in C^\beta([0,1]^d)$, any $\beta>0$, then the conditions of Theorem \ref{Theo:GPRandLoss} can be verified with $\epsilon_n = n^{-(\alpha\wedge \beta)/(2\alpha+d)}$. See the proof of Theorem \ref{Cond:GPCondition} for additional details.
\end{example}

The rate $n^{-\beta/(2\beta+d)}$ is known to be minimax-optimal in models without covariates for $\beta$-regular ground truths, see \cite{BSvZ15}, and this conclusion can be extended to the present setting as well. Example \ref{Ex:GPWav} shows that the optimal rate can be achieved by a smoothness-matching prior (i.e.~with $\alpha=\beta$), as expected from the general theory of Gaussian priors, \cite{vdVvZ08}.  The following results, proved in Section \ref{Sec:ProofAdaptGP} of the Supplement, shows that adaptation to smoothness (up to an extra log-factor) is possible via a standard hierarchical construction.

\begin{proposition}\label{th:Gauss:adapt:lambda}
Assume that $\rho_0 =\rho_{w_0}$ for some $w_0\in C^\beta([0,1]^d)$, $\beta>0$, and $\eta:\R\to(0,\infty)$ a fixed, smooth, strictly increasing, uniformly Lipschitz and bijective function. Consider data $D^{(n)}\sim P^{(n)}_{\rho_0}$ from the observation model \eqref{Eq:PointProc} with $\rho=\rho_0$ and $Z$ a stationary random field with values in $[0,1]^d$. Let the prior $\Pi $ be given by \eqref{Eq:GPPrior} with $W$ the following hierarchical Gaussian wavelet expansion,
\begin{align*} 
	W(z) 
	&:= \sum_{l=1}^{L}\sum_{k=1}^{2^{ld}}g_{lk}\psi_{lk}(z), 
	\qquad z\in[0,1]^d,
	\qquad g_{lk}\iid N(0,1),\\
	L&\sim \Pi_L, 
	\qquad \Pi_L(L=l)\propto e^{-C_L2^{ld}l}, \qquad C_L>0.
\end{align*}
Set $\epsilon_n= n^{-\beta/(2\beta+d)}\log n$. Then, for all sufficiently large $M>0$, as $n\to\infty$, 
$$
	E_{\rho_0}^{(n)}
	\Bigg[\Pi\Big(\rho : 
	\frac{1}{n} \|\lambda_\rho^{(n)} - \lambda_{\rho_0}^{(n)}\|_{L^1(\Wcal_n)} 
	>M \epsilon_n
	\Big| D^{(n)}\Big)\Bigg]
	\to 0.
$$
\end{proposition}

\begin{remark}[Extensions]
	In this section, we have considered compact covariate spaces, which is the natural setting for the asymptotic analysis of Gaussian priors, e.g.~\cite[Chapter 11]{GvdV17}. Following our proofs, extensions to the unbounded case could be studied under a decay-to-zero condition for the ground truth, using metric entropy bounds for function spaces on unbounded domains as in \cite{nickl2007bracketing}. Also, adaptation could be pursued with other hierarchical priors than the one from Proposition \ref{th:Gauss:adapt:lambda}, e.g.~as in \cite{LvdV07,vWvZ15, G23}; further see \cite[Chapter 11.6]{GvdV17}. These are interesting avenues for future work. In the next section, adaptive rates for unbounded covariates will be obtained using mixtures of Gaussians priors.
\end{remark}

%
%
%

\subsubsection{Unbounded covariate spaces; nonparametric mixtures of Gaussians priors}\label{subsec:mixtprior}

We next consider the case of unbounded covariates, taking $\Zcal=\R^d$ for notational simplicity. Drawing from the connection between intensity and density estimation, we model the function $\rho:\R^d\to[0,\infty)$ as a nonparametric location mixture of Gaussians (with non-normalised mixing measure). Denote by $\varphi_\Sigma$ the probability density function of the centred $d$-variate normal distribution with covariance matrix $\Sigma\in\R^{d, d}$. Let the prior $\Pi $ be given by the law of the random function
\begin{equation}\label{prior:mixtG1}
	\rho (z) = \int_{\mathbb R^d} \varphi_{\Sigma}(z-\mu) dQ(\mu) 
	=( \varphi_\Sigma \ast Q)(z), 
	\qquad z\in\R^d, \qquad Q(\cdot)  = \sum_{j=1}^J q_j \delta_{\mu_j}(\cdot),
\end{equation}
where $ (q_1, \cdots, q_J)  := M(p_1, \cdots, p_J)$ with $\sum_{j=1}^J p_j=1$ and $M\sim \Pi_M$ (independently). Specifically, we consider two types of priors of the above form, where either:
\begin{enumerate}
\item[(i)] $Q$ is a Gamma process with finite base measure $\bar \alpha \Pi_\mu$, where $\bar \alpha >0$ and $\Pi_\mu$ is probability measure on $\R^d$. In this case, $J= \infty$.

\item[(ii)]  $J \sim \Pi_J$ is a Poisson or Geometric random variable, and conditionally given $J$, for some $\alpha_J>0$ such that $J\alpha_J<\bar\alpha$ for some $\bar\alpha>0$, and for some probability measure $\Pi_\mu$ on $\R^d$,
$$
	\mu_1,\dots,\mu_J\iid \Pi_\mu; \qquad
	\quad (p_1 , \cdots, p_J) \sim \mathcal D(\alpha_J, \cdots, \alpha_J),
$$
with $p_1, \cdots, p_J$ independent of $\mu_1,\dots,\mu_J$.

\end{enumerate} 
Further, independently of $Q$, the covariance matrix $\Sigma$ is assigned an inverse-Wishart   prior if $d>1$, and a square-rooted inverse-Wishart prior (under which $\Sigma^{1/2}$ is inverse-Wishart distributed) when $d=1$.

 	The following result shows that the resulting posterior distributions attain (up to a log factor) adaptive posterior contraction rates in the empirical $L^1$-metric \eqref{Eq:RandLoss}. The proof is given in Section \ref{Sec:ProofMixtures} of the Supplement.

\begin{proposition}\label{Theo:MixRandLoss}
Assume that $\rho_0\in C^\beta(\R^d), \ \beta>0$, satisfies $\inf_{z\in\R^d}\rho_0(z)>0$. Consider data $D^{(n)}\sim P^{(n)}_{\rho_0}$ from the observation model \eqref{Eq:PointProc} with $\rho=\rho_0$ and $Z$ a stationary, almost surely locally bounded, random field with absolutely continuous invariant measure $\nu $. Consider a location mixture of Gaussians prior $\Pi $ defined as in \eqref{prior:mixtG1}, where $Q$ is of type either (i) or (ii), and such that $\Pi_\mu$ and $\Pi_M$ have positive  and continuous densities
$$
	\pi_\mu (\mu) \propto e^{-c_1|\mu|^{c_2}},
	\qquad \mu\in\R^d;
	\qquad \pi_M(s) \lesssim  e^{-s^{a_M}},
	\qquad s>0,
$$
for constants $c_1,c_2, a_M>0$. Set $\epsilon_n=n^{-\beta/(2\beta+d)}(\log n)^t$, $t>0$. Then, for some sufficiently large $t>0$, as $n\to\infty$,
$$
	E_{\rho_0}^{(n)}
	\Bigg[\Pi\Big(\rho : \frac{1}{n} \|\lambda_\rho^{(n)} 
	- \lambda_{\rho_0}^{(n)}\|_{L^1(\Wcal_n)}> \epsilon_n
	\Big| D^{(n)}\Big)\Bigg]
	\to 0.
$$
\end{proposition}

\begin{remark}[Mixture priors in the literature]
Mixture of Gaussians priors are known to achieve (up to log-factors) adaptive minimax posterior contraction rates in density estimation; see \cite{KRvdV10,STG13,CDB17}, or further \cite{BRR22}, where the target is supported near a possibly unknown manifold. To our knowledge, in the context of point processes, mixture priors had been previously considered, in models without covariates, only with kernels with compact support, \cite{DRRS17,sulem2024bayesian}.
\end{remark}

\begin{remark}[Extensions]
	In Proposition \ref{Theo:MixRandLoss}, we consider an isotropic multivariate construction similar to the density estimation procedures of \cite{KRvdV10} for one-dimensional domains, and of \cite{STG13} in the multivariate case. Hybrid location-scale mixtures, e.g.~as in \cite{NR17,BRR22}, represent an interesting alternative for future research. Further, we note that in Proposition \ref{Theo:MixRandLoss} the ground truth $\rho_0$ is assumed to be uniformly bounded away from zero. Extensions to other tail behaviours could be derived along the lines of \cite[Condition (C2)]{KRvdV10}, or \cite{STG13,CDB17}. Finally, mixtures with compactly supported kernels (e.g.~of the Beta family) could be employed to tackle the case of bounded covariates, following \cite{DRRS17,sulem2024bayesian}.
\end{remark}

%
%
%

\subsection{$L^1$-posterior contraction rates under ergodicity}
\label{Subsec:ErgCov}

While the empirical $L^1$-loss \eqref{Eq:RandLoss} is useful for controlling `prediction errors', it is only indirectly related to the main target $\rho$. A common strategy in the existing literature to relate estimation of $\lambda_\rho^{(n)}=\rho\circ Z^{(n)}$ to inference on $\rho$ is to consider the point pattern on $\Zcal$ induced by transforming the original points via $Z^{(n)}$, cf.~\cite{BCST12,BGMMM20}. However, this typically requires $Z^{(n)}$ to have non-vanishing gradient, a condition that is violated in many applications where similar covariate values are repeatedly observed across the domain. In this case, modelling $Z$ as a stationary and ergodic random field is instead often realistic, \cite{DVJ88,G95,R00,C15}. In \cite{G08}, ergodicity of the covariates was the key enabling factor for their asymptotic consistency result for kernel estimators.

	Here, we adopt the latter perspective and assume, as in \cite{G08}, that $Z=(Z(x),\ x\in\R^D)$ is a stationary ergodic random field. We then make the following crucial observation on the empirical $L^1$-loss \eqref{Eq:RandLoss}. In view of the (spatial) ergodic theorem, for increasing (regularly-shaped) domains $\Wcal_n\to\R^D$, the covariate-dependent distance $n^{-1}\|\lambda^{(n)}_\rho - \lambda^{(n)}_{\rho_0}\|_{L^1(\Wcal_n)}=\textnormal{vol}(\Wcal_n)^{-1}\int_{\Wcal_n}|\rho(Z(x))-\rho_0(Z(x))|dx$ almost surely satisfies, as $n\to\infty$,
\begin{equation}
\label{Eq:ErgodicTheo}
\begin{split}
	\frac{1}{n}\|\lambda^{(n)}_\rho - \lambda^{(n)}_{\rho_0}\|_{L^1(\Wcal_n)}
	&\to \int_\Zcal | \rho(z) - \rho_0(z)|d\nu(z)=\|\rho - \rho_0\|_{L^1(\Zcal,\nu)},
\end{split}
\end{equation}
where $\nu $ is the stationary distribution of $Z$. Heuristically, combining this with the typical central limit scaling (of order $\vol(\Wcal_n)^{-1/2}= n^{-1/2}$) for the variance of spatial averages of `sufficiently mixing' ergodic random fields, motivates the expectation that Theorem \ref{Theo:GenRandCov} should imply posterior contraction also in the standard non-random metric $\|\cdot\|_{L^1(\Zcal,\nu)}$, at the same rate $\epsilon_n$ obtained in empirical $L^1$-loss. In Sections \ref{Subsec:GaussRates} and \ref{Subsec:PoissRates} below, we will pursue this argument for two important classes of random covariate fields, namely Gaussian processes and Poisson random tessellations, showing that, uniformly over sets $\{\rho\in \Rcal_n : n^{-1}\|\lambda^{(n)}_\rho - \lambda^{(n)}_{\rho_0}\|_{L^1(\Wcal_n)} \le M \epsilon_n\}$ of posterior concentration, 
\begin{equation}
\label{Eq:RandLossConc}
	\Big|\frac{1}{n}\|\lambda^{(n)}_\rho - \lambda^{(n)}_{\rho_0}\|_{L^1(\Wcal_n)}
	-	\|\rho - \rho_0\|_{L^1(\Zcal,\nu)}\Big|= o_{P_{Z^{(n)}}}(\epsilon_n),
	\qquad n\to\infty.
\end{equation}
To do so, we will employ precise concentration inequalities for spatial averages, derived from recent results of \cite{DG20AHL,DG20ALEA}, and more broadly from the literature on concentration of measure, \cite{AS94,BL97,GNO15}.

	Further, we will also consider the case of discrete covariate spaces with finite cardinality (i.e.~categorical covariates), for which simpler arguments will lead to near-parametric rates of convergence.
	
%
%
%

\subsubsection{Ergodic Gaussian covariates; Gaussian wavelet priors}
\label{Subsec:GaussRates}

We present our main result in global loss, in the setting where the random covariate field is a (transformation of a) stationary and ergodic Gaussian process. Specifically, we assume that $Z$ arises as described in the following condition, cf.~Remark \ref{Rem:GaussCov} below.

\begin{condition}\label{Cond:BoundGaussCov}
Let $\tilde Z^{(h)}:=(\tilde Z^{(h)}(x), \ x\in\R^D),\ h=1,\dots,d,$ be independent, almost surely locally bounded, centred and stationary Gaussian processes with integrable covariance functions $K^{(h)}\in L^1(\R^D)$, where $K^{(h)}(x):=\textnormal{Cov}(Z^{(h)}(x),$ $Z^{(h)}(0))$, $x\in\R^D$. Further assume without loss of generality that $K^{(h)}(0)=1$. Let the covariate process $Z=(Z(x), \ x\in \R^D)$ be given by $Z(x) := \Phi(\tilde Z(x))$, where $\Phi:\R^d\to[0,1]^d$ is defined as in  \eqref{Eq:BoundedPhi} with $\phi$ the standard normal cumulative distribution function.
\end{condition}

	Note that since $\tilde Z(x)$ has standard $d$-variate normal distribution for all $x\in\R^D$, the stationary distribution $\nu $ of $Z$ equals the Lebesgue measure on $[0,1]^d$.

	We proceed modelling the unknown $\rho:[0,1]^d\to[0,\infty)$ via truncated Gaussian wavelet expansions as in \eqref{Eq:GaussWavPrior}. For these, under Condition \ref{Cond:BoundGaussCov}, we derive exponential concentration inequalities for spatial averages based on multivariate extensions of well-known functional inequalities for Gaussian measures, jointly with a chaining argument from empirical process theory to achieve the required uniformity over the wavelet prior support. See Section \ref{Subsec:MultGPCov} in the Supplement for details. We then combine this with Theorem \ref{Theo:GenRandCov} to derive posterior contraction rates in the standard $L^1$-metric. To pursue this argument, pathologically-shaped domains must be ruled out, and we here assume that for some $r\in(0,1/2)$
\begin{align}
\label{Eq:SpatialWn}
	\left[-rn^{1/D},rn^{1/D}\right]^D\subseteq \Wcal_n,
	\qquad \left(\text{vol}(\Wcal_n)=n\right),
	\qquad n\in\N.
\end{align}
Equation \eqref{Eq:SpatialWn} guarantees that $\Wcal_n$ grows uniformly in all spatial directions, and is similar to the shape regularity condition underpinning the analysis in \cite{G08}. Due to some technical aspects of the proof, we also need to restrict the prior construction to smooth, uniformly Lipschitz, strictly increasing link functions $\eta$ with bounded and uniformly Lipschitz derivative $\eta'$, similar to those in \cite{GS13}. Further, we require the left tail of $\eta'$ to satisfy, for some $w_0<0$ and $a_\eta>0$,
\begin{equation}
\label{Eq:LinkTail}
	\eta'(w)\ge \frac{1}{|w|^{a_\eta}}, \qquad  w<w_0,
\end{equation} 
cf.~Example \ref{Ex:Link} below.

\begin{theorem}\label{Theo:GaussWav}
Let $\Wcal_n\subset\R^D$ be a measurable and bounded set satisfying \eqref{Eq:SpatialWn}. For some $\beta>1+d(1+a_\eta/2)$, assume that $\rho_0 \in C^\beta([0,1]^d)\cap H^\beta([0,1]^d)$ satisfies $\inf_{z\in[0,1]^d}\rho_0(z)>0$. Consider data $D^{(n)}\sim P^{(n)}_{\rho_0}$ from the observation model \eqref{Eq:PointProc} with $\rho=\rho_0$ and $Z$ a stationary random field satisfying Condition \ref{Cond:BoundGaussCov}. Let the prior $\Pi $ be given by \eqref{Eq:GPPrior} for $\eta:\R\to(0,\infty)$ as above and $W$ the truncated Gaussian wavelet series in \eqref{Eq:GaussWavPrior} with $\alpha=\beta$. Set $\epsilon_n=n^{-\beta/(2\beta+d)}$. Then, for all sufficiently large $M>0$, as $n\to\infty$,
$$
	E_{\rho_0}^{(n)}
	\Bigg[\Pi\Big(\rho : \|\rho - \rho_0\|_{L^1([0,1]^d)} > M\epsilon_n
	\Big| D^{(n)}\Big)\Bigg]
	\to 0.
$$
\end{theorem}

	The proof can be found in Section \ref{Sec:ProofMainGaussWav} of the Supplement. Theorem \ref{Theo:GaussWav} establishes that the posterior contracts, in standard $L^1$-distance, around any $\beta$-H\"older regular ground truth at optimal rate $n^{-\beta/(2\beta+d)}$. The minimal smoothness requirement $\beta>1+d(1+a_\eta/2)$ is closely tied to the chaining technique developed in the proof and is likely an artefact of the employed argument.

\begin{example}[Link function]\label{Ex:Link}
Consider the link function
$$
	\eta(w)=h*g(w),
	\qquad w\in\R,
	\qquad g(w) = \frac{1}{1-w}1_{\{w<0\}}+(1+w)1_{\{w\ge0\}},
$$
where $h:\R\to[0,\infty)$ is smooth, compactly supported and such that $\int_{\R}h(u)du=1$. It can readily be checked that $\eta$ satisfies the requirements of Theorem \ref{Theo:GaussWav} with $a_\eta=2$. For general exponents, similar link functions can be constructed as well, cf.~Example 24 in \cite{GN20}.
\end{example}

\begin{remark}[Ergodic Gaussian covariates]\label{Rem:GaussCov}
Gaussian processes are ubiquitously employed in spatial statistics to model random fields, motivated by both `physical' considerations and methodological convenience, e.g.~\cite[Section 2.3]{C15}. They represent the main example for the covariate process in the kernel consistency result of \cite{G08}. The standard normal assumption in Condition \ref{Cond:BoundGaussCov} amounts to the common practice of standardising the covariates, cf.~\cite[Section 3.2.1]{G08}. The transformation through $\Phi$ can then be regarded as a convenient `pre-processing' step that compactifies the covariate space $\Zcal$; it implies no loss of information in view of the invertibility of $\Phi$.

	For stationary Gaussian processes, a sufficient condition for ergodicity is that the covariance function decays to zero, e.g.~\cite{C15}, entailing the intuitive (and often realistic) scenario where the covariates at distant locations are progressively less correlated. In particular, this is implied in Condition \ref{Cond:BoundGaussCov} by the integrability of the kernels $K^{(h)}$. The latter is in fact a slightly stronger requirement that is verified as long as $K^{(h)}$ is bounded and satisfies $K^{(h)}(x) \lesssim |x|^{-D+\kappa}$, any $\kappa>0$, for large $|x|$. This is closely related to some sufficient conditions guaranteeing strongly mixing properties, e.g.~\cite[Proposition 1.4]{DG20ALEA}.
\end{remark}

\begin{remark}[Extensions]
The specific choice of the `pre-processing' map $\Phi$ in Condition \ref{Cond:BoundGaussCov} is mostly for convenience; extensions to other bounded transformations requires mostly notational changes through the proofs.

	In Theorem \ref{Theo:GaussWav}, we focused for conciseness on non-adaptive priors with fixed regularity. Adaptation could be pursued with hierarchical procedures, along the lines of Proposition \ref{th:Gauss:adapt:lambda}.

	An inspection of the concentration inequalities employed in the proof of Theorem \ref{Theo:GaussWav} motivates extensions to general priors that yield asymptotic concentration over functions with uniformly bounded gradient. We provide additional details and a general formulation in Section \ref{Subsec:BoundGaussRates} of the Supplement.
\end{remark}

%
%
%

\subsubsection{Poisson random tessellations; Gaussian process priors}
\label{Subsec:PoissRates}

The second class of stationary ergodic covariate processes we consider are piecewise constant random fields originating from random tessellations. These naturally appear as models for discrete heterogeneous structures in a variety of applications, \cite{T02,MS07}. Throughout this section, we focus on the case $d=1$, that is univariate covariate processes.

\begin{definition}\label{Cond:PoissTess}
For $\Xi := (\xi_r)_{r\ge1}$ a standard Poisson point process on $\R^D$, let $(\Ccal_r)_{r\ge1}$ be the associated Voronoi tessellation, that is, the random partition of $\R^D$ given by
$$
	\Ccal_r := \left\{ x\in \R^D : |x - \xi_r| = \inf_{r\ge1 }|x - \xi_r|\right\}.
$$
For some probability measure $\nu $ supported on $\Zcal\subseteq\R$, let the covariate process $Z = (Z(x), \ x\in \R^D)$ be given by
$$
	Z(x) = \sum_{r\ge1}\zeta_r1_{\Ccal_r}(x), 
	\qquad \zeta_r\iid \nu,
	\qquad x\in\R^D.
$$
\end{definition}

	In other words, $Z$ is piecewise constant over the random cells $(\Ccal_r)_{r\ge1}$, with cell-wise values $\zeta_r$ randomly sampled from $\nu $, which is accordingly the stationary distribution. Ergodicity of $Z$ is implied by the standard Poisson process assumption on $\Xi$, which entails that large sets in the random tessellation occur with small probability. For covariate processes arising as in Definition \ref{Cond:PoissTess}, a combination of results in \cite{DG20AHL,DG20ALEA} yields exponential concentration inequalities for spatial averages, uniformly over sup-norm balls, cf.~Section \ref{Subsec:PoissTessCov} in the Supplement. We jointly employ these with Theorem \ref{Theo:GenRandCov} to derive posterior contraction rates in the metric $\|\cdot\|_{L^1(\Zcal,\nu)}$.

		To illustrate, assume that $\Zcal=[0,1]$ and, as in Section \ref{Subsec:RandGPRates}, model $\rho:[0,1]\to [0,\infty)$ with a prior constructed as in \eqref{Eq:GPPrior} for $W\sim\Pi_W$ a draw from a centred Gaussian Borel probability measure on $C([0,1])$. We  show below that, under suitable tuning, this yields minimax-optimal posterior contraction rates. The proof can be found in Section \ref{Subsec:SupplPoissRates} of the Supplement.

\begin{theorem}\label{Cor:GaussPoissTess}
Let $\Wcal_n\subset\R^D$ be a measurable and bounded set satisfying \eqref{Eq:SpatialWn}. Assume that $\rho_0 =\rho_{w_0}$ for some $w_0\in C^\beta([0,1]^d)$, $\beta>0$, and $\eta:\R\to(0,\infty)$ a fixed, smooth, strictly increasing, uniformly Lipschitz and bijective function satisfying $\sup_{w\in[0,1]}\eta(w)\le M_1, \ M_1>0$.
Consider data $D^{(n)}\sim P^{(n)}_{\rho_0}$ from the observation model \eqref{Eq:PointProc} with $\rho=\rho_0$ and $Z$ a stationary random field constructed as in Definition \ref{Cond:PoissTess} for some probability measure $\nu $ supported on $[0,1]$. Let the prior $\Pi $ be given by \eqref{Eq:GPPrior} with $W$ a Gaussian process on $[0,1]$ satisfying Condition \ref{Cond:GPCondition} with $\alpha=\beta$ and RKHS $\Hcal_W$. Set $\epsilon_n=n^{-\beta/(2\beta+1)}$, and further assume that there exists a sequence $(w_{0,n})_{n\ge1}\subset\Hcal_W$ satisfying
$$
	\|w_0 - w_{0,n}\|_\infty\lesssim \epsilon_n;
	\qquad \|w_{0,n}\|_{\Hcal_W}^2\lesssim n\epsilon_n^2.
$$
Then, for all sufficiently large $M_2>0$, as $n\to\infty$,
$$
	E_{\rho_0}^{(n)}
	\Bigg[\Pi\Big(\rho : \|\rho - \rho_0\|_{L^1([0,1],\nu)} > M_2\epsilon_n
	\Big| D^{(n)}\Big)\Bigg]
	\to 0.
$$
\end{theorem}

\begin{remark}[Bounded link function]
The restriction to bounded link functions in Theorem \ref{Cor:GaussPoissTess} is motivated by the concentration inequalities for Poisson random tessellations, which are uniform over sup-norm balls, cf.~Section \ref{Subsec:PoissTessCov} in the Supplement. A widely used example of such link function is the sigmoidal one; see e.g.~\cite{AMM09,KvZ15}.

		While required in our analysis, the boundedness of the link is a relatively mild assumption. In particular, as the constant $M_1$ in Theorem \ref{Cor:GaussPoissTess} only affects the factor $M_2$ premultiplying the rate $\epsilon_n$, the boundedness assumption can be removed by letting $M_1\to\infty$ arbitrarily slow, replacing $M_2$ with an arbitrarily slow sequence $M_n\to\infty$. 
\end{remark}

\begin{remark}[General Gaussian process priors]
As remarked after Proposition \ref{Theo:GPRandLoss}, Condition \ref{Cond:GPCondition} is verified by a large class of Gaussian priors, including -- but not limited to -- the previously employed wavelet expansions. Compared to Theorem \ref{Theo:GaussWav}, Theorem \ref{Cor:GaussPoissTess} thus allows for greater flexibility in the prior. This is mainly due to the difference in the concentration inequalities for ergodic Gaussian covariates and Poisson random tessellations. For the former, we have leveraged specific wavelet properties to obtain the required uniformity over the prior support. Extending Theorem \ref{Theo:GaussWav} to more general Gaussian priors is a major technical challenge. 
\end{remark}

%
%
%
%
%

\subsubsection{Finite covariate spaces} \label{sec:finiteZ}

We conclude this section with the case where the covariate process $Z$ only takes finitely-many distinct values, namely $\Zcal = \{z_1,\dots,z_K\}\subset(\R^d)^K$ for some fixed $K\in\N$. This results in a parametric model where the unknown function $\rho:\Zcal\to[0,\infty)$ is identified with the vector $(\rho(z_k))_{k=1}^K\in [0,\infty)^K$. The following result shows that any prior satisfying a mild polynomial tail condition achieves the parametric rate (up to a logarithmic factor). See Section \ref{Sec:ProofParametric} in the Supplement for the proof.

\begin{proposition}\label{prop:finiteCov}
For $\Zcal = \{z_1,\dots,z_K\}\subset(\R^d)^K$, assume that $\rho_0=(\rho_0(z_k))_{k=1}^K\in [0,\infty)^K$ satisfies $\min_{k=1,\dots,K}$ $ \rho_0(z_k) >0$. Consider data $D^{(n)}=(N^{(n)},Z^{(n)})\sim P^{(n)}_{\rho_0}$ from the observation model \eqref{Eq:PointProc} with $\rho=\rho_0$ and $Z$ a stationary ergodic random field taking values in $\mathcal Z$ with invariant measure $\nu $ such that $\min_{k=1,\dots,K} \nu(z_k) >0$. Let $\Pi $ be an absolutely continuous prior distribution on $[0,\infty)^K$, satisfying for all $k=1,\dots,K$ and some constant $b>0$,
\begin{equation}
\label{Eq:DiscrPriorCond}
	\Pi\left( \rho=(\rho(z_k))_{k=1}^K\in [0,\infty)^K: |\rho(z_k) | > u\right) 
	\leq C u^{-b}, \qquad u>0.
\end{equation}
Then, for all sufficiently large $M>0$, as $n\to\infty$,
$$
	E_{\rho_0}^{(n)}
	\Bigg[\Pi\left( \rho: \sum_{k=1}^K | \rho(z_k) - \rho_0(z_k) | \nu(z_k)  
	> M \frac{ \sqrt{\log n} }{ \sqrt{n }} \Bigg| D^{(n)}\right)\Bigg] \to 0.
$$

\end{proposition}

	For example, Proposition \ref{prop:finiteCov} applies to covariate processes arising  as the stratification (in a finite number of classes) of a stationary ergodic Gaussian processes, like the ones from Section \ref{Subsec:GaussRates}.

%
%
%
%
%

\section{Local and global analysis of Pólya tree-type priors}
\label{sec:localrates}

We now turn to studying a family of Pólya tree-type priors, constructed in the spirit of the density estimation procedures of \cite{CM21,M17}. We will show that these posses a remarkable ability to model heterogeneous intensities, achieving optimal point-wise rates that adapt to the local smoothness, and also optimal $L^1$-rates under a global regularity assumption. The design of spatially-adaptive procedures is a delicate problem in Bayesian nonparametrics; for example, Gaussian priors are known to be unsuited (even within hierarchical models) to recovering inhomogeneous or structured functions, \cite{RR23,GRSH22,agapiou2024laplace}.

%
%
%

\subsection{P\'olya tree priors for covariate-based intensity functions}

To construct P\'olya tree priors in the present setting, assume that the covariate space $\Zcal$ is compact; for notational simplicity, take $\Zcal=[0,1]^d$. For the purpose of deriving point-wise rates this poses no additional restrictions on the covariate process $Z$, since if $\Zcal$ were unbounded then transforming $Z$ via a smooth bijective map $\Phi:\R^d\to[0,1]^d$, constructed e.g.~as in Section \ref{Subsec:RandGPRates}, would imply that estimating $\rho$ at any point $z_0\in\R^d$ is in fact equivalent to estimating $\rho\circ\Phi^{-1}$ at $\Phi(z_0)\in[0,1]^d$.

%

\subsubsection{Prior construction}
\label{Sec:PTConstruction}

	Take a deterministic sequence of binary partitions $\Pcal^{(L_n)}:=(\Pcal_l,\ 1\le l\le L_n)$ of $[0,1]^d$, with $L_n\in\N$ to be chosen below, where $\Pcal_0:=[0,1]^d$ and each partition $\Pcal_l$ is obtained by splitting every set of $\Pcal_{l-1}$ into two subsets; see Example \ref{Ex:Diadic}. Accordingly, $|\Pcal_l|=2^l$. Throughout, we will use the following notations, borrowing common terminology for tree-type priors. For any $1\le l\le L_n$,  let $\Pcal_l:=(B_\eps,\ \eps\in\Ecal_l)$ be the partition at level $l$, with $\Ecal_l:=\{0,1\}^l$ the corresponding set of indices and $\mathcal E_{[n]}:=\cup_{l=1}^{L_n}\Ecal_l$. Each set $B_\eps\in\Pcal_l$ has two children $B_{\eps-},B_{\eps+}\in\Pcal_{l+1}$, with $\eps- := (\varepsilon,0),\eps+ := (\varepsilon, 1) \in\Ecal_{l+1}$. Also, for any $\eps\in\Ecal_l$, we denote the index of its parent bin by $P(\eps)\in\Ecal_{l-1}$ and that of its twin by $\aeps\in\Ecal_l$. In other words, $B_\eps$ and $B_{\aeps}$ are the children of $B_{P(\eps)}$.

	Given $\Pcal^{(L_n)}$, define for all $\varepsilon\in\Ecal_{[n]}$,
$\alpha_n(\eps) :=  \mu_n(B_{\eps}) / \mu_n(B_{P(\eps)})$, where $\mu_n(A):=n^{-1}\int_AZ(x)dx$, $A\subseteq\Wcal_n$ measurable, is the (normalised) push-forward of the Lebesgue measure under $Z$. Note that  $\alpha_n(\eps+) + \alpha_n(\eps-) =1$ by construction. For $\rho:[0,1]^d\to[0,\infty)$, we write in slight abuse of notation
\begin{align}
\label{Eq:BasicPT}
	\rho(\eps)\equiv \rho(B_\eps):=\int_{\Wcal_n}\rho_0(Z(x))1_{\{Z(x)\in B_\eps\}}d\mu_n(x);
	\qquad \rho^*:=\int_{\Wcal_n}\rho(Z(x))d\mu_n(x),
\end{align}
and let
$$
	y_{\eps+}
	:=\frac{\rho(\eps+)}{\rho(\eps)\alpha_n(\eps+)};
	\qquad y_{\eps-}
	:=\frac{\rho(\eps-)}{\rho(\eps) \alpha_n(\eps-)};
	\qquad \bar y_\eps:=y_\eps\alpha_n(\eps).
$$
Note that if $\rho$ is constant over $B_\eps$, then $y_{\eps+} =1= y_{\eps-}$; accordingly, as $l$ increases, $y_{\eps}\to1$ for all $\eps \in \Ecal_l$ provided that $\rho$ is continuous. In particular, the following is a piecewise constant approximation of $\rho$,
\begin{equation}
\label{prior:rho}
	\tilde \rho(z) = \frac{\rho(\eps_l(z))}{\mu_n(B_{\eps_l}(z))} = \rho^*\prod_{l\le L_n}y_{\varepsilon_{l}(z)},
	\qquad z\in [0,1]^d;
	\qquad \tilde\rho(\eps_{L_n}(z)) :=  \rho^* \prod\limits_{l \le L_n }\bar y_{\eps_l(z)}, 
\end{equation}
having denoted by $\eps_l(z) \in \Ecal_l$ the index of the bin at level $l$ containing $z$, that is $z \in B_{\eps_l(z)}$. 
Using this representation, we construct the following spike-and-slab prior. Fix $1\le L_0\le L_n $; for all $1\le l\le L_n$ and all $\eps\in \Ecal_l$, let $\bar y_{\eps+}:=1-\bar y_{\eps-}$, and draw $\bar y_{\eps-}$ according to
\begin{equation}
\label{SASprior}
\begin{split}
	\bar y_{\eps-} &\sim q_\eps \delta_{\alpha_n(\eps-)} 
	+(1-q_\eps)\text{Beta}(\alpha_\eps \alpha_n(\eps-),\alpha_\eps\alpha_n(\eps+)),
	\qquad  \varepsilon \in \Ecal_l, \qquad  l \geq  L_0,\\
	\bar y_{\eps-} &\sim \text{Beta}(\alpha_{\eps,1},\alpha_{\eps,2}), 
	\qquad 0<\alpha_{\eps,1},\alpha_{\eps,2}<\infty,
	\qquad \varepsilon \in \Ecal_l, \qquad l < L_0,\\
	\rho^*  & \sim \pi_\rho
	\qquad \text{independently of $(\bar y_\eps, \ \eps\in\Ecal_{[n]})$},
	\end{split}
\end{equation}
for constants $q_\eps,\alpha_\eps>0$ to be chosen depending only on the level of $\eps$, and with $\pi_\rho$ an absolutely continuous probability distribution on $[0,\infty)$ with positive density.

	In other words, the prior is built by putting a prior on the terms  $\bar y_\eps=\rho(\eps)/\rho( \peps)$ and by modelling $\rho $ as piecewise constant over the bins $B_{\eps}$ for $\eps \in \Ecal_{L_n}$. This is similar to construction by \cite{CM21}, which however is based on the definition $\rho(\eps)=\int_{B_\eps} \rho(z)dz$ as opposed to the covariate-dependent one in \eqref{Eq:BasicPT}. This makes the bins unbalanced, which then motivates `centring' the prior of $\bar y_\eps$ in \eqref{SASprior} at $\alpha_n(\eps)$.

\begin{example}[Dyadic partitions]\label{Ex:Diadic}
The simplest instance for the above sequence of partitions $\Pcal^{(L_n)}$ is a dyadic one, where the children are iteratively obtained by a deterministic splitting of the parent bins at the midpoint along each axis, alternatively. For example, for $d=2$, the first three partitions are then
\begin{align*}
	\Pcal_0&=\{[0,1]^2\};\qquad \Pcal_1=\{[0,1/2]\times[0,1],[1/2,1/2]\times[0,1]\};\\
	\Pcal_2&=\{[0,1/2]^2,[0,1/2]\times[1/2,1],[1/2,1]\times[0,1/2],[1/2,1]^2\}.
\end{align*}
In the result to follow, this will not be a strict requirement, and more general partitions (e.g.~quantile-based ones) could be used as well.
\end{example}

%
%
%

\subsubsection{Polya-tree posteriors} \label{Sec:PolyaTPost}

Write shorthand $(\rho^*, \bar y):= (\rho^*, \bar y_\eps,\ \eps \in \Ecal_{[n]})$. Under the parametrisation \eqref{prior:rho}, the likelihood equals
 \begin{equation}
\label{likeli}
	L_n(\rho^*, \bar y)
	=e^{-\rho^* G_n}\prod_{B_\eps\in \mathcal{P}_{L_n}}\rho(\eps)^{N_{\eps}}
	=(\rho^*)^{N_n}e^{-\rho^*n}
	\prod_{\eps\in \Ecal_{[n]}}\bar y_\eps^{N_\eps},
\end{equation}
where, for $N^{(n)}=\{X_1,\dots,X_{N_n}\}$ the point pattern \eqref{Eq:PointProc}, we have denoted by $G_n:=\int_{\Wcal_n}dx=n$ and by $N_\eps:=\sum_{i=1}^{N_n}1_{\{Z(X_i)\in B_\eps\}}$. Note that $L_n(\rho^*, \bar y)$ factorises into two terms depending only on $\rho^*$ and $\bar y$ respectively. Given prior \eqref{SASprior}, this leads to the product posterior distribution 
\begin{align}
\label{post}
	d\Pi(\rho^*, \bar y | D^{(n)}) 
	&= \pi( \rho^*|  D^{(n)}) d\rho^* \prod_{\eps \in \mathcal E_{[n]}} d\Pi(\bar y_\eps | D^{(n)}),
\end{align}
where the posterior density of $\rho^*|  D^{(n)}$ is given by
$$
	\pi( \rho^*|  D^{(n)}) 
	\propto 
	\pi_\rho (\rho^*)(\rho^*)^{N_n}e^{-\rho^* n}
	,
$$
and where the (independent) posterior distributions of $\bar y_\eps | D^{(n)}$ can be explicitly expressed as
\begin{equation}\label{post:polyaT}
\begin{split}
	d\Pi(\bar y_\eps | D^{(n)})&=\omega_1(\eps)\delta_{\alpha_n(\eps)} + \omega_2(\eps)
	\text{Beta}(N_\eps +\alpha_\eps \alpha_n(\eps), N_{\aeps}+ \alpha_\eps(1- \alpha_n(\eps)));\\
	\omega_1(\eps)&:=c_n(\eps)q_\eps \alpha_n(\eps)^{N_\eps} 
	(1-\alpha_n(\eps))^{N_{\aeps}} \delta_{\alpha_n(\eps)};\\
	\omega_2(\eps)&:=\frac{ c_n(\eps)(1 - q_\eps) B( \alpha_\eps \alpha_n(\eps) , \alpha_\eps(1- \alpha_n(\eps)) }{  B(N_\eps +\alpha_\eps \alpha_n(\eps) , N_{\aeps}+ \alpha_\eps(1- \alpha_n(\eps))},
\end{split}
\end{equation}
with $B(a,b) =  \Gamma (a+b)/( \Gamma(a)\Gamma(b))$ (the Beta function) and
$$ 
	c_n(\eps)^{-1} := q_\eps \alpha_n(\eps)^{N_\eps} (1-\alpha_n(\eps))^{N_{\aeps}}
	+(1 - q_\eps) \frac{ B( \alpha_\eps \alpha_n(\eps) , \alpha_\eps(1- \alpha_n(\eps)) }
	{  B(N_\eps +\alpha_\eps \alpha_n(\eps) , N_{\aeps}+ \alpha_\eps(1- \alpha_n(\eps))} .
$$

\begin{remark}[Conjugacy for $\bar y_\eps$]\label{Rem:ConjPT}
The expression for the posterior distribution of $\bar y_\eps|D^{(n)}$ follows from the well-known conjugacy of the Beta prior under the Binomial likelihood, and is akin to the expressions obtained in \cite[Theorem 1]{CM21} for density estimation. This property can be used as the basis for a Metropolis-within-Gibbs sampling algorithm for the product posterior $d\Pi(\rho^*,\bar y|D^{(n)})$, where the updates for $\bar y_\eps|D^{(n)}$ can be carried out straightforwardly via \eqref{post:polyaT}. In our analysis, conjugacy leads to some convenient simplifications, but it is not crucial. The key aspect is in fact the independence of the parameters $(\bar y_\eps, \ \eps\in\Ecal_n)$ under the posterior distribution.
\end{remark}

%
%
%

\subsection{Point-wise contraction rates for P\'olya tree priors} \label{sec:pointwise}

We present our main result in point-wise loss, showing that the P\'olya tree posterior adapts to the (possibly) spatially-varying smoothness of the ground truth, attaining optimal contraction rates towards $\rho(z_0)$ for any $z_0\in[0,1]^d$. The result holds under the following assumptions. We write $\eps^0_l:=\eps_l(z_0)$.

\begin{condition}\label{Cond:Diam}
	Assume that there exist constants $C_d>0$ and $0<c_1<1/2$ such that for all $l\geq L_0$ and all $\eps \in \mathcal E_l$, with probability tending to one (under the law of $Z$),
\begin{equation} 
\label{cond:split}
 	\text{diam}(B_\eps)\le C_d2^{-l/d};
	\qquad 
	c_1 \le  \alpha_n(\eps) \le 1-c_1;
	\qquad 
	\mu_n(B_\eps) \geq C_d^{-1}2^{-l},
\end{equation}
where $\text{diam}(B_\eps):=\max\{|x-y|,\ x,y\in B_\eps\}$ is the diameter of $B_\eps$. 
\end{condition}

\begin{condition}\label{Cond:Holder}
Let $\rho_0:[0,1]^d\to[0,\infty)$ satisfy the following.
\begin{itemize}
\item[(i)] $\rho_0$ is globally $\beta$-H\"older continuous, $\rho_0\in C^\beta([0,1]^d)$, and locally $\beta_0$-H\"{o}lder continuous in a neighbourhood of $z_0\in[0,1]^d$ for some $0<\beta\le\beta_0\le1$, that is
\begin{align*}
	|\rho_0(z_1)-\rho_0(z_2)|&\le 
	C_H|z_1-z_2|^\beta, \qquad \forall z_1,z_2\in[0,1]^d, \qquad C_H>0;\\
	|\rho_0(z)-\rho_0(z_0)| &\le  C_0|z-z_0|^{\beta_0}, \qquad \ \forall z :  |z-z_0| \le \delta_0,
	\qquad C_0,\delta_0>0.
\end{align*}

\item[(ii)] There exists a constant $0<c_0\le C_0$ such that, for all $1\le l \le L_n$,
$$
	c_0\mu_n(B_\epsltrue)\le \rho_0(\epsltrue)\le C_0\mu_n(B_\epsltrue).
$$

\end{itemize}
\end{condition}

	Condition \ref{Cond:Diam} entails that $\Pcal^{(L_n)}$ is sufficiently regular; in particular, the first inequality in \eqref{cond:split} is verified for the dyadic partitions of Example \ref{Ex:Diadic} (and similar ones). The  requirements on $\alpha_n(\eps) $ and $\mu_n(B_\eps)$ depend on the random field $Z$, but they are easy to check in practice since the latter is fully observed. We will revisit these assumptions in Section \ref{Subsec:LocErgCov} below for stationary ergodic covariates.

Condition \ref{Cond:Holder} concerns the ground truth, allowing in (i) its smoothness to vary across the domain. The point-wise contraction rates derived below are then entirely driven by the local regularity $\beta_0$. Condition \ref{Cond:Holder} (ii) is mild, holding in particular if $\rho_0$ is bounded and bounded away from zero near $z_0$.

\begin{theorem}\label{Theo:LocalRates}
For fixed $z_0\in[0,1]^d$ and some $0<\beta\le\beta_0\le1$, assume that $\rho_0$ satisfies Condition \ref{Cond:Holder}. Consider data $D^{(n)}\sim P^{(n)}_{\rho_0}$from the observation model \eqref{Eq:PointProc} with $\rho=\rho_0$ and $Z$ a stationary random field with values in $[0,1]^d$. Consider a P\'olya tree prior $\Pi $ constructed as in \eqref{SASprior}, for a sequence of partitions $\Pcal^{(L_n)}$ satisfying Condition \ref{Cond:Diam} for all $\varepsilon \in (\varepsilon_l^0, \ 1\le l \le L_n)$ with $2^{L_n} \le \delta n/\log n$ for some $\delta >0$ small enough. Further assume that the prior hyperparameters satisfy, for all $L_0 \le l\le L_n$,
\begin{itemize}
	\item[(i)] $0\le (1-q_{\epsltrue})\alpha_{\epsltrue} \le  2^{-lt}$ for some $t>0$, and $q_\epsltrue\geq c_2$ for some $c_2>0$;

	\item[(ii)]  $ \alpha_{\epsltrue}2^l=o(n)$ as $n\to\infty$.
\end{itemize} 
Set $\epsilon_n=(\log n/n)^{\beta_0/(2\beta_0+d)}$. Then, for all sufficiently large $M>0$, in $P^{(n)}_{\rho_0}$-probability as $n\to\infty$,
\[
	\Pi\left(\rho:|\rho_0(z_0)-\rho(z_0)| > M \epsilon_n\big |D^{(n)}\right) \to 0.
\]
\end{theorem}

	The proof is provided in Section \ref{sec:prooflocal} of the Supplement. 
The requirements (i) and (ii) are verified with the choices $\alpha_{\epsltrue} = \alpha$ (fixed) and $q_{\epsltrue} = 2^{-t_0 l}$ for $l\geq L_0$ and any $t_0>0$. We can also choose $\alpha_{\epsltrue} =  \alpha l^{q_0}$ for some $q_0\geq 1$, in which case we need $L_n $ to verify $2^{L_n}=o( n/(\log n)^{q_0})$.

\begin{remark}[Towards uniform pointwise rates]\label{Rk:unifpointwise}
n inspection of the proof of Theorem \ref{Theo:LocalRates} shows that the results holds uniformly in $z_0$ as soon as $t>2d /( 2 \beta+ d)$ in assumption (i). See Remark \ref{Rem:Unif} in the Supplement for details. This also implies posterior contraction in sup-norm at rate  $(n/\log n)^{-\beta/(2\beta+/d)}$, with $\beta\le \beta_0$ the global H\"older smoothness level, as well as the same rate for the loss $|\lambda^{(n)}(x_0)-\lambda_{\rho_0}^{(n)}(x_0)|$, any $x_0\in\Wcal_n$.
\end{remark}

\begin{remark}[Pólya tree and spatially-adaptive priors in the literature]
The prior in Theorem \ref{Theo:LocalRates} differs from \cite{CM21} in two aspects: firstly, the underlying space is multi-dimensional, and secondly, it is based on the non-uniform design $\alpha_n( \cdot)$, which turns out to be key in our proofs. A further difference is that the constraint on the `sparsity parameters' $q_{\epsltrue}$ is significantly milder, allowing for any $t>0$ in condition (i) as opposed to $t  > 1/2 + 1/\log 2$ in \cite{CM21}. Thus, our result applies to less informative priors. The weakened sparsity constraints arise from some sharper bounds in the proof of Lemma \ref{lem:2}, and also from the fact that we consider concentration at $z_0$ and not uniformly over $[0,1]^d$. Finally, Theorem \ref{Theo:LocalRates} also extends results in \cite{RR23} by considering models different from nonlinear regression and, more importantly, by treating the multivariate case. In particular, it remains unclear if the `repulsive' prior construction of \cite{RR23} could be adapted to the multivariate context under Condition \ref{Cond:Diam}. Relatedly, `optional Pólya tree’ as in \cite{wong2010optional,castillo2022optional} could also be studied.
\end{remark}
	
%
%
%

\subsection{Tree-inducing partitions for stationary ergodic covariate processes} 
\label{Subsec:LocErgCov}

We here discuss the validity of Condition \ref{Cond:Diam} in the case where, similarly to Section \ref{Subsec:ErgCov}, $Z$ is assumed to be stationary and ergodic. For simplicity, assume $\Pcal^{(L_n)}$ to be dyadic, cf.~Example \ref{Ex:Diadic}; as previously observed, this satisfies the first inequality in \eqref{cond:split}. Turning to the bounds for $\mu_n(B_\eps)$ and $\alpha_n(\varepsilon)$, note that if the stationary distribution $\nu $ of $Z$ has a continuous density that is bounded and bounded away from zero, then for some $c_\nu>0$,
$$
 	c_\nu^{-1}2^{-l} \le \nu(B_\eps) \le c_\nu2^{-l};
	\qquad \frac{ 1 }{2 c_\nu^2} 
	\le  \frac{ \nu(B_\eps)  }{ \nu(B_{P(\eps)})},
	\qquad \eps\in\Ecal_{[n]}.
$$

	Under ergodicity, $\mu_n(B_\eps)\to\nu(B_\eps)$ almost surely, and we may expect the above to also hold with $\nu $ replaced by $\mu_n$, for large $n$. In particular, if $|\mu_n(B_\eps)  - \nu(B_\eps)| = o_{P_{Z^{(n)}}}(\nu(B_\eps))$, then Condition \ref{Cond:Diam} is verified with $c_1:= 1/(2c_\nu^2)$ and $C_d := \max (2, c_\nu) $, as established in the following result. Recall that Condition \ref{Cond:Diam} needs to be valid only for indices $\eps\in\cup_{l\le L_n} \eps_l(z_0) $ and for the indices of their neighbours. Let  $\bar{\mathcal E}_n(z_0)$ denote the set of all such indices, with cardinality $|\bar{\mathcal E}_n(z_0)|\le 2L_n = O(\log n) $.

\begin{proposition}\label{prop:ass:Z:diam}
Let $Z$ be a stationary random field with values in $[0,1]^d$ and with invariant measure $\nu$. Assume that there exists a constant  $C_Z<\infty$ such that, for all $n\in\N$,
\begin{equation}
\label{mixing}
	\sup_{\eps \in \bar{\mathcal E}_n(z_0)} \int_{\R^D}
	\textnormal{Corr}( 1_{B_\eps}(Z(0)), 1_{ B_\eps}(Z(x)))dx \le C_Z.
\end{equation} 
Then, for any arbitrary sequence $M_n\to\infty$ as $n\to\infty$, we have, for all sufficiently large $n$,
$$ 
	P_{Z^{(n)}}\left(  |\mu_n(B_\eps) - \nu(B_\eps) | 
	> \frac{ M_n \sqrt{ \nu(B_\eps)\log n  }}{\sqrt{ n }},\ 
	\forall \eps \in  \bar{\mathcal E}_n(z_0)\right) \lesssim  \frac{ C_Z }{ M_n^2 } .
$$
\end{proposition}

	Since, if $2^l = o(n/\log n)$, we have $\sqrt{\nu(B_\eps) \log n }/\sqrt{ n } \lesssim 2^{-l/2} \sqrt{ \log n} / \sqrt{n}  = o(2^{-l})$, Proposition \ref{prop:ass:Z:diam} implies that Condition \ref{Cond:Diam} holds for all $l\leq L_n$ provided that $ 2^{L_n} = o(n/\log n)$. This is in accordance with the assumptions of Theorem \ref{Theo:LocalRates}.

\begin{example}[Ergodic Gaussian covariates]
Proposition \ref{prop:ass:Z:diam} (and therefore, also Theorem \ref{Theo:LocalRates}) applies when $Z$ arises according to Condition \ref{Cond:BoundGaussCov} as a (transformed) stationary ergodic Gaussian process. Indeed, in this case, for all $\eps \in \bar{\mathcal E}_n(z_0)$,
\begin{align*}
 	\text{Cov}( 1_{B_\eps}(Z(0)), 1_{ B_\eps}(Z(x))) 
	&\lesssim \int_{\Phi^{-1}(B_\eps)^2} \varphi  ( z_1)
	\left| \varphi_{(1-K(0,x)^2)} (z_2 - K(0,x)z_1) -\varphi (z_2) \right|dz_2 dz_1\\
	&\lesssim \frac{  |K(0,x)| }{1-K(0,x)^2 }
	\int_{\Phi^{-1}(B_\eps)^2}
	 \varphi( z_1)\varphi(z_2) dz_2 dz_1
	 \lesssim \nu(B_\eps)^2\frac{  |K(0,x)| }{1-K(0,x)^2 },
\end{align*}
where $\varphi$ is the standard normal probability density function. This shows that \eqref{mixing} is verified with constant $C_Z:=\sup_{x\in\R^D}|K(0,x)|(1-K(0,x)^2)^{-1}<\infty$.
\end{example}

%
%
%

\subsection{$L^1$-posterior contraction rates for Pólya tree priors} \label{PT:L1loss}

We conclude with a global analysis of Pólya tree priors, in the spirit of the investigations from first part of the paper. We first obtain  posterior contraction rates in the empirical $L^1$-loss from Theorem \ref{Theo:GenRandCov}, and then combine them with specific prior properties to show concentration in a standard $L^1$-distance.

\begin{theorem} \label{thm:L1:polya}
Assume that $\rho_0\in C^\beta([0,1]^d)$ for some $0<\beta\le1$. Consider data $D^{(n)}\sim P^{(n)}_{\rho_0}$from the observation model \eqref{Eq:PointProc} with $\rho=\rho_0$ and $Z$ a stationary random field with values in $[0,1]^d$ and stationary distribution $\nu$ with bounded density. Consider a P\'olya tree prior $\Pi $ constructed as in \eqref{SASprior}, for a sequence of partitions $\Pcal^{(L_n)}$ satisfying Condition \ref{Cond:Diam} 
for all $\varepsilon \in \Ecal_{[n]}$ with $2^{L_n} \le \delta n/\log n$ for some $\delta >0$ small enough. Further assume that the prior hyperparameters satisfy, for all $L_0 \le l\le L_n$ and all $\eps \in \mathcal E_l  $, 
 $ 2^{-lt_1}\le (1-q_{\varepsilon})\le  2^{-lt_2}$ for some $t_1\geq t_2\geq 2\beta /(2\beta+d) $, and $ \alpha_{\varepsilon} \leq C_2$ for some $C_2>0$. 
 Set $\epsilon_n=(\log n/n)^{\beta/(2\beta+d)}$. Then, for all sufficiently large $M>0$, in $P^{(n)}_{\rho_0}$-probability as $n\to\infty$,
\[
	\Pi\left(\rho:\|\lambda^{(n)}_\rho-\lambda^{(n)}_{\rho_0}\|_{L^1(\Wcal_n)}
	> M \epsilon_n\big |D^{(n)}\right) \to 0 ;\qquad
	\Pi\left(\rho:\|\rho-\rho_0\|_{L^1([0,1]^d,\nu)} > 
	2M \epsilon_n\big |D^{(n)}\right) \to 0.
\]
\end{theorem} 

Note that under the more stringent assumption on the prior $t_2 \geq 2d/(2\beta+d)$, the above $L^1$-rates could be deduced from a sup-norm result as described in Remark \ref{Rk:unifpointwise}. They are proved here directly using the global techniques from Section \ref{Sec:GlobLoss}, which allows to consider less informative priors.

%
%
%
%
%

\section{Summary and outlook}\label{sec:Summary}

In this article, we have developed the first asymptotic analysis of nonparametric Bayesian procedures for inference on the intensity function of a covariate-driven point process. We have considered various families of priors distributions, worked under global and local loss functions, and included several scenarios for the covariates, covering a range of potential real-world settings. Our main results in global loss (Theorems \ref{Theo:GaussWav} and \ref{Cor:GaussPoissTess}) show that in the presence of stationary ergodic covariates suitable Gaussian priors achieve optimal $L^1$-contraction rates. This is underpinned by a general theory for an empirical $L^1$-loss (Theorem \ref{Theo:GenRandCov}), which we have also applied to location mixture of Gaussians priors (Proposition \ref{Theo:MixRandLoss}). Finally, we have devised novel Pólya tree-type priors, with which we have achieved (adaptive) rates both in point-wise and $L^1$-loss (Theorems \ref{Theo:LocalRates} and \ref{thm:L1:polya}, respectively).

	Overall, our result hint at a general capability of Gaussian priors to achieve, in the context of covariate-based intensity estimation, optimal performances under global smoothness assumptions and global losses, in line with the broader statistical literature, \cite{vdVvZ08}. In view of Section \ref{subsec:mixtprior}, mixture priors offer a valuable alternative to hierarchical Gaussian priors for adaptation, and they are natural models in the case of unbounded covariates. However, it is unclear at present whether the asymptotic analysis for mixture priors could be extended to standard (non-random) $L^1$-distances under ergodicity. Remarkably, the Pólya-tree priors from Section \ref{Subsec:LocErgCov} achieve optimal rates in both local and global losses; thus they may be preferable when estimation (and prediction) at specific covariate values and locations is of primary interest. However,  due to their piecewise constant structure, these priors are limited to the smoothness range $(0,1)$. Constructing priors with similar spatially-adaptive properties for higher regularities is a formidable challenge for future research.

	Several possible extensions of our results have been discussed in remarks throughout the text. We conclude with an overview on some further overarching research questions.

%
%
%

\subsection{Infill asymptotics}\label{Sec:Infill}

In our analysis, we have worked in the growing domain asymptotics. This setting is natural for many spatial statistical applications where all the `events' of interest have already `occurred', and is widely adopted in the literature, \cite{J81,H85,J93,ZZ05,G08}. The other relevant regime is known as `infill asymptotics' and prescribes observation of an increasing number of points over a fixed bounded window, or equivalently of repeated independent realisations of the pattern (e.g.~\cite{S99}). Our proof strategy for $L^1$-rates (cf.~Section \ref{Subsec:ErgCov}) could be extended to the latter setting, by treating the individual realisations as a `sub-sampling' of a single process over a larger area. The employed concentration inequalities for spatial averages (cf.~Section \ref{Sec:ConcIneq} in the Supplement) should then be replaced with suitable empirical process techniques. Alternatively, extensions of the Hellinger testing theory for point processes developed in \cite{BSvZ15,KvZ15} could be pursued.

%
%
%

\subsection{Statistical guarantees for uncertainty quantification}

Another important question is the frequentist validity of the uncertainty quantification delivered by the considered procedures, as it is known that, in infinite-dimensional models, even consistent posterior distributions may produce credible sets with asymptotically vanishing frequentist coverage, e.g.~\cite{DF86,C93,L11}. The posterior contraction rates derived here may serve as a starting `localising' step towards tackling these issues via nonparametric and semi-parametric Bernstein-von Mises theorems, \cite{CN13,CN14,CR15}, or via the approach with inflated credible balls as in \cite{botond2,rouss:szabo20}. This represents an interesting direction for the literature on nonparametric Bayesian intensity estimation, as we are not aware of any such result even in for models without covariates.

%
%
%

\subsection{Implementation}\label{Sec:Impl}

Lastly, implementation of nonparametric Bayesian inference for covariate-driven point processes remains largely unexplored, presenting an interesting methodological open problem. Given full observations $Z^{(n)}$ of the covariates over the window, the likelihood $L_n(\rho)$ in \eqref{Eq:Likelihood} formally corresponds to that of an inhomogeneous Poisson process with intensity $\lambda^{(n)}_\rho = \rho\circ Z^{(n)}$. This suggests that a number of existing algorithms for models without covariates could feasibly be extended to the present setting.

	For Gaussian priors similar to the ones considered here, \cite{AMM09} constructed an `exact' Markov chain Monte Carlo (MCMC) sampler, based on data augmentation to avoid numerical computation of the integral in \eqref{Eq:Likelihood} and other approximations. For the mixture of Gaussian priors studied in Section \ref{subsec:mixtprior}, the vast methodological literature on density estimation offers a promising foundation; see e.g.~\cite[Chapter 5]{GvdV17}. Finally, for the Pólya-tree priors investigated in Section \ref{sec:localrates}, the simple conditionally independent structure \eqref{post} and the partial conjugacy property \eqref{post:polyaT} can serve as the basis of a Metropolis-within-Gibbs sampling algorithm, cf.~Remark \ref{Rem:ConjPT}.

%
%
%
%
%

\appendix

\addcontentsline{toc}{section}{Acknowledgments}
\section*{Acknowledgments}
We are grateful to the Associate Editor and three anonymous Referees for many insightful comments that led to a substantial improvement of the paper.

\addcontentsline{toc}{section}{Funding}
\section*{Funding}

This research project has been partially funded by  the European Research Council (ERC) under the European Union's Horizon 2020 research and innovation programme (grant agreement No. 834175). In addition, M.G.~has also been partially supported by MUR, PRIN project 2022CLTYP4.

\bibliographystyle{acm}

\bibliography{PointProcRef}

\appendix
\addcontentsline{toc}{section}{Supplementary Material}
\section*{Supplementary Material}

In this supplement, we provide all the proofs of the results from the main article. We also state and prove the necessary auxiliary results, and  include some additional background material.

\section{Proof of Theorem \ref{Theo:GenRandCov}}
\label{Sec:ProofGenRandCov}


Let $U_n := \{\rho \in\Rcal_n: \|\lambda^{(n)}_\rho - \lambda^{(n)}_{\rho_0}\|_{L^1(\Wcal_n)} \le Mn\epsilon_n\}$ be the event whose posterior probability is of interest. By Bayes' formula, with $l_n(\rho)$ the log-likelihood given in \eqref{Eq:LogLik} below,
$$
	\Pi(U_n^c | D^{(n)}) 
	=
	\frac{N_n}{D_n} := \frac{\int_{U_n^c} e^{l_n(\rho) - l_n(\rho_0)}d\Pi(\rho)}
	{\int_{\Rcal} e^{l_n(\rho) - l_n(\rho_0)}d\Pi(\rho)}.
$$
Using Lemma 8.21 of \cite{GvdV17}, jointly with Lemma \ref{Lem:KLControl} below, we have as $n\to\infty$
$$
	P_{\rho_0}^{(n)}
	\left( D_n\le e^{- K_1n\epsilon_n^2}\Pi( B_{n,2}(\rho_0) ) \right) = O(1/(n \epsilon_n^2) ), 
	\qquad K_1:= 1 + 2\|\rho_0\|_{L^\infty(\Zcal)},
$$
and similarly, for any $M_n\to\infty$,
$$
	P_{\rho_0}^{(n)}\left( D_n\le e^{- M_nn\epsilon_n^2}\Pi( B_{n,0}(\rho_0) ) \right) = O(1/M_n ).
	\qquad n\to\infty,
$$
where $B_{n,0}(\rho_0),B_{n,2}(\rho_0)\subset\Rcal$ are defined as before the statement of Theorem \ref{Theo:GenRandCov}. Hence, by assumption \eqref{Eq:SmallBall}, as $n\to\infty$,
\begin{equation*}
\begin{split}
	E_{\rho_0}^{(n)}[\Pi(U_n^c | D^{(n)}) ]
	&\le
	E_{\rho_0}^{(n)}\left[\Pi(U_n^c | D^{(n)}) 
	1_{\{ D_n > e^{-(C_1+K_1)n\epsilon_n^2}\}}\right]
	+O(1/(n\epsilon_n^2)).
\end{split}
\end{equation*}
Note that
\begin{equation*}
\begin{split}
	E_{\rho_0}^{(n)}&\left[\Pi(U_n^c | D^{(n)}) 
	1_{\{ D_n > e^{-(C_1+K_1)n\epsilon_n^2} \}}\right]\\
	&\le 
	E_{\rho_0}^{(n)}\left[\Pi(\Rcal_n^c | D^{(n)}) 
	1_{\{ D_n > e^{-(C_1+K_1)n\epsilon_n^2} \}}\right]\\
	&\quad +
	E_{\rho_0}^{(n)}
	\left[\Pi(\rho\in\Rcal_n :
	\|\lambda^{(n)}_\rho - \lambda^{(n)}_{\rho_0}\|_{L^1(\Wcal_n)} > Mn\epsilon_n  | D^{(n)}) 
	1_{\{ D_n > e^{-(C_1+K_1)n\epsilon_n^2} \}}\right],
\end{split}
\end{equation*}
with the first expectation being upper bounded by
\begin{equation*}
\begin{split}
	e^{(C_1+K_1)n\epsilon_n^2}
	\int_{\Rcal_n^c}
	E_{\rho_0}^{(n)}\left[ e^{l_n(\rho) - l_n(\rho_0)}\right] d\Pi(\rho)
	&= e^{(C_1+K_1)n\epsilon_n^2}
	\Pi(\Rcal_n^c)
	\le  e^{-n\epsilon_n^2} \le 1/(n\epsilon_n^2),
\end{split}
\end{equation*}
having used Fubini's theorem, the fact that  $E_{\rho_0}^{(n)}[ e^{l_n(\rho) - l_n(\rho_0)}] = E^{(n)}_\rho[1]=1$ and assumption \eqref{Eq:Sieves}. Next, recalling assumption \eqref{Eq:MetricEntropy}, 
by Lemma \ref{Lem:GlobalTests}, upon taking $M>\max\{((C_1+K_1+1)/K_{\rho_0})^{1/2},$ $((C_3+1)/K_{\rho_0})^{1/2},1\}$ with $K_{\rho_0}>0$ the constant in the statement of Lemma \ref{Lem:Tests}, for all $n$ large enough there exists a test $\phi_n$ such that
$$
	E_{\rho_0}^{(n)}[\phi_n|Z^{(n)}]\le  2e^{-(K_{\rho_0}M^2 - C_3) n\epsilon_n^2},
$$
and 
$$
	\sup_{\rho\in\Rcal_n : \|\lambda^{(n)}_\rho - \lambda^{(n)}_{\rho_0}\|_{L^1(\Wcal_n)}\ge M n\epsilon_n}
	E^{(n)}_\rho[1-\phi_n |Z^{(n)}]
	\le 2e^{-K_{\rho_0}M^2 n\epsilon_n^2}.
$$
It follows that
\begin{equation*}
\begin{split}
	&E_{\rho_0}^{(n)}
	\left[\Pi(\rho\in\Rcal_n :
	\|\lambda^{(n)}_\rho - \lambda^{(n)}_{\rho_0}\|_{L^1(\Wcal_n)} > Mn\epsilon_n  | D^{(n)}) 
	1_{\{ D_n > e^{-(C_1+K_1)n\epsilon_n^2} \}}\right]
	\\
	& \ \le
	E_{\rho_0}^{(n)}[E_{\rho_0}^{(n)}[\phi_n|Z^{(n)}]]\\
	&\ +
	E_{\rho_0}^{(n)}\left[\Pi(\rho\in\Rcal_n :
	\|\lambda^{(n)}_\rho - \lambda^{(n)}_{\rho_0}\|_{L^1(\Wcal_n)} > Mn\epsilon_n  | D^{(n)}) 
	1_{\{ D_n > e^{-(C_1+K_1)n\epsilon_n^2} \}}(1-\phi_n)\right]\\
	&\le (n\epsilon_n^2)^{-1}+e^{(C_1+K_1)n\epsilon_n^2}
	E_{\rho_0}^{(n)}
	\left[\int_{\{\rho\in\Rcal_n :
	\|\lambda^{(n)}_\rho - \lambda^{(n)}_{\rho_0}\|_{L^1(\Wcal_n)} > Mn\epsilon_n\}}
	e^{l_n(\rho) - l_n(\rho_0)}(1-\phi_n)d\Pi(\rho)
	\right].
\end{split}
\end{equation*}
Using (stochastic) Fubini's theorem and the tower property, the latter expectation equals
\begin{align*}
	E_{\rho_0}^{(n)}\left[\int_\Rcal 1_{\{\rho\in\Rcal_n :
	\|\lambda^{(n)}_\rho - \lambda^{(n)}_{\rho_0}\|_{L^1(\Wcal_n)} > Mn\epsilon_n\}}E_{\rho_0}^{(n)}\left[ 
	e^{l_n(\rho) - l_n(\rho_0)}(1-\phi_n)\Big|Z^{(n)}\right]d\Pi(\rho)\right]
\end{align*}
and since, for all $\rho\in\Rcal_n$ with $\|\lambda^{(n)}_\rho - \lambda^{(n)}_{\rho_0}\|_{L^1(\Wcal_n)} > Mn\epsilon_n$ and all $n$ large enough,
$$
	E_{\rho_0}^{(n)}\left[e^{l_n(\rho) - l_n(\rho_0)}(1-\phi_n)\Big|Z^{(n)}\right]
	= E_\rho^{(n)}[1-\phi_n|Z^{(n)}]
	\le 2e^{-K_{\rho_0} M^2 n\epsilon_n^2},
$$
we have
\begin{align*}
	E_{\rho_0}^{(n)}
	\Bigg[\Pi(\rho\in\Rcal_n :
	\|\lambda^{(n)}_\rho - \lambda^{(n)}_{\rho_0}\|_{L^1(\Wcal_n)} &> Mn\epsilon_n  | D^{(n)}) 
	1_{\{ D_n > e^{-(C_1+K_1)n\epsilon_n^2} \}}\Bigg]
	\\
	&\le
	(n\epsilon_n^2)^{-1} + 2e^{-(K_{\rho_0}M^2 - (C_1+K_1))n\epsilon_n^2}
	\le 2/(n\epsilon_n^2),
\end{align*}
concluding the proof.\qed

%
%
%

\subsection{Bounds on the KL-divergence and variation}

Recalling the likelihood in \eqref{Eq:Likelihood}, the log-likelihood function associated to data $D^{(n)}$ from model \eqref{Eq:PointProc} is given by
\begin{equation}
\label{Eq:LogLik}
	l_n(\rho):=\log \left[ \frac{dP^{(n)}_\rho}{dP^{(n)}_1}(D^{(n)}) \right]
	= \int_{\Wcal_n}\log(\lambda^{(n)}_\rho(x))d N^{(n)}(x)-\int_{\Wcal_n}\lambda^{(n)}_\rho(x)dx,
	\qquad \rho\in\Rcal,
\end{equation}
where $\Rcal\subset L^\infty(\Zcal)$ is a measurable collection of non-negative functions defined on $\Zcal$. The first auxiliary result for the Proof of Theorem \ref{Theo:GenRandCov} controls the Kullback-Leibler divergence and variation between intensities, defined respectively as
\begin{align*}
	\textnormal{KL}_n(\rho_1,\rho_2):=E_{\rho_1}^{(n)}[l_n(\rho_1)-l_n(\rho_2)];
	\quad
	\textnormal{V}_{2,n}(\rho_1,\rho_2):=E_{\rho_1}^{(n)}[(l_n(\rho_1)-l_n(\rho_2)
	-\textnormal{KL}_n(\rho_1,\rho_2))^2].
\end{align*}

\begin{lemma}\label{Lem:KLControl}
Let $\rho_0\in L^\infty(\Zcal)$ be non-negative valued. Let $B_{n,0}(\rho_0), B_{n,2}(\rho_0)\subset\Rcal$ be defined as before the statement of Theorem \ref{Theo:GenRandCov} for some positive sequence $(\epsilon_n)_{n\ge1}$. Then, 
\begin{equation}
\label{Eq:KLControl2}
	\sup_{\rho \in  B_{n,0}(\rho_0)}\textnormal{KL}_n(\rho_0 ,\rho) \le 2\|\rho_0\|_{L^\infty(\Zcal)} n\epsilon_n^2; 
	\qquad  \sup_{\rho \in  B_{n,2}(\rho_0)}	\textnormal{V}_{2,n}(\rho_0 ,\rho) \le 4\|\rho_0\|_{L^\infty(\Zcal)} n\epsilon_n^2
\end{equation}

\end{lemma}

\begin{proof}
Write shorthand $\lambda_\rho := \lambda^{(n)}_\rho$, and recall that, under $P^{(n)}_{\rho_0 }$, $N^{(n)}|Z^{(n)}$ is a Poisson process on $\Wcal_n$ with intensity $\lambda_{\rho_0}$, and hence for all integrable 
$f:\Wcal_n\to\R$, 
$$
	E^{(n)}_{\rho_0 }\left[\int_{\Wcal_n}f(x)dN^{(n)}(x)
	\Big| Z^{(n)}\right] = \int_{\Wcal_n}f(x) \lambda_{\rho_0}(x)dx
$$ 
(e..g, Proposition 2.7 in \cite{LP18}). Using this, we have
\begin{align*}
	E^{(n)}_{\rho_0 }\left[l_n(\rho_0 ) - l_n(\rho)
	\Big| Z^{(n)}\right]
	&=
	\int_{\Wcal_n}\lambda_{\rho_0}(x)\log\left( \frac{\lambda_{\rho_0}(x)}{\lambda_\rho(x)}
	\right)dx
	-\int_{\Wcal_n}(\lambda_{\rho_0}(x)-\lambda_\rho(x))dx\\
	&=
	\int_{\Wcal_n}\lambda_{\rho_0}(x) h\left( \frac{\lambda_\rho(x)}{\lambda_{\rho_0}(x)} \right)
	dx
\end{align*}
where $h(u) := u - 1 - \log u, \ u>0$. The function $h$ satisfies $h(u)\le 2(\sqrt{u} -1)^2$ for all $u\in[1,\infty)$ and $h(u)\le \log^2 u$ for all $u\in(0,1)$. Thus, recalling that $Z(x)\sim \nu $ for all $x\in\Wcal_n$, that $|\Wcal_n|=n$ and the notations $\bar\rho,\bar\rho_0, M_\rho$ and $M_{\rho_0}$ introduced before Theorem \ref{Theo:GenRandCov}, we obtain
\begin{align*}
	\textnormal{KL}_n(\rho_0, \rho)  
	&= |\Wcal_n|M_{\rho_0} \int_{\Zcal} \bar\rho_0(z) \frac{ \bar\rho_0(z) }{\bar \rho(z) } d\nu(z) 
	+|\Wcal_n| M_{\rho_0} h\left(\frac{M_\rho}{M_{\rho_0}}\right)\\
	&=M_{\rho_0}n\textnormal{KL}_\nu(\bar\rho_0, \bar\rho)  
	+M_{\rho_0} nh\left(\frac{M_\rho}{M_{\rho_0}}\right),
\end{align*}
and similarly 
$$
	\textnormal{V}_{2,n}(\rho_0, \rho) 
	\le  2 M_{\rho_0}n\left[ \int_\Zcal \bar \rho_0 (z)\log^2\left( \frac{ \bar\rho_0(z) }{\bar \rho(z) } \right) d\nu(z) 
	+ ( M_\rho -  M_{\rho_0})^2 \right].
$$
Therefore, for all $\rho\in B_{n,0}(\rho_0)$,
$$
	\textnormal{KL}_n(\rho_0, \rho) \le  2nM_{\rho_0} \epsilon_n^2 \le 2 \|\rho_0\|_{L^\infty(\Zcal)} n\epsilon_n^2
$$
while for all $\rho\in B_{n,2}(\rho_0)$
$$ 
	\textnormal{V}_{2,n}(\rho_0, \rho)  \le  4\|\rho_0\|_{L^\infty(\Zcal)} n\epsilon_n^2.
$$
\end{proof}

%
%
%

\subsection{Tests for alternatives separated in empirical $L^1$-distance}

The following lemma provides a construction of tests for simple alternatives separated with respect to the covariate-dependent loss function appearing in Theorem \ref{Theo:GenRandCov}

\begin{lemma}\label{Lem:Tests}
For all non-negative valued $\rho_1\in L^\infty(\Zcal)$, there exists a test $\phi_{\rho_1}$ satisfying, for all $n\in\N$, 

$$
	E_{\rho_0}^{(n)}[\phi_{\rho_1}|Z^{(n)}]
	\le2e^{-K_{\rho_0}\|\lambda^{(n)}_{\rho_1}-\lambda^{(n)}_{\rho_0}\|_{L^1(\Wcal_n)}
	\min\left\{1,\frac{1}{n}\|\lambda^{(n)}_{\rho_1}-\lambda^{(n)}_{\rho_0}\|_{L^1(\Wcal_n)}
	\right\}},
$$
and
\begin{align*}
	&\sup_{\rho:\|\lambda^{(n)}_\rho-\lambda^{(n)}_{\rho_1}\|_{L^1(\Wcal_n)}
	\le \frac{1}{2}\|\lambda^{(n)}_{\rho_1}-\lambda^{(n)}_{\rho_0}\|_{L^1(\Wcal_n)}}
	E_\rho^{(n)}[1-\phi_{\rho_1}|Z^{(n)}]\\
	&\qquad\qquad\qquad\ 
	\le 2e^{-K_{\rho_0}\|\lambda^{(n)}_{\rho_1}-\lambda^{(n)}_{\rho_0}\|_{L^1(\Wcal_n)}
	\min\left\{1,\frac{1}{n}\|\lambda^{(n)}_{\rho_1}-\lambda^{(n)}_{\rho_0}\|_{L^1(\Wcal_n)}
	\right\}},
\end{align*}
where $K_{\rho_0}:=\min\{1/6,1/(4\|\rho_0\|_{L^\infty(\Zcal)})\}/32$.
\end{lemma}

\begin{proof}
We start with some preliminary observations. For $Y$ a Poisson random variable with parameter $\gamma>0$, the (exponential) Markov inequality yields, for any $a,y>0$, 
\begin{align*}
	\Pr(Y-\gamma \ge y) \le e^{-ay}E[ e^{aY-a\gamma}]
	=e^{-ay-a\gamma-\gamma+\gamma e^a}.
\end{align*}
The right hand side is minimised in $a$ by taking $a=\log(y+\gamma)-\log \gamma$. It follows that $\Pr(Y-\gamma \ge y) \le e^{-\gamma g(y/\gamma)}$, where $g(u):=(1+u)\log(1+u)-u$. As $g(u) \ge u^2/(2+2u/3)$ for all $u>0$, 
\begin{equation}
\label{Eq:PO}
	\Pr(Y-\gamma \ge y) \le \exp\left\{-\frac{y^2}{2\gamma+2y/3}\right\}
	\le \exp\left\{-\frac{y^2}{2(y+\gamma)}\right\}.
\end{equation}
For any measurable set $A\subseteq \Wcal_n$, let $N^{(n)}(A)$ the number of points belonging to $A$, satisfying, under $P^{(n)}_\rho$, $N^{(n)}(A)|Z^{(n)}\sim \textnormal{Po}(\Lambda^{(n)}_\rho(A))$, with $\Lambda^{(n)}_\rho(A) := \int_{A}\lambda_\rho^{(n)}(x)dx $. By \eqref{Eq:PO}, we obtain that for any positive sequence $(\eta_n)_{n\ge1}$,
\begin{equation}
\label{Eq:POSpec}
	P^{(n)}_{\rho_0 }\left(N^{(n)}(A)-\Lambda^{(n)}_{\rho_0}(A) \ge \eta_n
	\Big| Z^{(n)}\right)
	\le \exp\left\{-\frac{\eta_n^2}{2(\eta_n+\Lambda^{(n)}_{\rho_0}(A))}
	\right\}.
\end{equation}
Similarly, it holds that
\begin{equation*}
	P^{(n)}_{\rho_0 }\left(N^{(n)}(A)-\Lambda^{(n)}_{\rho_0}(A)\le -\eta_n
	\Big| Z^{(n)}\right)
	\le \exp\left\{-\frac{\eta_n^2}{2(\eta_n+\Lambda^{(n)}_{\rho_0}(A))}\right\}.
\end{equation*}

	We proceed constructing the tests. Write shorthand $\lambda_{\rho_1} = \lambda^{(n)}_{\rho_1}$ and $\Lambda_{\rho_1}(A) := \Lambda^{(n)}_{\rho_1}(A)$, and define the set  $A:=\{x\in \Wcal_n: \lambda_{\rho_1}(x) \ge \lambda_{\rho_0}(x)\}$. Then, $A^c=\{x\in \Wcal_n: \lambda_{\rho_0}(x)> \lambda_{\rho_1}(x)\}$ and
$$
	\|\lambda_{\rho_1} - \lambda_{\rho_0}\|_{L^1(\Wcal_n)}
	= \Lambda_{\rho_1}(A)-\Lambda_{\rho_0}(A)
	+\Lambda_{\rho_0}(A^c)-\Lambda_{\rho_1}(A^c).
$$
We first handle the case where $\Lambda_{\rho_1}(A)-\Lambda_{\rho_0}(A) \ge \Lambda_{\rho_0}(A^c)-\Lambda_{\rho_1}(A^c)$. Take the indicator $\phi_{\rho_1,A}:=$ $1_{\{N^{(n)}(A)-\Lambda_{\rho_0}(A) \ge \eta_n\}}$ with the choice $\eta_n:=\frac{1}{4}\|\lambda_{\rho_1} - \lambda_{\rho_0}\|_{L^1(\Wcal_n)}$. Then, by \eqref{Eq:POSpec},
\begin{equation*}
\begin{split}
	E^{(n)}_{\rho_0}[\phi_{\rho_1,A}|Z^{(n)}]
	&\le \exp\left\{-\frac{\eta_n^2}{4\max\{\eta_n, 
	\Lambda_{\rho_0}(A)\}}\right\}\\
	&\le 
	\exp\left\{-\frac{1}{16}\min\Big\{\|\lambda_{\rho_1} - \lambda_{\rho_0}\|_{L^1(\Wcal_n)},
	\frac{\|\lambda_{\rho_1} - \lambda_{\rho_0}\|_{L^1(\Wcal_n)}^2}{4\|\rho_0\|_{L^\infty(\Zcal)} n}\Big\}\right\}.
\end{split}
\end{equation*}

	Now consider any non-negative valued alternative $\rho\in L^\infty(\Zcal)$ such that $\|\lambda_\rho -\lambda_{\rho_1}\|_{L^1(\Wcal_n)}\le\frac{1}{2}\|\lambda_{\rho_1} - \lambda_{\rho_0}\|_{L^1(\Wcal_n)}$. It follows
\begin{align*}
	|\Lambda_\rho(A)-\Lambda_{\rho_1}(A)|
	&\le\|\lambda_\rho -\lambda_{\rho_1}\|_{L^1(\Wcal_n)}
	\le \frac{1}{2}\|\lambda_{\rho_1} - \lambda_{\rho_0}\|_{L^1(\Wcal_n)}
	\le \Lambda_{\rho_1}(A)-\Lambda_{\rho_0}(A).
\end{align*}
Therefore,
\begin{align*}
	E_\rho^{(n)}&[1-\phi_{\rho_1,A}|Z^{(n)}]\\
	&=P_\rho^{(n)}
	\left(N^{(n)}(A)-\Lambda_\rho(A)
	 <\eta_n-\Lambda_\rho(A)+\Lambda_{\rho_1}(A)
	+\Lambda_{\rho_0}(A)-\Lambda_{\rho_1}(A)\Big| Z^{(n)}\right)\\
	&\le P_\rho^{(n)}\left(N^{(n)}(A)-\Lambda_\rho(A)
	<\eta_n-(\Lambda_{\rho_1}(A)-\Lambda_{\rho_0}(A))\Big| Z^{(n)}\right).
\end{align*}
Recalling $\eta_n=\frac{1}{4}\|\lambda_{\rho_1} - \lambda_{\rho_0}\|_{L^1(\Wcal_n)}$, we have
$$
	\Lambda_{\rho_1}(A)-\Lambda_{\rho_0}(A)-\eta_n
	 \ge\frac{1}{4}\|\lambda_{\rho_1} - \lambda_{\rho_0}\|_{L^1(\Wcal_n)},
$$
so that in view of the display after \eqref{Eq:POSpec},
\begin{align*}
	E_\rho^{(n)}[1-\phi_{\rho_1,A}|Z^{(n)}]
	&\le P_\rho^{(n)}\left(N^{(n)}(A)-\Lambda_\rho(A)
	<-\frac{1}{4}\|\lambda_{\rho_1} - \lambda_{\rho_0}\|_{L^1(\Wcal_n)}
	\Big| Z^{(n)}\right)\\\
	&\le 
	\exp\left\{
	-\frac{1}{16}\min\Big\{
	\|\lambda_{\rho_1} - \lambda_{\rho_0}\|_{L^1(\Wcal_n)},
	\frac{\|\lambda_{\rho_1} - \lambda_{\rho_0}\|_{L^1(\Wcal_n)}^2}{4\Lambda_\rho(A)}\Big\}\right\}.
\end{align*}
Note that $\Lambda_\rho(A)\le\|\lambda_\rho \|_{L^1(\Wcal_n)}\le\|\lambda_\rho -\lambda_{\rho_1}\|_{L^1(\Wcal_n)}+\|\lambda_{\rho_1}\|_{L^1(\Wcal_n)}$, which is further bounded by
\begin{align*}
	\frac{1}{2}&\|\lambda_{\rho_1}-\lambda_{\rho_0}\|_{L^1(\Wcal_n)}
	+\|\lambda_{\rho_1} - \lambda_{\rho_0}\|_{L^1(\Wcal_n)}
	+\|\lambda_{\rho_0}\|_{L^1(\Wcal_n)}\\
	&=\frac{3}{2}\|\lambda_{\rho_1} - \lambda_{\rho_0}\|_{L^1(\Wcal_n)}
	+\Lambda_{\rho_0 }(\Wcal_n) 
	\le 2\max\Big\{\frac{3}{2}
	\|\lambda_{\rho_1} - \lambda_{\rho_0}\|_{L^1(\Wcal_n)},\|\rho_0\|_{L^\infty(\Zcal)} n\Big\}.
\end{align*}
Combined with the previous display, this implies
\begin{equation*}
\begin{split}
	&E_\rho^{(n)}[1-\phi_{\rho_1,A}|Z^{(n)}] \\
	&\ 
	\le \exp\left\{-\frac{1}{16}\min\Big\{\|\lambda_{\rho_1} - \lambda_{\rho_0}\|_{L^1(\Wcal_n)},
	\frac{1}{2}\min\Big\{\frac{1}{6}\|\lambda_{\rho_1} - \lambda_{\rho_0}\|_{L^1(\Wcal_n)},
	\frac{\|\lambda_{\rho_1} - \lambda_{\rho_0}\|_{L^1(\Wcal_n)}^2}{4\|\rho_0\|_{L^\infty(\Zcal)} n} \Big\}\Big\}\right\}\\
	&\ \le \exp\left\{-\frac{1}{32}\min\Big\{\frac{1}{6}
	\|\lambda_{\rho_1} - \lambda_{\rho_0}\|_{L^1(\Wcal_n)},
	\frac{\|\lambda_{\rho_1} - \lambda_{\rho_0}\|_{L^1(\Wcal_n)}^2}{4\|\rho_0\|_{L^\infty(\Zcal)}n}\Big\}\right\}.
\end{split}
\end{equation*}

	For the case $\Lambda_{\rho_0}(A^c)-\Lambda_{\rho_1}(A^c) \ge \Lambda_{\rho_1}(A)-\Lambda_{\rho_0}(A)$, define, again with $\eta_n:=\frac{1}{4}\|\lambda_{\rho_1} - \lambda_{\rho_0}\|_{L^1(\Wcal_n)}$, the test $\phi_{\rho_1,A^c}:= 1_{\{N^{(n)}(A)-\Lambda_{\rho_0}(A)\le-\eta_n\}}$. Arguing as above, we then obtain
\begin{align*}
	E^{(n)}_{\rho_0 }&[\phi_{\rho_1,A^c}|Z^{(n)}]
	\le  \exp\left\{-\frac{1}{16}\min\Big
	\{\|\lambda_{\rho_1} - \lambda_{\rho_0}\|_{L^1(\Wcal_n)},
	\frac{\|\lambda_{\rho_1} - \lambda_{\rho_0}\|_{L^1(\Wcal_n)}^2}{4\|\rho_0\|_{L^\infty(\Zcal)} n}
	\Big\}\right\},
\end{align*}
and, for any non-negative $\rho\in L^\infty(\Zcal)$ with  $\|\lambda_\rho -\lambda_{\rho_1}\|_{L^1(\Wcal_n)}\le\frac{1}{2}\|\lambda_{\rho_1} - \lambda_{\rho_0}\|_{L^1(\Wcal_n)}$,
\begin{align*}	
	&E_\rho^{(n)}[1-\phi_{\rho_1,A^c}|Z^{(n)}]
	\le \exp\left\{-\frac{1}{32}\min\Big\{\frac{1}{6}
	\|\lambda_{\rho_1} - \lambda_{\rho_0}\|_{L^1(\Wcal_n)},
	\frac{\|\lambda_{\rho_1} - \lambda_{\rho_0}\|_{L^1(\Wcal_n)}^2}{4\|\rho_0\|_{L^\infty(\Zcal)}n}\Big\}\right\}.
\end{align*}
The proof is then concluded, with $K_{\rho_0} = \min\{1/6,1/(4\|\rho_0\|_{L^\infty(\Zcal)})\}/32$, setting
\begin{align*}
	\phi_{\rho_1}
	&:=\phi_{\rho_1,A}1_{\{\Lambda_{\rho_1}(A)
	-\Lambda_{\rho_0}(A) \ge\Lambda_{\rho_0}(A^c)-\Lambda_{\rho_1}(A^c)\}}
	+\phi_{\rho_1,A^c}1_{\{\Lambda_{\rho_0}(A^c)-\Lambda_{\rho_1}(A^c)
	 \ge \Lambda_{\rho_1}(A)-\Lambda_{\rho_0}(A)\}}.
\end{align*}
\end{proof}

	The final auxiliary result employs the tests for simple alternatives of Lemma \ref{Lem:Tests} to construct tests to control the numerator of posterior distributions.

\begin{lemma}\label{Lem:GlobalTests}
Let $\Rcal_n\subseteq\Rcal$ be measurable sets satisfying condition \eqref{Eq:MetricEntropy} for some $C_3>0$ and a positive sequence $\epsilon_n\to0$ such that $n \epsilon_n^2\to\infty$. Then, for all $M>\max\{(C_3/$ $K_{\rho_0})^{1/2},1\}$, with $K_{\rho_0}>0$ the constant defined in the statement of Lemma \ref{Lem:Tests} and all $n$ large enough, there exists a test $\phi_n$ such that 
$$
	E_{\rho_0}^{(n)}[\phi_n|Z^{(n)}]
	\le 2e^{-(K_{\rho_0}M^2 - C_3)n\epsilon_n^2},
$$
and 
$$
	\sup_{\rho\in\Rcal_n : \|\lambda_\rho^{(n)} - \lambda_{\rho_0}^{(n)}\|_{L^1(\Wcal_n)}\ge M n\epsilon_n}
	E^{(n)}_\rho[1-\phi_n |Z^{(n)}]
	\le 2e^{-K_{\rho_0}M^2 n\epsilon_n^2}.
$$
\end{lemma}

\begin{proof}
Writing $\lambda_\rho = \lambda^{(n)}_\rho$, cover the set $\{\rho\in\Rcal_n: \|\lambda_\rho - \lambda_{\rho_0}\|_{L^1(\Wcal_n)}\ge M n\epsilon_n\}$ by sup-norm balls of radius $\epsilon_n/2$ and centres $(\rho_l)_{l=1}^{\Ncal_n}$, where $\Ncal_n$ is the covering number by balls of such sup-norm radius. For each $\rho_l$, by Lemma \ref{Lem:Tests}, there exists a test $\phi_{\rho_l}$ satisfying
$$
	E_{\rho_0}^{(n)}[\phi_{\rho_l}|Z^{(n)}]
	\le2e^{-K_{\rho_0}\|\lambda^{(n)}_{\rho_l}-\lambda^{(n)}_{\rho_0}\|_{L^1(\Wcal_n)}
	\min\left\{1,
	\frac{1}{n}\|\lambda^{(n)}_{\rho_l}-\lambda^{(n)}_{\rho_0}\|_{L^1(\Wcal_n)}\right\}},
$$
and
\begin{align*}
	&\sup_{\rho:\|\lambda^{(n)}_\rho-\lambda^{(n)}_{\rho_l}\|_{L^1(\Wcal_n)}
	\le \frac{1}{2}\|\lambda^{(n)}_{\rho_l}-\lambda^{(n)}_{\rho_0}\|_{L^1(\Wcal_n)}}
	E_\rho^{(n)}[1-\phi_{\rho_1}|Z^{(n)}]\\
	&\qquad\qquad\qquad\ 
	\le 2e^{-K_{\rho_0}\|\lambda^{(n)}_{\rho_l}-\lambda^{(n)}_{\rho_0}\|_{L^1(\Wcal_n)}
	\min\left\{1,
	\frac{1}{n}\|\lambda^{(n)}_{\rho_l}-\lambda^{(n)}_{\rho_0}\|_{L^1(\Wcal_n)}\right\}}.
\end{align*}
If $\|\lambda_{\rho_l} - \lambda_{\rho_0}\|_{L^1(\Wcal_n)}\ge M n\epsilon_n$, we have for all $n$ large enough such that $\epsilon_n<1/M$,
\begin{align*}
	E_{\rho_0}^{(n)}[\phi_{\rho_l}|Z^{(n)}]
	\le2e^{-K_{\rho_0} Mn \epsilon_n\min\{1,M \epsilon_n \}}
	= 2e^{-K_{\rho_0}M^2n\epsilon_n^2},
\end{align*}
as well as
\begin{align*}
	\sup_{\rho:\|\lambda_\rho-\lambda_{\rho_l}\|_{L^1(\Wcal_n)}
	\le\frac{1}{2}\|\lambda_{\rho_l} - \lambda_{\rho_0}\|_{L^1(\Wcal_n)}}
	E^{(n)}_\rho[1-\phi_{\rho_l}|Z^{(n)}]
	&\le 2e^{-K_{\rho_0}M^2n\epsilon_n^2}.
\end{align*}

	Now set $\phi_n:=\max_{l=1,\dots,\Ncal_n}\phi_{\rho_l}$. Then, for all such $n$,
\begin{equation*}
	E_{\rho_0}^{(n)}[\phi_n|Z^{(n)}]
	\le \sum_{l=1}^{\Ncal_n}E_{\rho_0}^{(n)}[\phi_{\rho_l}|Z^{(n)}]
	\le 2\Ncal_ne^{-K_{\rho_0}M^2n\epsilon_n^2},
\end{equation*}
which, since $\Ncal_n \le \Ncal(\epsilon_n/2 ;\Rcal_n,\|\cdot\|_{L^\infty(\Zcal)})\le e^{C_3n\epsilon_n^2}$ by assumption, is bounded by
$$
	E_{\rho_0}^{(n)}[\phi_n|Z^{(n)}]\le 2e^{-(K_{\rho_0}M^2 - C_3)n\epsilon_n^2}.
$$
The first claim then follows upon taking $M^2>\max\{C_3/K_{\rho_0},1\}$. On the other hand, as for for each $\rho\in\Rcal_n$ with $\|\lambda_\rho - \lambda_{\rho_0}\|_{L^1(\Wcal_n)}>M n \epsilon_n$ there exists, by construction, a centre $\rho_l$ with 
$$
	\|\lambda_\rho - \lambda_{\rho_l}\|_{L^1(\Wcal_n)}
	\le n\|\rho - \rho_l\|_{L^\infty(\Zcal)} 
	\le \frac{1}{2} n\epsilon_n 
	\le \frac{1}{2} Mn\epsilon_n
	\le \frac{1}{2} \|\lambda_{\rho_l} - \lambda_{\rho_0}\|_{L^1(\Wcal_n)},
$$
we have
$$
	E_\rho^{(n)}[1 - \phi_n |Z^{(n)}] 
	\le E_\rho^{(n)}[1 - \phi_{\rho_l}|Z^{(n)}]\le 2e^{-K_{\rho_0}M^2n\epsilon_n^2}.
$$
It thus follows that for all $n$ large enough,
$$
	\sup_{\rho\in\Rcal_n : \|\lambda_\rho - \lambda_{\rho_0}\|_{L^1(\Wcal_n)}\ge Mn\epsilon_n}
	E_\rho^{(n)}[1 - \phi_n|Z^{(n)}] \le 2e^{-K_{\rho_0}M^2n\epsilon_n^2}.
$$
\end{proof}

%
%
%
%
%

\section{Proofs for Sections \ref{Subsec:RandGPRates} and \ref{subsec:mixtprior}}

\subsection{Proof of Proposition \ref{Theo:GPRandLoss}}
\label{Sec:ProofGPRandLoss}


%

The proof follows verifying the conditions of Theorem \ref{Theo:GenRandCov} with $\epsilon_n=n^{-\alpha/(2\alpha+d)}$ (or a sufficiently large multiple thereof) using standard techniques in the posterior contraction rate theory for Gaussian priors, \cite{vdVvZ08}. Starting with the small ball lower bound \eqref{Eq:SmallBall}, by construction (cf.~\eqref{Eq:GPPrior}), each intensity $\rho$ in the support of $\epsilon_n$ takes the form $\rho=\rho_w$ for some $w\in C([0,1]^d)$. Recalling the notation
$$
	\bar \rho(z) = \frac{\rho(z)}{M_\rho},
	\qquad z\in[0,1]^d,
	\qquad M_\rho=\int_{[0,1]^d}\rho(z)d\nu(z),
$$
standard computations (e.g.~as in the proof of Lemma 16 in \cite{G23}) imply, since the link $\eta$ is assumed to be uniformly Lipschitz, that if $\|w - w_0\|_\infty\lesssim \epsilon_n$ then
$$
	\max \left\{ \textnormal{KL}_\nu (\bar \rho_w,\bar \rho_{0}), 
	\int_{[0,1]^d} \bar \rho_{0}(z)\log^2\left( \frac{ \bar\rho_{0}(z) }{\bar \rho_w(z) } \right) d\nu(z),
	|M_{\rho_w} - M_{\rho_{0}}| \right\}
	\lesssim  \|w - w_0\|_\infty.
$$
Therefore, via Lemma \ref{Lem:GPSmallBall} below, for sufficiently large constants $K_1,K_2>0$, the prior probability in \eqref{Eq:SmallBall} is lower bounded by
$$
	\Pi_W(w:\|w - w_0\|_\infty\le K_1\epsilon_n)\ge e^{-K_2n\epsilon_n^2}.
$$
Turning to conditions \eqref{Eq:Sieves} and \eqref{Eq:MetricEntropy}, let 
$$
	\Rcal_n:=\left\{ \rho_w, \ w\in \Bcal_n\right\},
$$
with $\Bcal_n$ defined, for $H_1,H_2>0$ to be chosen, as in \eqref{Eq:GPSieves} below. By the first statement in Lemma \ref{Lem:GPSieves}, it follows that, for all sufficiently large $n$,
$$
	\Pi(\Rcal_n^c)\le \Pi_W(\Bcal_n^c)\le e^{-(2+2\|\rho_0\|_{L^\infty(\Zcal)}+K_2)n\epsilon_n^2},
$$
provided that $H_1,H_2$ are large enough. For such choices, the second statement in Lemma \ref{Lem:GPSieves} yields, in view of the assumed Lipschitzianity of $\eta$, that for some $K_3>0$,
$$
	\log \Ncal(\epsilon_n; \Rcal_n,\|\cdot\|_{L^\infty(\Zcal)})\le \log \Ncal(K_3\epsilon_n; \Bcal_n,\|\cdot\|_\infty)
	\lesssim n\epsilon_n^2,
$$
 concluding the proof.\qed

%

	The following lemma provides, for Gaussian process priors satisfying Condition \ref{Cond:GPCondition}, a lower bound for small ball probabilities in sup-norm used in the proof of Proposition \ref{Theo:GPRandLoss}.

\begin{lemma}\label{Lem:GPSmallBall}
Let $w_0,\ \Pi_W , \epsilon_n$ and $w_{0,n}$ be as in Proposition \ref{Theo:GPRandLoss}. Then, for all sufficiently large $L_1>0$ there exists $L_2>0$ such that
$$
	\Pi_W(w:\|w - w_0\|_\infty\le L_1\epsilon_n)\ge e^{-L_2n\epsilon_n^2}.
$$
\end{lemma}

\begin{proof}
By the triangle inequality, provided that $L_1$ is large enough, the probability of interest is lower bounded by
$$
	\Pi_W(w:\|w - w_{0,n}\|_\infty\le L_1\epsilon_n/2)
$$
which, using Corollary 2.6.18 of \cite{GN16}, since $\|w_{0,n}\|^2_{\Hcal_W}\lesssim n\epsilon_n^2$ by assumption, is greater than
$$
	e^{-\frac{1}{2}\|w_{0,n}\|_{\Hcal_W}^2}\Pi_W(w:\|w\|_\infty\le L_1\epsilon_n/2)
	\ge
	e^{-K_1n\epsilon_n^2}\Pi_W(w:\|w\|_\infty\le L_1\epsilon_n/2)
$$
for some $K_1>0$. Under Condition \ref{Cond:GPCondition}, the metric entropy estimate in Theorem 4.3.36 of \cite{GN16} yields, for some $K_2>0$, for all $\varepsilon>0$,
\begin{equation}
\label{Eq:SobolevMetricEntropy}
	\log \Ncal(\varepsilon;\{w:\|w\|_{\Hcal_W}\le 1\},\|\cdot\|_\infty)
	\le \log \Ncal(\varepsilon;\{w:\|w\|_{H^{\alpha+d/2}}\le K_2\},\|\cdot\|_\infty)
	\lesssim \varepsilon^{-d/(\alpha+d/2)}.
\end{equation}
Then, by Theorem 1.2 of \cite{LL99}, since $\epsilon_n\to0$,
$$
	\Pi_W(w:\|w\|_\infty\le L_1\epsilon_n/2)
	\ge e^{-K_3(L_1\epsilon_n)^{-d/\alpha}}
	\ge e^{-K_4n\epsilon_n^2}
$$
for some $K_3,K_4>0$, whence the claim follows with $L_2=K_1+K_4$.
\end{proof}

	Next, we construct, for the Gaussian priors of interests, sieves with bounded complexity and whose complementary sets have exponentially vanishing prior probability.

\begin{lemma}\label{Lem:GPSieves}
Let $\Pi_W $ and $\epsilon_n$ be as in Proposition \ref{Theo:GPRandLoss}. Define, for $H_1,H_2>0$, the sets
\begin{equation}
\label{Eq:GPSieves}
	\Bcal_n 
	=
		\{ w = w_1 + w_2: \|w_1\|_\infty \leq H_1\epsilon_n,
		\|w_2\|_{\Hcal_W} \leq H_2\sqrt n\epsilon_n\}.
\end{equation}
Then, for every $H_3>0$, there exists $H_1,H_2>0$ large enough such that, for all sufficiently large $n$,
$$
	\Pi_W(\Bcal_n^c) \leq e^{-H_3n \epsilon_n^2}.
$$
Furthermore, for every $H_1,H_2>0$,
$$
	\log \Ncal(\epsilon_n;\Bcal_n ,\|\cdot\|_\infty)\lesssim n\epsilon_n^2.
$$
\end{lemma}

\begin{proof}
Borell's isoperimetric inequality (e.g.~\cite{GN16}, Theorem 2.6.12) gives, with $\phi$ the standard normal cumulative distribution function,
\begin{align*}
\label{Isop}
	\Pi_W(\Bcal_n)
	&\ge \phi (\phi^{-1}(\Pi_W(w: \|w\|_{\infty} \leq H_1\epsilon_n) +  H_2 \sqrt n\epsilon_n).
\end{align*}
Provided that $H_1>0$ is sufficiently large, using Lemma \ref{Lem:GPSmallBall} and the standard inequality $\phi^{-1}(u) \geq -\sqrt{2\log (1/u)}$ for $0<u<1$, we obtain
$$
	\Pi_W(\Bcal_n) \ge \phi( (H_2 - K_1) \sqrt n\epsilon_n )
$$ 
for some $K_1>0$. Further, taking $H_2$ large enough, the quantity $(H_2 - K_1) \sqrt n\epsilon_n$ can be made larger than $-\phi^{-1}(e^{-H_3n\epsilon_n^2})$, whence the first claim follows since then
$$
	\phi( (H_2 - K_1) \sqrt n\epsilon_n )\ge \phi( -\phi^{-1}(e^{-H_3n\epsilon_n^2})) = 1 - e^{-H_3n\epsilon_n^2}.
$$
The second claim follows by applying the metric entropy estimate \eqref{Eq:SobolevMetricEntropy}, recalling the assumed continuous embedding of $\Hcal_W$ into $H^{\alpha+d/2}([0,1]^d)$, and noting that by construction, for some $K_2>0$,
$$
	\log \Ncal(\epsilon_n;\Bcal_n,\|\cdot\|_\infty)
	\le \Ncal(K_2\epsilon_n;\{w:\|w\|_{H^{\alpha+d/2}}\le H_2 n\epsilon_n^2\},\|\cdot\|_\infty)
	\lesssim (n\epsilon_n^2)^{d/(\alpha+d/2)}\le n\epsilon_n^2.
$$
\end{proof}

%
%
%

\subsection{Proof of Proposition \ref{th:Gauss:adapt:lambda}}
\label{Sec:ProofAdaptGP}


%

We verify the assumptions \eqref{Eq:SmallBall}-\eqref{Eq:MetricEntropy} of Theorem \ref{Theo:GenRandCov}, following the standard pattern for randomly truncated series priors, e.g.~\cite{arbel13}. Let $L_n\in\N$ be such that $2^{L_n} \simeq n^{1/(2\beta +d)}$, and let $w_{0,n}:=\sum_{l\le L_n}\sum_{k\le2^{ld}}\langle w_0,\psi_{lk}\rangle_{L^2}\psi_{lk}$ be the projection of $w_0$ onto the wavelet approximation space 
\begin{equation}
\label{Eq:ApproxSpace}
	\Psi_{L_n}:=\textnormal{span}(\psi_{lk}, \ l\le L_n, \ k=1,\dots,2^{ld}).
\end{equation}
Note that $\textnormal{dim}(\Psi_{L_n})=O(2^{L_nd})$ as $n\to\infty$, and that since $w_0\in C^\beta([0,1]^d)$,
\begin{equation}
\label{Eq:ApproxProp}
	\|w_0 - w_{0,n}\|_\infty 
	\lesssim
	2^{-\beta L_n} 
	\simeq n^{-\beta/(2\beta +d)} = o(\epsilon_n).
\end{equation}
See e.g.~\cite[Chapter 4.3]{GN16} for details. Arguing as in the proof of Proposition \ref{Theo:GPRandLoss}, using the triangle inequality and the Sobolev embedding $H^{d/2+\kappa}([0,1]^d)\subset C([0,1]^d)$, with arbitrarily small $\kappa>0$, the probability in \eqref{Eq:SmallBall} is lower bounded by
\begin{align*}
	\Pi&( w:\|w - w_{0,n}\|_{H^{d/2+\kappa}} \leq K_1 \epsilon_n ) \\
	&\ge\Pi_L(L=L_n) \Pr\left( \sum_{l=1}^{L_n}\sum_{k=1}^{2^{ld}}2^{2l(d/2+\kappa)}(g_{lk} 
	- \langle w_0,\psi_{lk}\rangle_{L^2})^2 \le (K_1\epsilon_n)^2 \right),
\end{align*}
for some $K_1>0$. In view of the tail assumption on $\Pi_L $ and the choice of $L_n$, the first term in the right hand side is greater than a multiple of $e^{-C_LL_n 2^{L_nd}}\ge e^{-K_2 n^{d/(2\beta+d)}\log n}=e^{-K_2 n \epsilon_n^2}$ for some $K_2>0$. The second term is lower bounded by
\begin{align*}
	\Pr&\Big(\textnormal{dim}(\Psi_n)\max_{l\le L_n, k\le 2^{ld}}(g_{lk} 
	- \langle w_0,\psi_{lk}\rangle_{L^2})^2 \le (K_1\epsilon_n)^2 
	2^{-2L_n(d/2+\kappa)}\Big)\\
	&\ge \prod_{l\le L_n, k\le 2^{ld}}\Pr(|g_{lk} - \langle w_0,\psi_{lk}\rangle_{L^2}|\le n^{-K_3})\\
	&\ge \prod_{l\le L_n, k\le 2^{ld}}e^{-\frac{|\langle w_0,\psi_{lk}\rangle_{L^2} |^2}{2}}\Pr(|g_{lk}|\le n^{-K_3})
	= e^{-\frac{\|w_{0,n}\|_{L^2}^2}{2}}\prod_{l\le L_n, k\le 2^{ld}}\Pr(|g_{lk}|\le n^{-K_3})
\end{align*}
for a sufficiently large constant $K_3>0$. Using that $\|w_{0,n}\|_{L^2}\le \|w_0\|_{L^2}<\infty$, and that, since $g_{lk}\iid N(0,1)$, for all $n$ large enough we have $\Pr(|g_{lk}|\le n^{-K_3})\ge K_4 n^{-K_3}$ for some $K_4>0$, the last display is lower bounded by a multiple of
$$
	(n^{-K_3})^{\textnormal{dim}(\Psi_{L_n})} \ge e^{-K_5 2^{L_nd}\log n} \ge e^{-K_6 n \epsilon_n^2},
$$
with $K_5,K_6>0$, concluding the derivation of Condition \eqref{Eq:SmallBall}. Moving onto Conditions \eqref{Eq:Sieves} and \eqref{Eq:MetricEntropy}, set
$$	
	\Rcal_n:=\{w\in\Vcal_{L_n+K_7} : \|w\|_{H^{d/2+\kappa}}\le n^{K_8}\}
$$
for arbitrarily small $\kappa>0$ and for $K_7,K_8>0$ to be chosen below. In particular, taking $K_8>d/(2\beta+d)$, we obtain
\begin{align*}
	\Pi(\Rcal_n^c)
	&\le \Pi_L(L>L_n+K_7) + \Pr\left(\sum_{l=1}^{L_n+K_7}\sum_{k=1}^{2^{ld}}2^{2l(d/2+\kappa)}g_{lk} ^2
	> n^{2K_8} \right)\\
	&\le e^{-C_L (L_n+K_7) 2^{(L_n+K_7)d}}
	+\Pr\left(2^{2(L_n+K_7)(d/2+\kappa)}\sum_{l=1}^{L_n+K_7}\sum_{k=1}^{2^{ld}}g_{lk} ^2
	> n^{2K_8} \right)\\
	&\le e^{- K_9 2^{K_7d}n\epsilon_n^2}+\Pr\left(\sum_{l=1}^{L_n+K_7}\sum_{k=1}^{2^{ld}}(g_{lk}^2-1)
	> \frac{n^{2K_8}}{2} \right)\\
	&\le e^{- K_9 2^{K_7d}n\epsilon_n^2}+ e^{-K_{10}n^{4 K_8}/(\textnormal{dim}(\Vcal_{L_n+K_7})+n^{2 K_8})},
\end{align*}
with $K_9>0$, the last inequality following from an application of Theorem 3.1.9 in \cite{GN16}.
Upon choosing $K_7$ and $K_8$ sufficiently large, the last display can be made smaller than $e^{-C_2n \epsilon_n^2}$ for any $C_2>0$, proving Condition \eqref{Eq:Sieves} for the hierarchical Gaussian wavelet prior $\Pi $. Finally,  by the metric entropy estimate for balls in Euclidean spaces, e.g.~\cite[Proposition 4.3.34]{GN16}, the metric entropy in \eqref{Eq:MetricEntropy} is upper bounded, as required, by a multiple of
$$
	\textnormal{dim}(\Vcal_{L_n+K_7})\log n = 2^{(L_n+K_7)d}\log n\simeq n \epsilon_n^2.
$$
\qed

%
%
%

\subsection{Proof of Proposition \ref{Theo:MixRandLoss}}
\label{Sec:ProofMixtures}


%

We verify the conditions of Theorem \ref{Theo:GenRandCov}. For the small ball lower bound \eqref{Eq:SmallBall}, using Lemma 2 in \cite{STG13} in the case $d\geq 2$ and Lemma 1 in \cite{KRvdV10} if $d=1$, there exists $\rho_\beta\in C^\beta(\R^d)$ such that $\rho_\beta \geq \rho_0/2$ and, with $\Sigma := \sigma^2 I_d$ and $\sigma \equiv \sigma_n\to 0$ as $n\to\infty$,  
$$ 
	\|\rho_0 - \varphi_\Sigma\ast \rho_\beta\|_\infty 
	= O(\sigma^\beta).
$$
Moreover, let $a_n := a_0 (\log n)^{1/\tau}$ where $\tau>0$ is such that $\nu( \{ u:|u| > z\}  ) \leq e^{ - z^\tau }$ and $a_0>0$ is large enough. Denoting, in slight abuse of notation, by $\nu$ the probability density function of the invariant distribution $\nu$ of $Z$, for any $\rho$ satisfying $|\log\rho(z)|\le  n^{b}z^2$ for some $b>0$ and all $z\in\R^d$, we also have
\begin{equation}
\label{trunc1}
	\int_{\{z:|z|>a_n\}} \rho_0(z) \nu(z) |\log \rho(z) - \log \rho_0(z)| dz 
	\lesssim \nu(\{z : |z|>a_n\}) + n^{b}\int_{\{z:|z|>a_n\}} \nu(z)|z|^2dz = O(n^{-1}),
\end{equation}
and using the construction of Lemma B1 in \cite{STG13} if $d\geq 2$ and that of Lemma 4 in \cite{KRvdV10} if $d=1$,   for all $H>0$, there exists a discrete measure $Q_0(\cdot) = \sum_{j=1}^{K_n}p_{j,0}\delta_{\mu_j^*}(\cdot)$ on $[-a_n, a_n]^d$ with at most $K_n=O(\sigma^{-d} a_n^d|\log \sigma|^d) $ support points such that, when $\sigma$ is small enough,
$$ 
	|\varphi_\Sigma\ast\rho_\beta(z) - \varphi_\Sigma\ast Q_0(z)|
	\le  \sigma^H , \qquad  \forall |z|<2a_n.
$$
Let $P(\cdot) = \sum_{j=1}^{K_n} p_j \delta_{\mu_j}(\cdot)$ with $\sum_{j =1}^{K_n}|p_j- p_{j,0} |\le  \sigma^{b_1\beta}$, $|\mu_j - \mu_j^*| \le  \sigma^{b_1\beta}$, for some $b_1>0$ large enough. Then,
$$
	|\varphi_\Sigma\ast\rho_\beta(z) - \varphi_\Sigma\ast P(z)|\le  2\sigma^H , \qquad \forall |z|<2a_n,
$$
and writing $\rho_{P, \Sigma} := \varphi_\Sigma\ast P$, we have
\begin{align*}
	\textnormal{KL}_\nu(\bar \rho_0, \bar\rho_{Q, \Sigma}) &
	= -\int \rho_0(z) \log\left(1 + \frac{\rho_{Q, \Sigma}(z)-\rho_0(z)}{ \rho_0(z)} \right) d\nu(z) + M_{\rho_0} -  M_{\rho_{Q, \Sigma}} \\
	& \lesssim \int_{\{z:|z|\le  a_n\}} \frac{(\rho_{Q, \Sigma}(z)-\rho_0(z))^2}{ \rho_0(z)} d\nu(z) 
	+\int_{\{z:|z|>a_n\}}\rho_0(z)+ \rho_{P, \Sigma} (1+n^b|z|^2)d\nu (z) \\
& \lesssim \sigma^{2\beta}+ O(n^{-1}) .
\end{align*}
This, as in Theorem 4 of \cite{STG13}, shows that condition \eqref{Eq:SmallBall} is verified with $v_0\sigma^\beta (\log n)^{t_0} =  \epsilon_n $  for some $t_0, v_0>0$.

	We proceed verifying the sieve and metric entropy conditions, \eqref{Eq:Sieves} and \eqref{Eq:MetricEntropy} respectively. In \cite{STG13}, metric entropy estimates in $L^1$-metric are obtained. Set $\xi_n:=n^{-\beta/(2\beta+d)}$. Define the set, for $\epsilon_n = (\log n)^t \xi_n$, $t>0$,
$$
	\mathcal Q_n := \left\{  (Q, \Sigma),\ Q := A\sum_{h=1}^H p_h \delta_{\mu_h}, A \le  A_n, 
	 \max_{h\le  H}|\mu_h| \le  b_n, \ \sum_{h>H}p_h<\epsilon_n\sigma_n^d/A_n,
	 \sigma_n^2 \le  \text{eig}_j(\Sigma) \le  \bar\sigma_n^2 \right\}
$$
with the choices $A_n := n^{a_1}$, for some $a_1>0$, $H:= H_0\lfloor n\xi_n^2/\log n\rfloor$ with $H_0>0$,  $b_n  := n^\gamma$, for some $\gamma>0$, with $\sigma_n := u (n\xi_n^2)^{-1/d}$ for  $u>0$ a small constant, and with $\bar\sigma_n := e^{C n\xi_n^2\log n}$. Note that a similar set (with normalised measures $Q$) is constructed for the proof of Theorem 4 of  \cite{STG13}. Using Proposition 2 of \cite{STG13} gives, for all $K>0$, upon taking $\gamma, C>0$ large enough, 
$$ 
	\Pi(\mathcal Q^c_n) \le  \Pi(A>A_n) + e^{-2Kn\epsilon_n^2} 
	\le   e^{-Kn\epsilon_n^2},
$$
verifying Condition \eqref{Eq:Sieves}. Next, $u_n = o(1)$ and let $\hat A$ be an $u_n \sigma_n^d$-net of $[0, A_n]$, let $\hat R $ be a $u_n/A_n$-net of $[-b_n, b_n]^d$, let $\hat O_k$ be a $\delta_n$-net of the set of unitary matrices for  $\delta_n:= u_n\sigma_n^{d+2}/A_n$, and let $\hat P$ be a $\sigma_n^d u_n/A_n $-net of the $H$-dimensional simplex. Set 
$$
	\hat{\mathcal L} := ( \sigma_n^2(1 + u_n\sigma_n^d /A_n)^j,\ j=0, \cdots, J_n),
$$ 
where $J_n\in\N$ is chosen so that $J_n \simeq (\log A_n +\log\bar \sigma_n-\log \sigma_n)/(u_n\sigma_n^d)$. For any $(\Sigma, Q) \in \mathcal Q_n $, with $\Sigma = O^t \Lambda O$, let $\tilde \Sigma :=  O^t \hat\Lambda O$ where $\hat \Lambda $ is the closest element in Frobenius norm to $\Lambda$ in the set $( \text{diag}(\hat \lambda_1, \cdots, \hat\lambda_k), \ \hat \lambda_j \in \hat{\mathcal L})$. Set $\hat \Sigma = \hat O^t \hat\Lambda \hat O$, and let $\hat p$ be the closest element to $\tilde p := (p_h/\sum_{h'\le  H}p_{h'}, h\le  H)$ and $\hat \mu_h$ the closest element to $\mu_h $ in $\hat R$ for $h \le  H$. Then, for $\hat \rho := \rho_{\hat Q, \hat \Sigma}$ and $\hat Q := \sum_{h\le  H} \hat p_h \delta_{\hat \mu_h}$, we have that
\begin{align*}
	|\hat \rho - \rho_{Q, \Sigma} |(z) 
	&\le  \frac{|\hat A-A| }{ \sigma_n^d} + A\frac{\sum_{h>H}p_h}{\sigma_n^d} 
	+ A\sum_{h\le  H}\frac{ |p_h - \hat p_h|}{ \sigma_n^d} 
	+ A\sum_{h\le  H}\hat p_h |\varphi_{\hat \Sigma}(z-\hat \mu_h)  - \varphi_{\tilde \Sigma}(z-\hat \mu_h)| \\
 	& \quad\ + A\sum_{h\le  H}\hat p_h [| \varphi_{\tilde \Sigma}(z-\hat \mu_h)  
	- \varphi_{\Sigma}(z-\hat \mu_h)|+A|\varphi_{\Sigma}(z-\hat \mu_h)  - \varphi_{ \Sigma}(z-\mu_h)|]\\
 	& \le  3u_n	+A_n [\text{det}[\hat\Lambda\Lambda^{-1}]^{1/2}-1]\sigma_n^{-d} +	A_n\left\| \varphi_{\tilde \Sigma}(z) \left( e^{- z^t[\hat \Sigma^{-1} - \tilde\Sigma^{-1}]z/2} - 1\right) \right\|_\infty  \\
 	 & \ \!+\! A_n\left\| \varphi_{\tilde \Sigma}(z) \left( e^{- z^t[ \Sigma^{-1} - \tilde\Sigma^{-1}]z/2} - 1\right) 	\right\|_\infty + A_n\sigma_n^{-d}\max_h \left\| \varphi(z)  
	|e^{-|\mu_h - \hat \mu_h||z|/\sigma_n}-1 |\right\|_\infty \\
	& \quad  +  A_n\max_h\frac{|\mu_h - \hat \mu_h|^2}{\sigma_n^{2+d}}\\
 	& \le  5u_n + du_n + 4 A_n\delta_n \sigma_n^{-(2+d)}\bar\sigma_n^2
	+ u_n^2 \sigma_n^d /A_n \leq (10+d)u_n
\end{align*} 
where we have used the fact that $\varphi_{ \Sigma}(z-\mu_h) \le  \sigma_n^{-d}$. Further using the inequality
$$ 
	|e^{- z^t[\hat \Sigma^{-1} - \tilde\Sigma^{-1}]z/2} - 1| 
	\le  |e^{- z^t\tilde \Sigma^{-1/2}[I_d  -\tilde\Sigma^{1/2} \hat\Sigma^{-1}\tilde\Sigma^{1/2}]
	\tilde \Sigma^{-1/2}z/2} - 1|,
$$
together with 
$$
	|I_d  -\tilde\Sigma^{1/2} \hat\Sigma^{-1}\tilde\Sigma^{1/2}|
	\le  3 |\hat \Lambda||O\hat O^T-I_k|^2|\hat \Lambda^{-1}| + |\hat \Lambda|^{1/2}|O\hat O^T-I_k|
	|\hat \Lambda^{-1/2}|
	\le  4\delta_n \sigma_n^{-2}\bar\sigma_n^2 \leq 4u_n \sigma_n^d /A_n
$$
and with
$$
	|I_d  -\tilde\Sigma^{1/2} \Sigma^{-1}\tilde\Sigma^{1/2}| 
	=| I_d - \Lambda^{1/2}\hat\Lambda^{-1}\Lambda^{1/2}|\le  \sigma_n^d u_n/A_n,
$$
leads, as required, choosing $u_n = u_0 \xi_n$ 
$$
	\log N( \epsilon_n; \mathcal Q_n, \| \cdot\|_\infty) \lesssim d\log J_n + d(d-1)/2| \log \delta_n| 
	+ n\epsilon_n^2 \lesssim n\epsilon_n^2 .
$$
\qed

%
%
%
%
%

\section{Proofs and supplementary materials for Section \ref{Subsec:ErgCov}}
\label{Sec:AdditionalErgCov}

\subsection{Proof of Theorem \ref{Theo:GaussWav}}\label{Sec:ProofMainGaussWav}

For $W$ as in \eqref{Eq:GaussWavPrior} with $\alpha=\beta$, the support and RKHS $\Hcal_W$ of $W$ are given by the wavelet approximation space $\Psi_{L_n}$ defined in eq.~\eqref{Eq:ApproxSpace}, cf.~\cite[Lemma 11.43]{GvdV17}. Further, the RKHS norm satisfies $\| \cdot\|_{\Hcal_W}=\|\cdot\|_{H^{\beta+d/2}}$, following from the wavelet characterisation of Sobolev spaces, e.g.~\cite[Section 4.3]{GN16}. Thus, $\Pi_W$ satisfies Condition \ref{Cond:GPCondition}. Since the link function is bijective and smooth, and since by assumption $\rho\in C^\beta([0,1]^d)\cap H^\beta([0,1]^d)$ is bounded away from zero, we have $\rho_0= \rho_{w_0} = \eta\circ w_0$ for $w_0=\eta^{-1}\circ \rho_0\in C^\beta([0,1]^d)\cap H^\beta([0,1]^d)$. Let $w_{0,n}:=\sum_{l\le L_n}\sum_{k\le2^{ld}}\langle w_0,\psi_{lk}\rangle_{L^2}\psi_{lk}$ be the wavelet projection of $w_0$. Then, by standard wavelet properties,
$$
	\|w_0 - w_{0,n}\|_\infty \lesssim 2^{-L_n\beta}\simeq \epsilon_n; \qquad
	\|w_{0,n}\|_{\Hcal_W}
	\le 2^{L_nd/2}\|w_{0,n}\|_{H^\beta}
	\lesssim \sqrt n \epsilon_n.
$$
An application of Theorem \ref{Theo:GPRandLoss} (and its proof) then implies that, for sufficiently large $M>0$, as $n\to\infty$,
\begin{align}
\label{Eq:Start}
	E_{\rho_0}^{(n)}
	\Bigg[\Pi\Big(\rho \in\Rcal_n: \frac{1}{n} \|\lambda^{(n)}_\rho - \lambda^{(n)}_{\rho_0}
	\|_{L^1(\Wcal_n)} &\le M\epsilon_n
	\Big| D^{(n)}\Big)\Bigg]\to1,
\end{align}
where $\Rcal_n = \{\rho_w, \ w\in\Bcal_n\}$ with 
\begin{equation}
\label{Eq:WaveletSieves}
	\Bcal_n =\{ w = w_1 + w_2: w_1,w_2\in \Psi_{L_n},\ \|w_1\|_\infty \leq K_1\epsilon_n,
		\|w_2\|_{H^{\beta+d/2}} \leq K_2\sqrt n\epsilon_n\},
\end{equation}
for large enough constants $K_1,K_2>0$. Noting that for all $w_1\in \Psi_{L_n}$ with $\|w_1\|_\infty \leq K_1\epsilon_n$ it holds $\|w_1\|_{H^{\beta+d/2}}\le 2^{L_n(\beta+d/2)}\|w_1\|_{L^2}\lesssim\sqrt n \epsilon_n,$
we have $\Bcal_n\subset \{w : \|w\|_{H^{\beta+d/2}}\le K_3 \sqrt n \epsilon_n\}$ for sufficiently large $K_3>0$. 

	Now for constants $K_4, K_5>0$ to be chosen below, consider the ($Z^{(n)}$-dependent) event
$$
	\Acal_n:=\left\{\|\rho - \rho_{0}\|_{L^1}
	\le \frac{K_4}{n} \|\lambda^{(n)}_\rho - \lambda^{(n)}_{\rho_{0}}
	\|_{L^1(\Wcal_n)}, \ \forall \rho\in\Rcal_n 
	: \|\rho - \rho_{0}\|_{L^1}>K_5\epsilon_n\right\}
$$
Then, writing $P_Z$ for the law of $Z^{(n)}$,
\begin{equation}
\label{Eq:FirstSteps}
\begin{split}
	&E_{\rho_0}^{(n)}
	\Bigg[\Pi\Big(\rho\in\Rcal_n : \|\rho - \rho_0\|_{L^1} > K_5\epsilon_n
	\Big| D^{(n)}\Big)\Bigg]\\
	&= E_{\rho_0}^{(n)}\left[\Pi\Big(\rho\in\Rcal_n : \|\rho - \rho_0\|_{L^1} 
	> K_5\epsilon_n
	\Big| D^{(n)}\Big)1_{\Acal_n}\right]
	+E_{\rho_0}^{(n)}\left[\Pi\Big(\rho\in\Rcal_n : \|\rho - \rho_0\|_{L^1} 
	> K_5\epsilon_n
	\Big| D^{(n)}\Big)1_{\Acal^c_n}\right]\\
	&\ \le E_{\rho_0}^{(n)}\left[\Pi\Big(\rho\in\Rcal_n : \|\lambda_\rho^{(n)} - \lambda^{(n)}_{\rho_0}\|_{L^1(\Wcal_n)} 
	> (K_5/K_4)\epsilon_n
	\Big| D^{(n)}\Big)\right]
	+P_Z(\Acal_n^c)
	=o(1)+P_Z(\Acal_n^c),
\end{split}
\end{equation}
provided that $K_5/K_4>M$, having used \eqref{Eq:Start} and the fact that on the event $\Acal_n$, 
$$
	\{\rho\in\Rcal_n : \|\rho - \rho_0\|_{L^1} 
	> K_5\epsilon_n\}\subset
	\{\rho\in\Rcal_n : \|\lambda_\rho^{(n)} - \lambda^{(n)}_{\rho_0}\|_{L^1(\Wcal_n)} 
	> (K_5/K_4)\epsilon_n\}.
$$
Thus, there remains to show that $P_Z(\Acal_n)\to1$. Set $\rho_{0,n}:=\eta\circ w_{0,n}$. For all $\rho\in\Rcal_n$, recalling that $\vol(\Wcal_n)=n$, that $\eta$ is uniformly Lipschitz, and using the inequality $| |u| - |v||\le |u - v|$ for all $u,v\in\R$,
\begin{align*}
	&\left| \frac{1}{n} \|\lambda^{(n)}_\rho - \lambda^{(n)}_{\rho_0}\|_{L^1(\Wcal_n)} - 
	\frac{1}{n} \|\lambda^{(n)}_\rho - \lambda^{(n)}_{\rho_{0,n}}\|_{L^1(\Wcal_n)} \right|\\
	&\ = \left| \frac{1}{n}\int_{\Wcal_n}|\eta(  w(Z(x))) - \eta(w_0(Z(x)))|dx - 
	\frac{1}{n}\int_{\Wcal_n}|\eta(  w(Z(x)) - \eta(  w_{0,n}(Z(x)))|dx \right|\\
	 &\ \le
	 \frac{1}{n}\int_{\Wcal_n}| \eta(w_0(Z(x)))
	  - \eta(  w_{0,n}(Z(x)))| dx
	  \lesssim 
	  \|w_0 - w_{0,n}\|_\infty
	  \lesssim \epsilon_n.
\end{align*}
Likewise, for all $\rho\in\Rcal_n$, $| \|\rho- \rho_0\|_{L^1} - \|\rho - \rho_{0,n}\|_{L^1} |\lesssim \epsilon_n$. Hence, in view of \eqref{Eq:FirstSteps}, the claim of Theorem \ref{Theo:GaussWav} is proved if we show that for sufficiently large $K_4,K_5>0$, 
$$
	P_Z\left(\|\rho - \rho_{0,n}\|_{L^1}
	\le \frac{K_4}{n} \|\lambda^{(n)}_\rho - \lambda^{(n)}_{\rho_{0,n}}
	\|_{L^1(\Wcal_n)}, \ \forall \rho\in\Rcal_n : \|\rho - \rho_{0,n}\|_{L^1}>K_5\epsilon_n\right)\to 1.
$$
Provided that $K_4>1$, the probability on the left hand is greater than
\begin{align*}
	&P_Z\Bigg(\frac{1}{K_4}\|\rho - \rho_{0,n}\|_{L^1}
	\le \frac{1}{n} \|\lambda^{(n)}_\rho - \lambda^{(n)}_{\rho_{0,n}}
	\|_{L^1(\Wcal_n)} \le K_4\|\rho - \rho_{0,n}\|_{L^1},
	\ \forall \rho\in\Rcal_n: \|\rho - \rho_{0,n}\|_{L^1}>K_5\epsilon_n\Bigg)\\
	&\ \ge 1 - P_Z\Bigg(
	\sup_{\rho\in\Rcal_n: \|\rho - \rho_{0,n}\|_{L^1}>K_5\epsilon_n}
	\left| \frac{\frac{1}{n} \|\lambda^{(n)}_\rho - \lambda^{(n)}_{\rho_{0,n}}
	\|_{L^1(\Wcal_n)} 
	- \|\rho - \rho_{0,n}\|_{L^1}}{\|\rho - \rho_{0,n}\|_{L^1}}\right|
	> \min\left\{1-\frac{1}{K_4},K_4-1 \right\}\Bigg).
\end{align*}
Thus, for $K_6 := 1-1/K_4\in(0,1)$, there remains to show that
\begin{equation}
\label{Eq:ProbInt}
\begin{split}
	P_Z\Bigg(\sup_{w\in\Bcal_n: \|\rho_w - \rho_{0,n}\|_{L^1}>K_5\epsilon_n}
	\Bigg| \frac{1}{n}\int_{\Wcal_n} 
	\frac{|\eta(  w(Z(x))) - \eta(  w_{0,n}(Z(x)))|}
	{\|\rho_w - \rho_{0,n}\|_{L^1}} -1 dx\Bigg|
	> K_6 \Bigg)
	\to0.
\end{split}
\end{equation}

	We proceed with a chaining argument. Let $(w_j, \ j\le J_n)$ be a $K_7\epsilon_n/\sqrt n$-net for $\{w\in\Bcal_n:\|\rho_w - \rho_{0,n}\|_{L^1}>K_5\epsilon_n\}$ in $\|\cdot\|_\infty$-metric, with $K_7>0$ to be chosen below. Note that by the metric entropy bound for Euclidean balls (e.g., \cite[Theorem 4.3.34]{GN16}), 
\begin{equation}
\label{Eq:EuclMetrEntr}
	J_n 
	\le e^{K_8\textnormal{dim}(\Psi_{L_n})\log n} 
	\le e^{K_9 2^{L_nd}\log n} 
	\le e^{K_{10} n\epsilon_n^2\log n}
\end{equation}
for $K_8,K_9,K_{10}>0$. Thus, for all $w\in \Bcal_n\cap \{w:\|\rho_w - \rho_{0,n}\|_{L^1}>K_5\epsilon_n\}$, there exists $w_{j^*}\in (w_j, \ j\le J_n)$ such that $\|w - w_{j^*}\|_\infty\le K_7 \epsilon_n/\sqrt n$, as well as, since $\eta$ is uniformly Lipschitz, $\|\rho_w - \rho_{w_{j^*}}\|_\infty\le K_{\eta} \|  w -   w_{j^*}\|_\infty \le K_\eta K_7 \epsilon_n/\sqrt n$ for some $K_{\eta}>0$. It follows that for all $x\in \Wcal_n$,
\begin{align*}
	\Bigg| &\frac{|\eta(  w(Z(x))) - \eta(  w_{0,n}(Z(x)))|}
	{\|\rho_w - \rho_{0,n}\|_{L^1}}
	- \frac{|\eta(  w_{j^*}(Z(x))) - \eta(  w_{0,n}(Z(x)))|}
	{\|\rho_{w_{j^*}} - \rho_{0,n}\|_{L^1}}  \Bigg|\\
	&\le 
	\frac{\Big| |\eta(w(Z(x))) - \eta(  w_{0,n}(Z(x)))| -
	|\eta(  w_{j^*}(Z(x))) - \eta(  w_{0,n}(Z(x)))| \Big|}
	{\|\rho_w - \rho_{0,n}\|_{L^1}}\\
	&\quad + |\eta(  w_{j^*}(Z(x))) - \eta(  w_{0,n}(Z(x)))|
	\frac{\Big| \|\rho_{w_{j^*}} - \rho_{0,n}\|_{L^1} - \|\rho_w  
	- \rho_{0,n}\|_{L^1}\Big|}{\|\rho_w  
	- \rho_{0,n}\|_{L^1}\|\rho_{w_{j^*}} 
	- \rho_{0,n}\|_{L^1}}\\
	&\le \frac{1}{K_5\epsilon_n} \|\rho_w  - \rho_{w_{j^*}} \|_\infty
	+ \frac{1}{K_5^2\epsilon_n^2}\|\rho_{w_{j^*}} - \rho_{0,n}\|_\infty
	 \|\rho_w  - \rho_{w_{j^*}} \|_{L^1}\\
	&\le\frac{K_\eta K_7}{K_5\sqrt n} 
	+ \frac{2}{K_5^2\epsilon_n^2}\sup_{w\in\Bcal_n}\|w\|_\infty\|\rho_w  - \rho_{w_{j^*}}\|_\infty
	\le \frac{K_\eta K_7}{K_5\sqrt n} + \frac{K_{11}}{K_5^2\epsilon_n^2}
	K_3 \sqrt n \epsilon_n\frac{K_\eta K_7\epsilon_n}{\sqrt n}
	\le K_6/2,
\end{align*}
upon taking $K_7>0$ sufficiently small, having used the fact that $\Bcal_n\subset \{w : \|w\|_{H^{\beta+d/2}}\le K_3 \sqrt n \epsilon_n\}$ and the continuous embedding $H^{\beta+d/2}([0,1]^d)\subset C([0,1]^d)$ holding for $\beta>0$. The probability in \eqref{Eq:ProbInt} is thus upper bounded by
\begin{equation}
\label{Eq:AlmostThere}
\begin{split}
	 P_Z\Bigg(\sup_{w\in\Bcal_n: \|\rho_w - \rho_{0,n}\|_{L^1}>K_5\epsilon_n}
	\Bigg| &\frac{1}{n}\int_{\Wcal_n} 
	\frac{|\eta(  w_{j^*}(Z(x))) - \eta(  w_{0,n}(Z(x)))|}
	{\|\rho_{w_{j^*}} - \rho_{0,n}\|_{L^1}} -1 dx\Bigg| + \frac{K_6}{2}
	> K_6 \Bigg)\\
	&\le
	P_Z\Bigg(\max_{j=1,\dots,J_n}
	\Bigg| \frac{1}{n}\int_{\Wcal_n} 
	\frac{|\eta(  w_{j}(Z(x))) - \eta(  w_{0,n}(Z(x)))|}
	{\|\rho_{w_j} - \rho_{0,n}\|_{L^1}} -1 dx\Bigg|
	> \frac{K_6}{2} \Bigg)\\
	&\le
	J_n\sup_{w\in\Bcal_n:\|\rho_w - \rho_{0,n}\|_{L^1}>K_5\epsilon_n}
	P_Z\Bigg(
	\Bigg| \frac{1}{n}\int_{\Wcal_n} 
	f_w(Z(x)) dx\Bigg|
	> \frac{K_6}{2} \Bigg),
\end{split}
\end{equation}
where
$$
	f_w := \frac{ |\rho_w - \rho_{0,n}|}{\|\rho_w  
	- \rho_{0,n}\|_{L^1}} - 1, \qquad  w\in\Bcal_n:\|\rho_w - \rho_{0,n}\|_{L^1}>K_5\epsilon_n.
$$
An application of the empirical process concentration inequality in Proposition \ref{Prop:MultGaussSupConc}, with the class $\Fcal_n := \{f_w,0\}$, now yields that for some $K_{12}>0$,
\begin{equation}
\label{Eq:AlmostThere}
	P_Z\left(\Bigg| \frac{1}{n}\int_{\Wcal_n} 
	f_w(Z(x)) dx\Bigg|\ge \frac{K_{12}}{\sqrt n}\|\nabla f_w\|_{L^\infty([0,1]^d;\R^d)}(1 + y)\right)\le e^{-\frac{y^2}{2}},
	\qquad \forall y>0.
\end{equation}
Next, by Lemma \ref{Lem:GradientNorm}, we have $\| \nabla f_w\|_{L^\infty([0,1]^d;\R^d)} \lesssim n^\frac{d(1+a_\eta/2)+1+\kappa}{2\beta+d}$ for all $w\in\Bcal_n \cap\{w:\|\rho_w - \rho_{0,n}\|_{L^1}>K_5\epsilon_n\}$. Taking $y:= K_{13}\sqrt n \epsilon_n\sqrt{\log n}$ in \eqref{Eq:AlmostThere} with $K_{13}>0$ to be chosen below then yields, for some $K_{14}>0$,
$$
	P_Z\left(\Bigg| \frac{1}{n}\int_{\Wcal_n} 
	f_w(Z(x)) dx\Bigg|>K_{14} n^\frac{d(1+a_\eta/2)+1+\kappa}{2\beta+d} \epsilon_n\sqrt{\log n} \right)
	\le e^{-\frac{K_{13}^2}{2}n \epsilon_n^2\log n}.
$$
Finally, noting that, as $\beta > d(1+a_\eta/2)+1$ by assumption and $\kappa>0$ is arbitrarily small,
$$
	n^\frac{d(1+a_\eta/2)+1+\kappa}{2\beta+d}
	\epsilon_n\sqrt{\log n} = n^\frac{d(1+a_\eta/2)+1+\kappa -\beta}{2\beta+d}
	\sqrt{\log n}\to 0,
$$
we have
\begin{align*}
	\sup_{w\in\Bcal_n:\|\rho_w - \rho_{0,n}\|_{L^1}>K_5\epsilon_n}
	P_Z\Bigg(
	\Bigg| \frac{1}{n}\int_{\Wcal_n} 
	\frac{|\eta(w(Z(x))) - \eta(  w_{0,n}(Z(x)))|}
	{\|\rho_w - \rho_{0,n}\|_{L^1}} &-1 dx\Bigg|
	> \frac{K_6}{2} \Bigg)
	\le e^{-\frac{K_{13}^2}{2}n \epsilon_n^2\log n}.
\end{align*}
Combined with \eqref{Eq:EuclMetrEntr} and \eqref{Eq:AlmostThere} this gives, as required, that the probability in \eqref{Eq:ProbInt} is upper bounded by
$
	e^{K_{10} n\epsilon_n^2\log n} e^{-\frac{K_{13}^2}{2}n\epsilon_n^2\log n}\to 0
$, upon taking $K_{13}>\sqrt{K_{10}/2}$. 

\qed

%

	The following lemma is used in the proof of Theorem \ref{Theo:GaussWav}, providing the required gradient sup-norm bound for the application of the empirical process concentration inequality derived in Section \ref{Sec:EmpProcSup} below. Let $W^{1,\infty}([0,1]^d)$ be the Sobolev space of functions $f\in L^\infty([0,1]^d)$ with weak partial derivatives $\partial_h f\in L^\infty([0,1]^d)$, $h=1,\dots,d$. For $f\in W^{1,\infty}([0,1]^d)$, the weak gradient is given by $\nabla f:=(\partial_1 f,\dots,\partial_d f)\in L^\infty([0,1]^d;\R^d)$.

\begin{lemma}\label{Lem:GradientNorm}
Let $\eta:\R\to(0,\infty)$ be smooth, uniformly Lipschitz, strictly increasing link functions with bounded and uniformly Lipschitz derivative $\eta'$ satisfying the left tail condition \eqref{Eq:LinkTail} for some $a>0$. Let $\Bcal_n$ be the set in \eqref{Eq:WaveletSieves}, with $\Psi_{L_n}$ the wavelet approximation space in \eqref{Eq:ApproxSpace} at level $L_n\in\N$ such that $2^{L_n}\simeq n^{1/(2\beta+d)}$, with $\epsilon_n=n^{-\beta/(2\beta+d)}$ and with $\beta,K_1,K_2>0$. For $w\in \Bcal_n$, recall the notation $\rho_w=\eta\circ w$. Let $(w_{0,n})_{n\ge1}\subset \Psi_{L_n}$ be a fixed sequence satisfying $\|w_{0,n}\|_\infty\le K_3 \sqrt n \epsilon_n$ for some sufficiently large $K_3>0$, and define the functions
$$
	f_w := \frac{ |\rho_w - \rho_{w_{0,n}}|}{\|\rho_w  
	- \rho_{w_{0,n}}\|_{L^1}} - 1, \qquad  w\in\Bcal_n,
$$
Then, $f_w\in W^{1,\infty}([0,1]^d)$ and
$$
	\| \nabla f_w\|_{L^\infty([0,1]^d;\R^d)} \le K_4 n^\frac{d(1+a_\eta/2)+1+\kappa}{2\beta+d},
$$
for some sufficiently large $K_4>0$ and where $\kappa>0$ is arbitrarily small.
\end{lemma}

\begin{proof}
The fact that $f_w\in W^{1,\infty}([0,1]^d)$ follows from the regularity of the wavelet basis and the assumed smoothness of the link function $\eta$. In particular, we have
\begin{align*}
	\| \nabla f_w\|_{L^\infty([0,1]^d;\R^d)}=\frac{\| \eta'\circ w \nabla w- \eta'\circ w_{0,n}\nabla w_{0,n}\|_{L^\infty([0,1]^d;\R^d)}}
	{\|\rho_w  - \rho_{0,n}  \|_{L^1}}.
\end{align*}
Fix $w\in\Bcal_n$. The numerator is bounded by
\begin{equation*}
\begin{split}
	\| \eta'&\circ w\|_{L^\infty}\| \nabla w- \nabla w_{0,n}\|_{L^\infty([0,1]^d;\R^d)} 
	+ \| \nabla w_{0,n}\|_{L^\infty([0,1]^d;\R^d)}\|\eta'\circ w - \eta'\circ w_{0,n}\|_{L^\infty}\\
	&\lesssim \| \nabla w - \nabla w_{0,n}\|_{L^\infty([0,1]^d;\R^d)} +
	\|  w -   w_{0,n}\|_{L^\infty}
\end{split}
\end{equation*}
having used that $\eta'$ is bounded and Lipschitz. For the wavelet-Besov spaces $B^\alpha_{pq}([0,1]^d)$, $\alpha\ge0$, $p,q\in[1,\infty]$, defined e.g.~as in \cite[p.370]{GN16}, recall the continuous embeddings $B^{1+d+\kappa}_{1\infty}([0,1]^d)\subset W^{1,\infty}$ $([0,1]^d)$ (e.g., \cite[p.360]{GN16}) and $L^1([0,1]^d)\subset B^0_{1\infty}([0,1]^d)$ (e.g., \cite[Proposition 4.3.11]{GN16}), implying that the last display is upper bounded by a multiple of
$$
	\|w - w_{0,n}\|_{B^{1+d+\kappa}_{1\infty}}
	\le  2^{J_n(1+d+\kappa)}\|w - w_{0,n}\|_{B^0_{1\infty}}
	\lesssim n^{\frac{1+d+\kappa}{2\beta+d}}\|w - w_{0,n}\|_{L^1}.
$$
For the denominator, note that for all $z\in[0,1]^d$,
\begin{align*}
	|  w(z) -   w_{0,n}(z)|
	&=|\eta^{-1}(\eta(  w(z))) - \eta^{-1}(\eta(  w_{0,n}(z))|
	=\frac{1}{\eta'(\eta^{-1}(\zeta))}| \eta(  w(z)) - \eta(  w_{0,n}(z))|
\end{align*}
for some $\zeta$ lying between $ \eta(  w(z))$ and $\eta(  w_{0,n}(z))$. As argued at the beginning of the proof of Theorem \ref{Theo:GaussWav}, $\Bcal_n\subset \{w:\|w\|_\infty\le K_3\sqrt n \epsilon_n\}$ provided that $K_3>0$ is large enough. Then, since $\|w\|_\infty,\|w_{0,n}\|_\infty\le K_{3}\sqrt n \epsilon_n$, and $\eta$ is increasing, necessarily $\eta(  w(z)), \eta(  w_{0,n}(z))\in[\eta(-K_{3}\sqrt n \epsilon_n),\eta(K_{3}\sqrt n \epsilon_n)]$ for all $z\in[0,1]^d$, whence
\begin{align*}
	|  w(z) -   w_{0,n}(z)|
	\le \frac{1}{\min_{u\in [\eta(-K_3\sqrt n \epsilon_n),\eta(K_3\sqrt n \epsilon_n)]}\eta'(\eta^{-1}(u))}
	| \eta(  w(z)) - \eta(  w_{0,n}(z))|.
\end{align*}
Since $\eta^{-1},\eta$ are increasing and $\eta'(v)>1/|v|^{a_\eta}$ for all $v<v_0$ by assumption, the right hand side is upper bounded by
\begin{align*}
	\frac{1}{\min_{v\in [-K_3\sqrt n \epsilon_n,K_3\sqrt n \epsilon_n]}\eta'(v)}
	| \eta(  w(z)) - \eta(  w_{0,n}(z))|
	&= \frac{1}{\eta'(-K_3\sqrt n \epsilon_n)}| \eta(  w(z)) - \eta(  w_{0,n}(z))|\\
	&\lesssim (\sqrt n \epsilon_n)^{a_\eta}| \eta(  w(z)) - \eta(  w_{0,n}(z))|.
\end{align*}
It follows that
\begin{align*}
	{\|\rho_w  - \rho_{0,n}  \|_{L^1}}
	\gtrsim \frac{1}{(\sqrt n \epsilon_n)^{a_\eta}} \|w - w_{0,n}\|_{L^1} = n^{-\frac{a_\eta d/2}{2\beta+d}}\|w - w_{0,n}\|_{L^1},
\end{align*}
which combined with the above bound for the numerator shows that for all $w\in\Bcal_n$,
$$
	\| \nabla f_w\|_{L^\infty([0,1]^d;\R^d)}
	\lesssim n^{\frac{1+d+\kappa}{2\beta+d}}n^{\frac{a_\eta d/2}{2\beta+d}}
	=n^\frac{d(1+a_\eta/2)+1+\kappa}{2\beta+d}.
$$
\end{proof}

%
%
%

\subsection{Ergodic Gaussian covariates; priors on functions with bounded gradient}
\label{Subsec:BoundGaussRates}

The foundational tool underlying the proof of Theorem \ref{Theo:GaussWav} is an exponential concentration inequality for spatial averages of (multivariate) stationary ergodic Gaussian processes, derived in Proposition \ref{Prop:MultGaussConcIneq} below. The key technical challenge in proving Theorem \ref{Theo:GaussWav} is then to lift such inequality to a suitable set of posterior concentration, cf.~\eqref{Eq:ProbInt}, which we have achieved through a chaining argument and a delicate use of the truncated wavelet structure of the prior.

	On the other hand, we note that the basic exponential inequality from Proposition \ref{Prop:MultGaussConcIneq} holds uniformly over sets of functions with uniformly bounded gradient. This lead us to consider the case of  general posterior distributions that asymptotically concentrate over such collections. For these, the abstract theory for the empirical $L^1$-loss developed in Section \ref{Subsec:GenTheo} can straightforwardly be combined with Proposition \ref{Prop:MultGaussConcIneq} to obtain posterior contraction rates in standard $L^1$-distances.

	Similarly to the setting of Theorem \ref{Theo:GaussWav}, we consider multivariate random covariate fields arising as (possibly transformed) stationary and ergodic Gaussian processes, under slightly more general conditions.

\begin{condition}\label{Cond:GenGaussCov}
Let $\tilde Z^{(h)}:=(\tilde Z^{(h)}(x), \ x\in\R^D),\ h=1,\dots,d,$ be independent, almost surely locally bounded, centred and stationary Gaussian processes with integrable covariance functions $K^{(h)}\in L^1(\R^D)$, where $K^{(h)}(x):=\textnormal{Cov}(Z^{(h)}(x),$ $Z^{(h)}(0))$, $x\in\R^D$. Further assume without loss of generality that $K^{(h)}(0)=1$, for $h=1,\dots,d$. Let the covariate process $Z=(Z(x), \ x\in \R^D)$ be given by $Z(x) := \Phi(\tilde Z(x))$, where $\Phi:\R^d\to \Zcal$ is a continuously differentiable map with uniformly bounded partial derivatives. Let $\nu $ be the stationary distribution of $Z$, given by the push-forward of the $d$-variate standard normal distribution under $\Phi$.
\end{condition}

	Note that Condition \ref{Cond:BoundGaussCov} is a special case of Condition \ref{Cond:GenGaussCov}, corresponding to a particular choice for the map $\Phi$. In the result to follow, more general transformations, including the identity $\Phi(z)=z$, are allowed.

\begin{proposition}\label{Theo:GaussGene}
Let $\Wcal_n\subset\R^D$ be a measurable and bounded set satisfying \eqref{Eq:SpatialWn}. Let $\rho_0\in C^1(\Zcal)$ be non-negative valued. Consider data $D^{(n)}\sim P^{(n)}_{\rho_0}$ from the observation model \eqref{Eq:PointProc} with $\rho=\rho_0$ and $Z$ a stationary random field satisfying Condition \ref{Cond:GenGaussCov}. Assume that the prior $\Pi $ is supported on $C^1(\Zcal)$ and satisfies \eqref{Eq:SmallBall} - \eqref{Eq:MetricEntropy} for some positive sequence $\epsilon_n\to0$ such that $n\epsilon_n^2\to\infty$. Further assume that, for some $M_1>\|\rho_0\|_{C^1}$, 
\begin{equation}
\label{Eq:BoundGradPost}
	E_{\rho_0}^{(n)}
	\left[\Pi\left(\rho:\| \nabla \rho\|_{L^\infty(\Zcal;\R^d)} 
	> M_1 \Big|D^{(n)}\right)\right]\to0
 \end{equation}
as $n\to\infty$. Then, for all sufficiently large $M_2>0$, as $n\to\infty$,
$$
	E_{\rho_0}^{(n)}
	\Bigg[\Pi\Big(\rho : \|\rho - \rho_0\|_{L^1(\Zcal,\nu)} > M_2 \epsilon_n
	\Big| D^{(n)}\Big)\Bigg]
	\to 0.
$$
\end{proposition}

\begin{proof}
Let $U_n := \{\rho\in\Rcal_n: \|\lambda^{(n)}_\rho - \lambda^{(n)}_{\rho_0}\|_{L^1(\Wcal_n)} \le Mn\epsilon_n,\ \| \nabla \rho\|_{L^\infty(\Zcal;\R^d)} \le M_1\}$, satisfying by assumption and by an application of Theorem \ref{Theo:GenRandCov}, for sufficiently large $M>0$, as $n\to\infty$, 
\begin{align*}
	E_{\rho_0}^{(n)}
	[\Pi(U_n| D^{(n)})]\to1.
\end{align*}
Let $V_n:=\{\rho\in C^1(\Zcal):\|\rho - \rho_0\|_{L^1(\Zcal,\nu)} \le M_2\nu_n\}$ for $M_2>0$ to be chosen below. Then,
\begin{align*}
 	\Pi( V_n^c | D^{(n)}) 
	&= \Pi( V^c_n \cap U_n | D^{(n)}) + o_{P^{(n)}_{\rho_0}}(1) 
	= \frac{ \int_{V^c_n \cap U_n} e^{l_n(\rho) - l_n(\rho_0)}d\Pi(\rho) }
	{ \int_{\mathcal R} e^{l_n(\rho) - l_n(\rho_0) } } + o_{P^{(n)}_{\rho_0}}(1).
\end{align*}
Denote by $D_n$ the denominator in the previous display. The proof of Theorem \ref{Theo:GenRandCov} shows that $P_{\rho_0}^{(n)}( D_n \leq e^{ - K_1 n \epsilon_n^2 } ) = o(1)$ for some constant $K_1>0$, so that by Fubini's theorem,
\begin{align*}
 	E_0^n& \left[\Pi( V_n^c | D^{(n)})\right] \\
	&\leq e^{  K_1 n \epsilon_n^2 }  
 	\int_{V_n^c \cap \{\rho\in\mathcal R_n:\|\nabla\rho\|_{L^\infty(\Zcal;\R^d)}\le M_1\}} P_{Z^{(n)}} (\|\lambda^{(n)}_\rho - \lambda^{(n)}_{\rho_0}
	\|_{L^1(\Wcal_n)} \le Mn\epsilon_n ) d\Pi(\rho)+ o(1).
\end{align*}
Fix any $\rho\in V_n^c \cap \{\rho\in\mathcal R_n:\|\nabla\rho\|_{L^\infty(\Zcal;\R^d)}\le M_1\}$. Then, if $\|\lambda^{(n)}_\rho - \lambda^{(n)}_{\rho_0}\|_{L^1(\Wcal_n)} \le Mn\epsilon_n $, necessarily 
$$
	\Delta^{(n)}(\rho) := \| \rho - \rho_0\|_{L^1(\mathcal Z)} 
	- \frac{1}{n}\|\lambda^{(n)}_\rho - \lambda^{(n)}_{\rho_0}
	\|_{L^1(\Wcal_n)} > (M_2-M)\epsilon_n
	\geq \frac{M_2}{2}\epsilon_n  
$$ 
upon taking $M_2 > 2M$. For all such $M_2$, it follows that
$$ 
	E_0^n \left[\Pi( V_n^c | D^{(n)})\right]  
	\leq e^{  K_1 n \epsilon_n^2 }  \int_{V_n^c \cap \{\rho\in\mathcal R_n:\|\nabla\rho\|_{L^\infty(\Zcal;\R^d)}\le M_1\}}
	 P_{Z^{(n)}} (\Delta^{(n)}(\rho) > M_2\epsilon_n/2 )d\Pi(\rho)+ o(1)  .  
$$ 
The concentration inequality in Proposition \ref{Prop:MultGaussConcIneq}, applied with $f := |\rho - \rho_0|- \|\lambda^{(n)}_\rho - \lambda^{(n)}_{\rho_0}\|_{L^1(\Wcal_n)}/n$, for $\rho\in C^1(\Zcal)$, whose (weak) gradient satisfies $ \| \nabla f  \|_{L^\infty(\Zcal;\R^d)}\le \| \nabla\rho  \|_{L^\infty(\Zcal;\R^d)}+\| \nabla \rho_0  \|_{L^\infty(\Zcal;\R^d)}$ now gives that
$$
 	\sup_{\rho\in V_n^c \cap \{\rho\in\mathcal R_n:\|\nabla\rho\|_{L^\infty(\Zcal;\R^d)}\le M_1\}}
	P_{Z^{(n)}} (\Delta^{(n)}(\rho) > M_2\epsilon_n/2 )\leq e^{- K_2 (M_2)^2 n\epsilon_n^2 }
$$
for some $K_2>0$. The claim then follows Taking $M_2>0$ large enough and combining the last two displays.
\end{proof}

The formulation of Proposition \ref{Theo:GaussGene} is slightly non-standard, as the assumption \eqref{Eq:BoundGradPost} involves the asymptotic behaviour of the posterior. A canonical way to verify this property is through the prior construction; in particular, if the sieves $\Rcal_n$ in \eqref{Eq:Sieves} satisfy
\begin{equation}\label{Eq:BoundSieves}
	\Rcal_n\subset\{\rho\in C^1(\Zcal) :\| \nabla \rho\|_{L^\infty(\Zcal;\R^d)} 
	\le M_1\},
\end{equation}
then, under the remaining assumptions of Proposition \ref{Theo:GaussGene} (which are standard), Theorem \ref{Theo:GenRandCov} implies that \eqref{Eq:BoundGradPost} is indeed valid. Examples of priors for which sieves satisfying \eqref{Eq:BoundSieves} can be constructed are basis expansions with uniform coefficients, e.g.~in the spirit of those considered in \cite{BSvZ15}; see also \cite{N20}. Alternatively, rescaled Gaussian process priors from recent inverse problems literature could also be employed, \cite{N23}.

	The statement of Proposition \ref{Theo:GaussGene} allows for some additional flexibility in cases where the property \eqref{Eq:BoundGradPost} may be obtained through a different argument, for example if posterior consistency (without rates or with sub-optimal ones) can be established in stronger H\"older- or Sobolev norms, cf.~\cite{DFM24,DFG24}. This also suggests to `enforce' condition \eqref{Eq:BoundGradPost} through `post-processing' posterior distributions achieving optimal posterior contraction rates in the empirical $L^1$-norm from Theorem \ref{Theo:GenRandCov}. For example, if an upper bound $\| \nabla \rho_0\|_{L^\infty(\Zcal;\R^d)}\le M_1$ is available, the (hierarchical) Gaussian priors and the mixture of Gaussians priors considered in Sections \ref{Subsec:RandGPRates} and \ref{subsec:mixtprior} respectively may be restricted to the ball $\Rcal_0:=\{\rho\in C^1(\Zcal):\| \nabla \rho\|_{L^\infty(\Zcal;\R^d)} \le M_1\}$, whereby the small ball probability lower bound \eqref{Eq:SmallBall} and the sieve condition \eqref{Eq:Sieves} need to be verified with $B_{n,2}(\rho_0)$ replaced by $B_{n,2}(\rho_0)\cap \Rcal_0$ and with $\Rcal_n$ replaced by $\Rcal_n\cap\Rcal_0$, respectively. In practice, incorporating such prior truncation step into inferential procedures often requires only minor methodological modifications; for instance, within a posterior sampling algorithm based on the unrestricted prior, it entails discarding the draws not belonging to $\Rcal_0$ (and normalising accordingly).
	
%
%
%

\subsection{Poisson random tessellations; Proof of Theorem \ref{Cor:GaussPoissTess} and extensions}
\label{Subsec:SupplPoissRates}

In Section \ref{Subsec:PoissRates}, we consider ergodic covariate processes arising as a Poisson random tessellation, cf.~Definition \ref{Cond:PoissTess}. For this class, using recent results in \cite{DG20AHL,DG20ALEA}, we derive exponential concentration inequalities  that hold uniformly over sup-norm balls, cf.~Section \ref{Subsec:PoissTessCov}. In  Proposition \ref{Theo:PoissTess} below, we show that, for general posterior concentrating over sets of uniformly bounded functions, these can be combined with Theorem \ref{Theo:GenRandCov} to derive posterior contraction rates in standard $L^1$-distances.

	Given the following result, Theorem \ref{Cor:GaussPoissTess} follows as an immediate corollary, upon noting that by the calculations in the proof of Proposition \ref{Theo:GPRandLoss}, the considered prior $\Pi $ satisfies the conditions \eqref{Eq:SmallBall} - \eqref{Eq:MetricEntropy} of Theorem \ref{Theo:GenRandCov} with $\epsilon_n=n^{-\beta/(2\beta+d)}$, and further that the asymptotic boundedness requirement \eqref{boundedrho} below is verified in view of the assumed upper bound on the link function.

\begin{proposition}\label{Theo:PoissTess}
Let $\Wcal_n\subset\R^D$ be a measurable and bounded set satisfying \eqref{Eq:SpatialWn}. Let $\rho_0\in L^\infty(\Zcal)$ be non-negative valued. Consider data $D^{(n)}\sim P^{(n)}_{\rho_0}$ from the observation model \eqref{Eq:PointProc} with $\rho=\rho_0$ and $Z$ a stationary random field constructed as in Condition \ref{Cond:PoissTess}. Assume that the prior $\Pi $ satisfies Conditions \eqref{Eq:SmallBall} - \eqref{Eq:MetricEntropy} for some positive sequence $\epsilon_n\to0$ such that $n\epsilon_n^2\to\infty$. Further assume that, for some $M_1>\|\rho_0\|_{L^\infty(\Zcal)}$, 
\begin{equation}
\label{boundedrho}
	E_{\rho_0}^{(n)}\left[\Pi\left(\rho:\| \rho\|_{L^\infty(\Zcal)} 
	> M_1 \Big|D^{(n)}\right)\right]\to0
 \end{equation}
as $n\to\infty$. Then, for all sufficiently large $M_2>0$, as $n\to\infty$,
$$
	E_{\rho_0}^{(n)}
	\Bigg[\Pi\Big(\rho : \|\rho - \rho_0\|_{L^1(\Zcal,\nu)} > M_2\nu_n
	\Big| D^{(n)}\Big)\Bigg]
	\to 0.
$$
\end{proposition}

\begin{proof}
The proof follows along the same line as the proof of Theorem \ref{Theo:GaussGene}, replacing the concentration inequality for integral functionals of stationary ergodic Gaussian process with the corresponding result for Poisson random tessellations. In particular, with $U_n := \{\rho\in\Rcal_n: \|\lambda^{(n)}_\rho - \lambda^{(n)}_{\rho_0}\|_{L^1(\Wcal_n)} \le Mn\epsilon_n,\ \| \rho\|_{L^\infty(\Zcal)} \le M_1\}$ and $V_n:=\{\rho\in \Rcal:\|\rho - \rho_0\|_{L^1(\Zcal,\nu)} \le M_2\nu_n\}$ for $M_2>0$ to be chosen below, arguing exactly as in the proof of Theorem \ref{Theo:GaussGene} yields
$$ 
	E_0^n \left[\Pi( V_n^c | D^{(n)})\right]  
	\leq e^{  K_1 n \epsilon_n^2 }  \int_{V_n^c \cap \{\rho\in\mathcal R_n:\|\rho\|_{L^\infty(\Zcal)}\le M_1\}}
	 P_{Z^{(n)}} (\Delta^{(n)}(\rho) > M_2\epsilon_n/2 )d\Pi(\rho)+ o(1),
$$ 
where $\Delta^{(n)}(\rho) := \| \rho - \rho_0\|_{L^1(\mathcal Z)} - \|\lambda^{(n)}_\rho - \lambda^{(n)}_{\rho_0}\|_{L^1(\Wcal_n)}/n$. The concentration inequality in Lemma \ref{Lem:PoissConcIneq}, applied with $f := |\rho - \rho_0| - \|\lambda^{(n)}_\rho - \lambda^{(n)}_{\rho_0}\|_{L^1(\Wcal_n)}/n$, for $\rho\in \Rcal$, satisfying $ \|  f  \|_{L^\infty(\Zcal)}\le \| \rho  \|_{L^\infty(\Zcal)}+\| \rho_0  \|_{L^\infty(\Zcal)}$ now gives that
$$
 	\sup_{\rho\in V_n^c \cap \{\rho\in\mathcal R_n:\|\rho\|_{L^\infty(\Zcal)}\le M_1\}}
	P_{Z^{(n)}} (\Delta^{(n)}(\rho) > M_2\epsilon_n/2 )\leq e^{- K_2 (M_2)^2 n\epsilon_n^2 }
$$
for some $K_2>0$. The claim then follows upon taking $M_2>0$ large enough and combining the last two displays.
\end{proof}

	Similarly to the statement of Proposition \ref{Theo:GaussGene}, the requirement \eqref{boundedrho} is formulated in terms of the posterior distribution. This assumption can naturally be verified through a suitable construction of the prior, e.g.~as in Theorem \ref{Cor:GaussPoissTess}. It also allows for some additional flexibility, along the lines discussed after the proof of Proposition \ref{Theo:GaussGene}.

%
%
%
%

\subsection{Proof of Proposition \ref{prop:finiteCov}}
\label{Sec:ProofParametric}

We directly apply Theorem \ref{Theo:GenRandCov}. Let $\epsilon_n = M \sqrt{\log n/n}$.  Since $\min_{k=1,\dots,K}\rho_0(k)>0$, the set $B_{n,2}(\rho_0)$ appearing in the statement of Theorem \ref{Theo:GenRandCov} satisfies 
$$
	B_{n,2}(\rho_0)\subset \left\{\rho: \max_{k=1,\dots,K}|\rho(z_k) 
	- \rho_0(z_k)| \le 1/\sqrt{n} \right\}
$$ for $n$ large enough. Also, $\Pi(B_{n,2}(\rho_0)) \geq e^{ -C_1 \log n}$ for some $C_1>0$. Let $\mathcal R_n := \{ \rho: \max_{k=1,\dots,K}|\rho(z_k)| \le n^{R_0}\}$ for $R_0>0$ to be chosen below. Then the prior condition \eqref{Eq:DiscrPriorCond} implies that 
$$
	\Pi( \mathcal R_n^c ) \leq K n^{-b R_0 } \leq e^{- C_2 \log n}
$$ 
for any $C_2 >0$, provided that $R_0$ is large enough. Finally 
$$
	\log \Ncal(\epsilon_n;\Rcal_n,\|\cdot\|_{L^\infty(\Zcal)})
	\le \log ( ( n^{R_0+1/2})^K) \leq K (R_0+1/2) \log n \leq n\epsilon_n^2
$$ 
provided that $M$ is large enough. Hence, an application of Theorem \ref{Theo:GenRandCov} gives
$$ 
	E_{\rho_0}^{(n)}
	\Bigg[\Pi\Bigg( \sum_{k=1}^K  | \rho(z_k) - \rho_0(z_k) | \mu_n(z_k)
	  > \epsilon_n
	\Bigg| D^{(n)}\Bigg)\Bigg] \to 0,
	\qquad \mu_n(z_k):=\frac{1}{n} \int_{\Wcal_n} 1_{ \{Z(x) = z_k \}}. 
$$

	Now define the ($Z^{(n)}$-dependent) event $\Acal_n = \{ \mu_n(z_k) \geq \nu(z_k)/2,\ k=1,\dots,K \}$.
	Since $Z$ is ergodic, for any measurable $A \subset \mathcal Z$, 
$\mu_n( A	) \rightarrow \nu(A)$ almost surely. 
Therefore $P_Z(\Acal_n)\to 1$. The claim then follows since on $\Acal_n$, if 
$\sum_{k=1}^K  | \rho(z_k) - \rho_0(z_k) |\mu_{n}(z_k)  \leq \epsilon_n$, then 
$$
	\sum_{k=1}^K  | \rho(z_k) - \rho_0(z_k) | \nu (z_k) \leq 2 \epsilon_n.
$$
\qed

%
%
%
%
%

\section{Concentration inequalities for spatial averages of ergodic processes}
\label{Sec:ConcIneq}

In this section we provide tools to uniformly control spatial averages (i.e., scaled integral functionals) of the stationary ergodic random fields considered in Sections \ref{Subsec:GaussRates} and  \ref{Subsec:PoissRates}.

%
%
%

\subsection{Concentration inequalities for multivariate Gaussian random fields}
\label{Subsec:MultGPCov}

\subsubsection{A sub-Gaussian concentration inequality for spatial averages}\label{subsec:concineqGaussian}

Consider a covariate process $Z=(Z(x),\ x\in\R^D)$ with values in $\Zcal\subseteq\R^d$ arising as described in Condition \ref{Cond:GenGaussCov} for some continuously differentiable map $\Phi:=(\Phi^{(1)},\dots,\Phi^{(d)}):\R^d\to\Zcal$ with uniformly bounded partial derivatives. Let $\textnormal{J}\Phi:=[\partial_{h'} \Phi^{(h)}]_{h,h'=1}^d\in L^\infty(\R^D;\R^{d,d})$ denote the Jacobian matrix associated to $\Phi$.

	For $\nu $ the stationary distribution of $Z$, denote the space of $\nu$-centred functions by
$$
	L^1_\nu(\Zcal) := \left\{ f \in L^1(\Zcal,\nu) : \int_{\Zcal} f(z)d\nu(z) = 0\right\},
$$
and, for a class of functions $\Fcal_n\subseteq L^1_\nu(\Zcal)$ and a sequence of measurable sets $\Wcal_n\subset \R^D$ satisfying $\vol(\Wcal_n)=n$ and the shape-regularity condition \eqref{Eq:SpatialWn}, consider the empirical process
\begin{align}
\label{Eq:EmpProc}
	X^{(n)}_f[Z]
	:=\frac{1}{n}\int_{\Wcal_n}f(Z(x))dx, 
	\qquad f\in\Fcal_n.
\end{align}
Note that by Fubini's theorem,
$$
	\textnormal{E}[X_f^{(n)}[Z] ]
	= \frac{1}{n}\int_{\Wcal_n}\textnormal{E}[\left[f(Z(x))\right] dx
	=\frac{1}{n}\int_{\Wcal_n}\left( \int_\Zcal f(z)d\nu(z) \right) dx = 0.
$$

\begin{proposition}\label{Prop:MultGaussConcIneq}
Let $Z$ be a stationary random field satisfying Condition \ref{Cond:GenGaussCov}. Then, for all $f\in W^{1,\infty}(\Zcal)\cap L^1_\nu(\Zcal)$, all $r>0$, and all $n\in\N$,
$$
	\Pr\left( \left| \frac{1}{n}\int_{\Wcal_n}f(Z(x))dx  \right| \ge  r \right)
	\le
	 4\exp\left\{-\frac{nr^2}
	{4d^3eC_{BL}C_KC_\Phi
	\|\nabla f\|_{L^\infty(\Zcal;\R^d)} ^2}
	 \right\},
$$
where $C_{BL}>0$ is the numerical constant appearing in the statement of Lemma \ref{Lem:MultFuncIneq} below, $C_K := \max_{h=1,\dots,d}\|K^{(h)}\|_{L^1(\R^D)}$ and $C_\Phi:=\|\textnormal{J}\Phi\|_{L^\infty(\R^d,\R^{d,d})}$.
\end{proposition}

\begin{proof}
Combining Lemma \ref{Lem:HigherMom} and \ref{Lem:PartDer} below gives the following bounds for the moments of the centred random variable $\frac{1}{n}\int_{\Wcal_n}f(Z(x))dx=X^{(n)}_f[Z]=X^{(n)}_f[\Phi\circ \tilde Z]$: for all $1\le p<\infty$,
\begin{align*}
	\textnormal{E}&[(X^{(n)}_f[Z])^{2p}]^{\frac{1}{p}}\\
	&\le  2pC_{BL}\sum_{h=1}^d
	 E\Bigg[\Bigg(\int_{\R^D}\int_{\R^D}
	 K^{(h)}(x-x')
	 \frac{1}{n}1_{\Wcal_n}(x)\Bigg|
	 \sum_{h'=1}^d\partial_{h'} f\left(\Phi\left(\tilde Z(x)\right)\right)
	 \partial_h\Phi^{(h')}\left(\tilde Z(x)\right)\Bigg|\\
	&\quad\times \frac{1}{n}1_{\Wcal_n}(x')\Bigg|
	 \sum_{h''=1}^d\partial_{h''} f\left(\Phi\left(\tilde Z(x')\right)\right)
	 \partial_h\Phi^{(h'')}\left(\tilde Z(x')\right)\Bigg|  dxdx'\Bigg)^p\Bigg]^\frac{1}{p}\\
	&\le 
	 \frac{2pd^2C_{BL}\|\nabla f\|^2_{L^\infty(\Zcal;\R^d)}
	 \|\textnormal{J}\Phi\|_{L^\infty(\R^d;\R^{d,d})}^2}{n^2}\sum_{h=1}^d
	 \int_{\Wcal_n}\int_{\Wcal_n}K^{(h)}(x-x')dxdx' \\
	&\le
	 \frac{2pd^2C_{BL}\|\nabla f\|^2_{L^\infty(\Zcal;\R^d)}
	 \|\textnormal{J}\Phi\|_{L^\infty(\R^d;\R^{d,d})}^2}{n^2}\sum_{h=1}^d
	 \|K^{(h)}\|_{L^1(\R^D)}\int_{\Wcal_n}dx'\\
	 &\le
	  \frac{2pd^3C_{BL}\|\nabla f\|^2_{L^\infty(\Zcal;\R^d)}
	 \|\textnormal{J}\Phi\|_{L^\infty(\R^d;\R^{d,d})}^2\max_{h=1,\dots,d}
	 \|K^{(h)}\|_{L^1(\R^D)}}{n}.
\end{align*}

	We proceed deriving an exponential moment bound. Using the previous display,
\begin{align*}
	\textnormal{E}&\left[\exp\left\{\frac{n}{4d^3eC_{BL}C_KC_\Phi\|\nabla f\|_{L^\infty(\Zcal;\R^d)} ^2}
	(X^{(n)}_f[Z])^2\right\} \right]\\
	&\qquad=\sum_{p=0}^\infty \frac{n^p}{(4d^3eC_{BL}C_KC_\Phi\|\nabla f\|_{L^\infty(\Zcal;\R^d)} ^2)^pp!} 
	E\left[ (X^{(n)}_f[Z])^{2p}\right]
	\le \sum_{p=0}^\infty \frac{p^p}{p!(2e)^p}.
\end{align*}
By Stirling's approximation, $p! >  \sqrt{2\pi p}(p/e)^pe^{1/(12 \log p + 1)}>\sqrt{2\pi p}(p/e)^p$, so that the latter series is upper bounded by
\begin{align*}
	\sum_{p=0}^\infty \frac{1}{\sqrt{2\pi p} 2^p}
	\le \sum_{p=0}^\infty \frac{1}{ 2^p} = 2.
\end{align*}
Conclude by Markov's inequality that, for all $r\ge 0$,
\begin{align*}
	&\Pr\left (X^{(n)}_f[Z] > r \right)\\
	&\ \le \textnormal{E}\left[\exp\left\{\frac{n}
	{4d^3eC_{BL}C_KC_\Phi
	\|\nabla f\|_{L^\infty(\Zcal;\R^d)} ^2}
	(X^{(n)}_f[Z])^2\right\} \right]
	 \exp\left\{-\frac{n}
	{4d^3eC_{BL}C_KC_\Phi
	\|\nabla f\|_{L^\infty(\Zcal;\R^d)} ^2}r^2\right\}\\
	&\ \le 2\exp\left\{-\frac{n}
	{4d^3eC_{BL}C_KC_\Phi
	\|\nabla f\|_{L^\infty(\Zcal;\R^d)} ^2}r^2\right\}.
\end{align*}
By similar computations, it also holds that 
\begin{align*}
	\Pr\left (X^{(n)}_f[Z] < -r \right)
	\le 2\exp\left\{-\frac{n}
	{4d^3eC_{BL}C_KC_\Phi
	\|\nabla f\|_{L^\infty(\Zcal;\R^d)} ^2}r^2\right\},
\end{align*}
which, combined with the previous display via a union bound, proves the claim.
\end{proof}

	The following lemma is the key technical tool for the proof of Lemma \ref{Prop:MultGaussConcIneq}. It provides certain Poincaré- and log-Sobolev-type inequalities for random variables arising as transformations $X[\tilde Z]$ of the multivariate Gaussian random field $\tilde Z$ introduced in Condition \ref{Cond:GenGaussCov} via measurable functionals $X: L^\infty_{\textnormal{loc}}(\R^D;\R^d)\to\R$. The result represents a multi-dimensional extension of Theorem 3.1 (i) in \cite{DG20AHL}. The inequalities are stated in terms of the partial Gateaux-derivatives $\partial_h X, \ h=1,\dots,d,$ of $X$, that is functionals $\partial_h X: L^\infty_{\textnormal{loc}}(\R^D;\R^d)\to L^1_{\textnormal{loc}}(\R^D)$ such that, for all compactly supported $\zeta \in L^\infty(\R^D)$,
 $$
 	\lim_{t\to0}\frac{X(\tilde z^{(1)},\dots,\tilde z^{(h)}+t\zeta,\dots,\tilde z^{(d)})}{t}
	=\int_{\R^D}\zeta(x)\partial_hX[\tilde z](x)dx,
	\ \ \ \forall \tilde z=(\tilde z^{(1)},\dots,\tilde z^{(d)})\in L^\infty_{\textnormal{loc}}(\R^D;\R^d).
 $$

\begin{lemma}\label{Lem:MultFuncIneq}
Let $\tilde Z$ be the $d$-variate Gaussian random field introduced in Condition \ref{Cond:GenGaussCov}. Then, the following Poincaré- and logarithmic Sobolev-type inequalities hold: for all measurable $X: L^\infty_{\textnormal{loc}}(\R^D;\R^d)\to\R$ for which the partial Gateaux-derivatives $\partial_hX :   L^\infty_{\textnormal{loc}}(\R^D;\R^d)\to L^1_{\textnormal{loc}}(\R^D)$ exist for all $h=1,\dots,d$, the random variable $X[\tilde Z]$ satisfies
\begin{align*}
	\textnormal{Var}[X[\tilde Z]]
	\le C_{BL} \sum_{h=1}^d \textnormal{E}\left[\int_{\R^D}\int_{\R^D}K^{(h)}(x-x')
	| \partial_h X[\tilde Z](x) |  |  \partial_h X[\tilde Z](x') | dxdx'\right],
\end{align*}
as well as
\begin{align*}
	\textnormal{Ent}[X[\tilde Z]^2] 
	\le C_{BL} \sum_{h=1}^d E\left[\int_{\R^D}\int_{\R^D}K^{(h)}(x-x')
	|\partial_h X[\tilde Z](x) |  | \partial_h X[\tilde Z](x') | dxdx'\right],
\end{align*}
where $\textnormal{Ent}[X[\tilde Z]^2] := \textnormal{E}\left[X[\tilde Z]^2\log\frac{X[\tilde Z]^2}{E[X[\tilde Z]^2]} \right]$ and $C_{BL}>0$ is a numerical constant.
\end{lemma}

\begin{proof}
We follow the proof of Theorem 3.1 (i) in \cite{DG20AHL}. The starting point is an application of the discrete Brascamp-Lieb inequality (e.g.~\cite{H08}): let $(W^{(h)}_k,\ h=1,\dots,d, \ k=1,\dots,M_h)$, be $M:=M_1+\dots+M_d$ independent standard normal random variables. Set $W^{(h)} := (W^{(h)}_1,\dots,W^{(h)}_{M_h})^T\sim N_{M_h}(0,I_{M_h})$ and $W := (W^{(1)T},\dots,W^{(d)T})^T\sim N_M(0,I_M).$
Consider matrices $F^{(h)}=[F^{(h)}_{kl}]_{k,l=1}^{M_h}\in \R^{M_h,M_h}$, and let $F\in \R^{M,M}$ be the block-diagonal matrix
$$
	F:=
	\begin{bmatrix}
	F^{(1)}  & \cdots  &\cdots    & \cdots  & 0 \\
	\vdots   & \ddots  &             &             & \vdots \\
	0          &  \cdots  & F^{(h)} & \cdots  & 0 \\
	\vdots  &              &             & \ddots  & \vdots  \\
	0         & \cdots    & \cdots  & \cdots  & F^{(d)}
	\end{bmatrix}.
$$
Consider the random vector $\tilde Z := FW \sim N_M(0,FF^T)$, which, due to the block-diagonal structure takes the form $\tilde Z=(\tilde Z^{(1)},\dots,\tilde Z^{(d)})$ with $Z^{(h)}:=F^{(h)}W^{(h)}\sim N_{M_h}(0,F^{(h)}F^{(h)T})$. Finally, for any differentiable function $X:\R^M\to\R$, consider the composition $X\circ F:\R^M\to \R$, associating to any $w=(w^{(1)}_1,\dots,w^{(1)}_{M_1},\dots,w^{(d)}_1,\dots,w^{(d)}_{M_d})^T\in\R^M$ the value
$$
	X\circ F(w) 
	=
	X\left(\sum_{l=1}^{M_1} F^{(1)}_{1l}w^{(1)}_l,\dots,
	\sum_{l=1}^{M_h} F^{(h)}_{1l}w^{(h)}_l,
	\dots,\sum_{l=1}^{M_h} F^{(h)}_{M_hl}w^{(h)}_l,
	\dots,\sum_{l=1}^{M_d} F^{(d)}_{M_dl}w^{(d)}_l\right).
$$
Then, by the Brascamp-Lieb inequality for standard Gaussian random vectors (e.g.~\cite{H08}), for a numerical constant $C_{BL}>0$,
\begin{align*}
	\max\{\textnormal{Var}[X(\tilde Z)],\textnormal{Ent}[X(\tilde Z)^2]\}
 	&\ \le C_{BL} \sum_{h=1}^d\sum_{k=1}^{M_h}
	\textnormal{E}\left[ \left(\frac{\partial (X\circ F)}{\partial w^{(h)}_k}(W) \right)^2 \right]\\
	&=C_{BL} \sum_{h=1}^d\textnormal{E}\left[\sum_{k=1}^{M_h} \left(\sum_{l=1}^{M_h}
	\partial_{M_1+\dots+M_{h-1}+l} X(\tilde Z)F_{lk} \right)^2 \right],
\end{align*}
where $\partial_{M_1+\dots+M_{h-1}+l} X:\R^M\to\R$ is the partial derivative of the function $X$ with respect to its $(M_1+\dots+M_{h-1}+l)^{\textnormal{th}}$ argument, with the sum $M_1+\dots+M_{h-1}$ being set equal to $0$ by convention if $h=1$. Denoting by
$$
	\nabla^{(h)}X
	:=\Big(\partial_{M_1+\dots+M_{h-1}+1} X,\dots,
	\partial_{M_1+\dots+M_{h-1}+M_h} X\Big)^T
	:\R^M \to R^{M_h},
$$
the expectation in the second to last display equals
\begin{align*}
	 \textnormal{E}&\left[ (\nabla^{(h)}X(\tilde Z))^T F^{(h)}F^{(h)T} \nabla^{(h)}X(\tilde Z) \right]\\
	 &\quad\le \sum_{k,l=1}^{M_h}
	 |(F^{(h)}F^{(h)T})_{kl}| \textnormal{E}\left[| \partial_{M_1+\dots+M_{h-1}+k}X(\tilde Z) | 
	 | \partial_{M_1+\dots+M_{h-1}+l}X(\tilde Z) |\right],
\end{align*}
which implies, recalling that $F^{(h)}F^{(h)T}$ is the covariance matrix of $\tilde Z^{(h)}=F^{(h)}W^{(h)}$, the inequality
\begin{equation}
\label{Eq:BL2}
\begin{split}
	\max&\{\textnormal{Var}[X(\tilde Z)],\textnormal{Ent}[X(\tilde Z)^2]\}\\
	&\le
	 C_{BL}\sum_{h=1}^d\sum_{k,l=1}^{M_h}
	|\textnormal{Cov}(\tilde Z^{(h)})_{kl}| \textnormal{E}\left[| \partial_{M_1+\dots+M_{h-1}+k}X(\tilde Z) | 
	| \partial_{M_1+\dots+M_{h-1}+l}X(\tilde Z) |\right].
\end{split}
\end{equation}

	We now extend the Brascamp-Lieb inequality \eqref{Eq:BL2} to the continuous setting. Let $\tilde Z$ be as in the statement of Lemma \ref{Lem:MultFuncIneq}. As in the proof of Theorem 3.1 (i) in \cite{DG20AHL}, we first consider functionals $X: L^\infty_{\textnormal{loc}}(\R^D;\R^d)\to\R$ that depend on their argument $\tilde z\in L^\infty_{\textnormal{loc}}(\R^D;\R^d)$ only through the spatial average of $\tilde z$ on the partition $\{Q_\varepsilon(y)\}_{y\in B_R\cap \varepsilon\Z^D}$ for some $\varepsilon,R>0$, where $Q_\varepsilon(y) := y+\varepsilon[-1/2,1/2)^D$ and $B_R:=\{y\in\R^D:|y|\le1\}$. That is, letting for any $\tilde z\in  L^\infty_{\textnormal{loc}}(\R^D;\R^d)$, any $y\in B_R\cap \varepsilon\Z^D$,
$$
	\tilde z_\varepsilon(y) :=(\tilde z^{(1)}_\varepsilon(y),\dots,\tilde z^{(d)}_\varepsilon(y))^T \in\R^d,
	\qquad \tilde z^{(h)}_\varepsilon(y):=\frac{1}{\varepsilon^D}\int_{Q_\varepsilon(y)}\tilde z^{(h)}(x)dx,
$$
we have $X[\tilde z]= X[\tilde z']$ whenever the associated collections of spatial averages $(\tilde z_\varepsilon(y))_{y\in B_R\cap \varepsilon\Z^D}$ and $(\tilde z'_\varepsilon(y))_{y\in B_R\cap \varepsilon\Z^D}$ coincide. In slight abuse of notation, write $X[\tilde z]= X( (\tilde z_\varepsilon(y))_{y\in B_R\cap \varepsilon\Z^D})$. Since, by assumption, $\tilde Z^{(h)}$, $h=1,\dots,d$, are independent centred stationary Gaussian processes, by construction the associated spatial averages $(\tilde Z^{(h)}_\varepsilon(y))_{y\in B_R\cap \varepsilon\Z^D}$, $h=1,\dots,d$, are independent (finite-dimensional) centred Gaussian random vector with covariance matrices $C^{(h)} := [C^{(h)}_{yy'}]_{y,y'\in B_R\cap \varepsilon\Z^D}$ given by
$$
	 C^{(h)}_{yy'}:= \textnormal{Cov}(\tilde Z^{(h)}_\varepsilon(y),\tilde Z^{(h)}_\varepsilon(y'))
	 =\frac{1}{\varepsilon^{2D}}\int_{Q_\varepsilon(y)}\int_{Q_\varepsilon(y')}
	 K^{(h)}(x-x')dxdx'.
$$
By the inequality \eqref{Eq:BL2}, it follows that
\begin{equation}
\label{Eq:ContBL}
\begin{split}
	&\max\{\textnormal{Var}[X[\tilde Z]],\textnormal{Ent}[X[\tilde Z]^2]\}\\
	&\ \ \le  C_{BL}\sum_{h=1}^d\sum_{y,y'\in B_R\cap \varepsilon\Z^D} |C^{(h)}_{yy'}|
	\textnormal{E}\left[ \left| \frac{\partial X }{\partial \tilde z_\varepsilon^{(h)}(y)}[\tilde Z]\right|  
	\left| \frac{\partial X }{\partial \tilde z_\varepsilon^{(h)}(y')}[\tilde Z]\right| \right]\\
	&\ \ =\frac{C_{BL}}{\varepsilon^{2D}} \sum_{h=1}^d 
	\sum_{y,y'\in B_R\cap \varepsilon\Z^D}
	 \int_{Q_\varepsilon(y)}\int_{Q_\varepsilon(y')}
	 K^{(h)}(x-x')
	\textnormal{E}\left[ \left| \frac{\partial X }{\partial \tilde z_\varepsilon^{(h)}(y)}[Z]\right|  
	\left| \frac{\partial X }{\partial \tilde z_\varepsilon^{(h)}(y')}[Z]\right| \right]dxdx'.
\end{split}
\end{equation}
We conclude the current step noting that for all $\tilde z\in  L^\infty_{\textnormal{loc}}(\R^D;\R^d)$, 
$$
	\partial_h X[\tilde z](\cdot)
	=\sum_{y\in B_R\cap \varepsilon\Z^D}
	\varepsilon^{-D}\frac{\partial X}{\partial \tilde z^{(h)}_\varepsilon(y)}[\tilde  z]
	1_{Q_\varepsilon(y)}(\cdot)\in L^1_{\textnormal{loc}}(\R^D),
	\qquad h=1,\dots,d.
$$
Indeed, for all compactly supported $\zeta\in L^\infty(\R^d)$,
\begin{align*}
	&\lim_{t\to0}\frac{X[\tilde z^{(1)},\dots,\tilde z^{(h)}+t\zeta,\dots,\tilde z^{(d)}] - X[\tilde z]}{t}\\
	&\ =
	\lim_{t\to0}\frac{X\left( (\tilde z^{(1)}_\varepsilon(y))_{y\in B_R\cap \varepsilon\Z^D},
	\dots,(\tilde z^{(h)}_\varepsilon(y)+ t\zeta_\varepsilon(y))_{y\in B_R\cap \varepsilon\Z^D},
	\dots,(\tilde z^{(d)}_\varepsilon(y))_{y\in B_R\cap \varepsilon\Z^D} \right) - X[\tilde z]}{t}\\
	&\ =
	\sum_{y\in B_R\cap \varepsilon\Z^D}\frac{\partial \tilde  X}{\partial \tilde z^{(h)}_\varepsilon(y)}[\tilde z]
	\zeta_\varepsilon(y)
	 =\int_{\R^D} \sum_{y\in B_R\cap \varepsilon\Z^D}\varepsilon^{-D}
	 \frac{\partial X}{\partial \tilde z^{(h)}_\varepsilon(y)}[\tilde z]1_{Q_\varepsilon(y)}(x)\zeta(x)dx.
\end{align*}
In particular,
$$
	\varepsilon^{-D} \frac{\partial X}{\partial \tilde z^{(h)}_\varepsilon(y)}[\tilde z] 
	 = \partial_h X[\tilde z](x),
	 \qquad \forall \ x\in Q_{\varepsilon}(y), \ y\in B_R\cap \varepsilon\Z^D.
$$
Combined with \eqref{Eq:ContBL}, this yields
\begin{align*}
	&\max\{\textnormal{Var}[X[\tilde Z]],\textnormal{Ent}[X[\tilde Z]^2]\}\\
	&\quad \le C_{BL}\sum_{h=1}^d 
	\sum_{y,y'\in B_R\cap \varepsilon\Z^D}
	 \int_{Q_\varepsilon(y)}\int_{Q_\varepsilon(y')}
	 K^{(h)}(x-x')\textnormal{E}\left[  | \partial_h X[\tilde Z](x) |  |  \partial_h X[\tilde Z](x') | \right]dxdx'\\
	 &\quad=C_{BL}\sum_{h=1}^d
	 \textnormal{E}\left[\int_{\R^D}\int_{\R^D}
	 K^{(h)}(x-x')
	 | \partial_h X[\tilde Z](x) |  |  \partial_h X[\tilde Z](x') | dxdx'\right].
\end{align*}

	For general functionals $X: L^\infty_{\textnormal{loc}}(\R^D;\R^d)\to\R$ as in the statement of Lemma \ref{Lem:MultFuncIneq}, the proof then follows via the same approximation argument as in the conclusion of the proof of Theorem 3.1 in \cite{DG20AHL}, approximating $\tilde Z$ by the collection $(\tilde Z_\varepsilon(y))_{y\in B_R\cap \varepsilon\Z^D}$, and letting $\varepsilon\to0$ and $R\to\infty$.
\end{proof}

	Leveraging the Poincaré- and log-Sobolev-type inequalities derived in Lemma \ref{Lem:MultFuncIneq}, we obtain in the next lemma bounds for the higher-order moments of functionals of the multivariate Gaussian random field $\tilde Z$. These follow from recasting the bounds in Proposition 1.10 (i) in \cite{DG20ALEA} in the present multivariate setting with integrable covariances.

\begin{lemma}\label{Lem:HigherMom}
Let $\tilde Z$ be the $d$-variate Gaussian random field introduced in Condition \ref{Cond:GenGaussCov}. Then, for all measurable $X: L^\infty_{\textnormal{loc}}(\R^D;\R^d)\to\R$ for which the partial Gateaux-derivatives $\partial_hX :   L^\infty_{\textnormal{loc}}(\R^D;\R^d)\to L^1_{\textnormal{loc}}(\R^D)$ exist for all $h=1,\dots,d$, the random variable $X[\tilde Z]$ satisfies for all $1\le p<\infty$,
\begin{align*}
	\textnormal{E}[(X[\tilde Z] &- \textnormal{E}[X[\tilde Z]])^{2p}]^\frac{1}{p}\\
	&\le  2pC_{BL} \sum_{h=1}^d
	\textnormal{E}\left[\left(\int_{\R^D}\int_{\R^D}
	 K^{(h)}(x-x') | \partial_h X[\tilde Z](x)|  | \partial_h X[\tilde Z] (x')|  dxdx'\right)^p\right]^\frac{1}{p},
\end{align*}
where $C_{BL}>0$ is the numerical constant appearing in the statement of Lemma \ref{Lem:MultFuncIneq}.
\end{lemma}

\begin{proof}
Without loss of generality, assume $\textnormal{E}[X[\tilde Z]]=0$. We follow the proof of Proposition 1.10 (i) in \cite{DG20ALEA}, using the fact that
\begin{align}
\label{Eq:Start2}
	\textnormal{E}[X[\tilde Z]^{2p}]^{\frac{1}{p}} - \textnormal{E}[X[\tilde Z]^2] = \int_1^p\frac{1}{q^2}
	\textnormal{E}[X[\tilde Z]^{2q}]^{\frac{1}{q}-1}\textnormal{Ent}[X[\tilde Z]^{2q}]dq,
\end{align}
cf.~\cite[p.254]{BGL14}. We estimate $\textnormal{Ent}[X[\tilde Z]^{2q}]$ for all $1\le q\le p$. Applying Lemma \ref{Lem:MultFuncIneq} to the random variable $|X[\tilde Z]|^q$ yields
$$
	\textnormal{Ent}[X[\tilde Z]^{2q}]\le
	   C_{BL}\sum_{h=1}^d
	 \textnormal{E}\left[\int_{\R^D}\int_{\R^D}
	 K^{(h)}(x-x')| \partial_h |X[\tilde Z]|^q(x)|  | \partial_h |X[\tilde Z] |^q (x')| dxdx'\right].
$$
By the chain rule,
$$
	 |\partial_h |X[\tilde Z]|^q| = q |X[\tilde Z]|^{q-1} |\partial_h |X[\tilde Z]|| 
	 = q |X[\tilde Z]|^{q-1} |\partial_h X[\tilde Z]|,
$$
so that by H\"older's inequality with exponents $(q/(q-1),q)$, 
\begin{align*}
	\textnormal{E}&\left[\int_{\R^D}\int_{\R^D}
	 K^{(h)}(x-x')
	 | \partial_h |X[\tilde Z]|^q(x)|  | \partial_h |X[\tilde Z] |^q (x')|  dxdx'\right]\\
	 &=\textnormal{E}\left[q^2|X[\tilde Z]|^{2(q-1)}\int_{\R^D}\int_{\R^D}
	 K^{(h)}(x-x')
	 | \partial_h X[\tilde Z](x)|  | \partial_h X[\tilde Z] (x')|  dxdx'\right]\\
	 &\le q^2\textnormal{E}\left[X[\tilde Z]^{2q}\right]^{1-\frac{1}{q}}
	 \textnormal{E}\left[\left(\int_{\R^D}\int_{\R^D}
	 K^{(h)}(x-x')| \partial_h X[\tilde Z](x)|  | \partial_h X[\tilde Z] (x')|  dxdx'\right)^q\right]^\frac{1}{q},
\end{align*}
implying that
\begin{align*}
	&\textnormal{Ent}[X[\tilde Z]^{2q}]\\
	&\le
	q^2C_{BL}\textnormal{E}\left[X[\tilde Z]^{2q}\right]^{1-\frac{1}{q}}
	\sum_{h=1}^d E\left[\left(\int_{\R^D}\int_{\R^D}
	K^{(h)}(x-x')| \partial_h X[\tilde Z](x)|  | \partial_h X[\tilde Z] (x')|  dxdx'\right)^q\right]^\frac{1}{q}.
\end{align*}
Replaced into \eqref{Eq:Start2}, this gives
\begin{align*}
	\textnormal{E}&[X[\tilde Z]^{2p}]^{\frac{1}{p}} \\
	&\le  \textnormal{E}[X[\tilde Z]^2] +  C_{BL}\sum_{h=1}^d\int_1^p
	\textnormal{E}\left[\left(\int_{\R^D}\int_{\R^D}K^{(h)}(x-x') 
	| \partial_h X[\tilde Z](x)|  | \partial_h X[\tilde Z] (x')|  dxdx'\right)^q\right]^\frac{1}{q}dq.
\end{align*}
We then use Lemma \ref{Lem:MultFuncIneq} and Jensen's inequality to bound $\textnormal{E}[X[\tilde Z]^2] = \textnormal{Var}[X[\tilde Z]]$ by
\begin{align*}
	 C_{BL}&\sum_{h=1}^d\textnormal{E}\left[\int_{\R^D}\int_{\R^D}
	 K^{(h)}(x-x')
	 | \partial_h X[\tilde Z](x) |  |  \partial_h X[\tilde Z](x') | dxdx'\right]\\
	 &\le C_{BL}\sum_{h=1}^d\textnormal{E}\left[\left(\int_{\R^D}\int_{\R^D}
	 K^{(h)}(x-x')
	 | \partial_h X[\tilde Z](x) |  |  \partial_h X[\tilde Z](x') | \right)^p dxdx'\right]^\frac{1}{p}.
\end{align*}
Similarly, for each $1\le q\le p$, $h=1,\dots,d$,
\begin{align*}
	\textnormal{E}&\left[\left(\int_{\R^D}\int_{\R^D}
	 K^{(h)}(x-x') | \partial_h X[\tilde Z](x)|  | \partial_h X[\tilde Z] (x')|  dxdx'\right)^q\right]^\frac{1}{q}\\
	 &\le
	 \textnormal{E}\left[\left(\int_{\R^D}\int_{\R^D}
	 K^{(h)}(x-x') | \partial_h X[\tilde Z](x)|  | \partial_h X[\tilde Z] (x')|  dxdx'\right)^p\right]^\frac{1}{p}.
\end{align*}
Combining the last three displays concludes the proof. 
\end{proof}

	The next lemma provides the partial Gateaux-derivatives of the functionals $X^{(n)}_f[\cdot]$ defining the empirical process \eqref{Eq:EmpProc}.

\begin{lemma}\label{Lem:PartDer}
Let $f\in W^{1,\infty}(\Zcal)$. Then, for all $\tilde z\in  L^\infty_{\textnormal{loc}}(\R^D;\R^d)$,
$$
	\partial_hX_f^{(n)}[\Phi \circ \tilde z](\cdot) = \frac{1}{n}1_{\Wcal_n}(\cdot)
	 \sum_{h'=1}^d\partial_{h'} f\left(\Phi\left(\tilde z(\cdot)\right)\right)\partial_h\Phi^{(h')}\left(\tilde z(\cdot)\right)
	\in L^1_{\textnormal{loc}}(\R^D).
$$
\end{lemma}

\begin{proof}
For any $h=1,\dots, d$, all compactly supported $\zeta\in L^\infty(\R^D)$, all $x\in \R^D$, and arbitrarily small $t_0>0$, the function
$$
	g_x : (-t_0,t_0)\to\R, \qquad g_x(t) 
	:=f\left(\Phi\left(\tilde z^{(1)}(x),
	\dots,\tilde z^{(h)}(x) + t\zeta(x),\dots,\tilde z^{(d)}(x)\right)\right),
$$
is in $W^{1,\infty}(-t_0,t_0)$, with (weak) derivative 
\begin{align*}
	g_x'(t) &= \sum_{h'=1}^d\partial_{h'} f\left(\Phi\left(\tilde z^{(1)}(x),
	\dots,\tilde z^{(h)}(x) + t\zeta(x),\dots,\tilde z^{(d)}(x)\right)\right)
	\\
	&\quad \times\partial_h\Phi^{(h')}\left(\tilde z^{(1)}(x),
	\dots,\tilde z^{(h)}(x) + t\zeta(x),\dots,\tilde z^{(d)}(x)\right)
	\zeta(x).
\end{align*}
Using this, the function
$$
	g:(-t_0,t_0) \to \R, 
	\qquad g(t) := X_f^{(n)}[z^{(1)},\dots, z^{(h)} + t\zeta,\dots,z^{(d)}]
	= \frac{1}{n}\int_{\Wcal_n} g_x(t)dx,
$$
is seen to be in $W^{1,\infty}(-t_0,t_0)$, with (weak) derivative
$$
	g'(t)
	= \frac{1}{n}\int_{\Wcal_n} g_x(t)dx.
$$
Thus,
\begin{align*}
	\lim_{t\to0}&
	\frac{X_f^{(n)}(z^{(1)},\dots, z^{(h)} + t\zeta,\dots,z^{(d)}) - X_f^{(n)}(z)}{t} \\
	& 
	= g'(0)
	=
	\int_{\R^D} \frac{1}{n}1_{\Wcal_n}(x) \sum_{h'=1}^d\partial_{h'} f\left(\Phi\left(\tilde z(x)\right)\right)\partial_h\Phi^{(h')}\left(\tilde z(x)\right)
	\zeta(x)
	dx,
	\end{align*}
whence the claim follows.
\end{proof}
	
%
%
%

\subsubsection{Inequalities for the suprema of spatial averages}
\label{Sec:EmpProcSup}

We now build on the sub-Gaussian concentration inequality provided in Proposition \ref{Prop:MultGaussConcIneq} to derive inequalities for the supremum of the empirical process defined in \eqref{Eq:EmpProc} over classes of functions $\Fcal_n\subseteq W^{1,\infty}(\Zcal^d)\cap L^1_\nu(\Zcal^d)$.

\begin{proposition}\label{Prop:MultGaussSupConc}
Let $Z$ be a stationary random field satisfying Condition \ref{Cond:GenGaussCov}. Let $\Fcal_n\subseteq W^{1,\infty}(\Zcal)\cap L^1_\nu(\Zcal^d)$ with $0\in \Fcal_n$. Then, for all $n\in\N$,
\begin{align}
\label{Eq:ExpectBound}
	\textnormal{E}\left[\sup_{f\in\Fcal_n} \left| \frac{1}{n}\int_{\Wcal_n}f(Z(x))dx\right| \right] 
	\le \frac{4\sqrt 2J_{\Fcal_n}}{\sqrt n},
\end{align}
where
$$
	J_{\Fcal_n}
	:=\int_0^{D_{\Fcal_n}}\sqrt{\log 2\Ncal(\tau;\Fcal_n,6C_{\Phi,K}
	\|\nabla[\cdot]\|_{L^\infty(\Zcal;\R^d)} ) }
	d\tau,
$$
with $D_{\Fcal_n}$ the diameter of $\Fcal_n$ with respect to the semi-metric $C_{\Phi,K}\|\nabla[\cdot]\|_{L^\infty(\Zcal;\R^d)}$, and where $C_{\Phi,K}:=\sqrt {2d^3eC_{BL}C_KC_\Phi}$ for $C_{BL}>0$ and $C_K,C_\Phi>0$ the constants appearing in the statements of Lemma \ref{Lem:MultFuncIneq} and Proposition \ref{Prop:MultGaussConcIneq}, respectively. Then, for all $r>0$ and all $n\in \N$,
\begin{equation}
\label{Eq:SuprConc}
	\Pr\left( \sup_{f\in\Fcal_n} \left|\frac{1}{n}\int_{\Wcal_n}f(Z(x))dx\right|
	\ge
	  \frac{196J_{\Fcal_n}}{\sqrt n}(1+r)  \right) \le \exp\left\{-\frac{r^2}{2}\right\}.
\end{equation}

\end{proposition}

\begin{proof}
By linearity, 
$$
	X_{f_1 - f_2}^{(n)}[Z] = \frac{1}{n}\int_{\Wcal_n}f_1(Z(x))-f_2(Z(x))dx
	 =X_{f_1}^{(n)}[Z] - X_{f_2}^{(n)}[Z].
$$
Hence, by Proposition \ref{Prop:MultGaussConcIneq}, for all $f_1,f_2\in \Fcal_n$, $f_1\neq f_2$,
\begin{align*}
	\Pr\left( \left | \sqrt nX^{(n)}_{f_1}[Z] - \sqrt nX^{(n)}_{f_2}[Z] \right| \ge r \right)
	&\le
	4\exp\left\{-\frac{r^2}{2C_{K,\Phi}^2 \|\nabla f_1 - \nabla f_2\|_{L^\infty(\Zcal;\R^d)}^2} \right\}.
\end{align*}
This shows that the centred process $\{\sqrt nX_f^{(n)}[Z], \ f\in \Fcal_n\}$ is sub-Gaussian with respect to the semi-metric $C_{K,\Phi} \|\nabla[\cdot]\|_{L^\infty(\Zcal;\R^d)} $. The chaining bound for sub-Gaussian processes (e.g., Theorem 2.3.7 of \cite{GN16}) then implies the first claim (since also $0\in\Fcal_n$).

	For the second, arguing as in the conclusion of the proof of Lemma 1 in \cite{NR20} gives, for any $r>0$,
$$
	\Pr\left( \sup_{f\in\Fcal_n} \sqrt n\left|X_f^{(n)}[Z]\right|
	\ge \textnormal{E}\left[\sup_{f\in\Fcal_n} \left|\sqrt nX_f^{(n)}[Z]\right| \right]
	+ r  \right) \le \exp\left\{-\frac{r^2}{2(196 J_{\Fcal_n})^2}\right\},
$$
whence the second claim follows using the expectation bound \eqref{Eq:ExpectBound} proved above.
\end{proof}

%

%
%
%

\subsection{Concentration inequalities for Poisson random tessellations}
\label{Subsec:PoissTessCov}

Throughout this section, let $Z$ be the univariate piecewise-constant process associated to a Poisson random tessellation arising as in Definition \ref{Cond:PoissTess}. Such random fields represent the primary example of processes satisfying `multiscale functional inequalities with oscillations' considered in \cite{DG20ALEA,DG20AHL}. For $\nu $ the stationary measure of $Z$, with support $\Zcal\subseteq\R$, recall the notation $L^1_\nu(\Zcal)=\{f\in L^1(\Zcal,\nu):\int_\Zcal f(z)d\nu(z)=0\}$. Given measurable sets $\Wcal_n\subset \R^D$ satisfying $\vol(\Wcal_n)=n$ and the shape-regularity condition \eqref{Eq:SpatialWn}, a combination of results in \cite{DG20ALEA,DG20AHL} yields the following concentration inequality for the centred spatial averages of $Z$,
$$
	X_f^{(n)}[Z] := \frac{1}{n}\int_{\Wcal_n}f(Z(x))dx, \qquad f\in L^1_\nu(\Zcal).
$$

\begin{lemma}\label{Lem:PoissConcIneq}
Let $Z$ be a stationary random field constructed as in Condition \ref{Cond:PoissTess}. Then, for each $f\in L^1_\nu(\Zcal)\cap L^\infty(\Zcal)$, all $r>0$ and all $n\in\N$,
$$
	\Pr\left(\left| \frac{1}{n} \int_{\Wcal_n} f(Z(x))dx\right| > r\right) 
	\le 2 \exp\left\{-\frac{n\min\{r,r^2\}}{1+C_Z+2\|f\|_{L^\infty(\Zcal)}} \right\},
$$
where $C_Z>0$ is a numerical constant.
\end{lemma}

\begin{proof}

Proposition 3.2 in \cite{DG20AHL} shows that $Z$ satisfies, for some constant $C_Z>0$, the following multiscale inequalities with weight function $\omega(\ell) := C_Ze^{-\frac{1}{C_Z}\ell^2}, \ \ell>0$: for all measurable $X : L^\infty(\R^D;\R)\to\R$,
$$
	\textnormal{Var}[X[Z]]
	\le \textnormal{E}\left[ \int_0^\infty \int_{\R^D}  \left(\partial^{\textnormal{osc}}_{B_{\ell+1}(x)}
	X[Z] \right)^2dx (\ell + 1)^{-2} \omega(\ell) d\ell\right]
$$
and
$$
	\textnormal{Ent}\left[ X[Z] \right] 
	\le \textnormal{E}\left[ \int_0^\infty \int_{\R^D}  \left(\partial^{\textnormal{osc}}_{B_{\ell+1}(x)}
	X[Z] \right)^2dx (\ell + 1)^{-2} \omega(\ell) d\ell\right].
$$
Above, $\partial^{\textnormal{osc}}_{B_{\ell + 1}(x)}X[Z]$ is the 'oscillation' of $X[Z]$ over $B_{\ell + 1}(x) := \{y\in \R^D  :  |y-x|\le\ell+1\}$, defined as the measurable envelope of
\begin{equation}
\label{Eq:OscDef}
\begin{split}
	\textnormal{sup}&
	\{X(Z_1), \ Z_1:\R^D\to\R\ \textnormal{measurable}: Z_1 = Z\
	 \textnormal{on}\ \R^D\backslash B_{\ell + 1}(x)\}\\
	 & \ \quad - \textnormal{inf}\{X(Z_2), \ Z_2:\R^D\to\R\ \textnormal{measurable} 
	 : Z_2 = Z\ \textnormal{on}\ \R^D\backslash B_{\ell + 1}(x)\},
\end{split}
\end{equation}
cf.~Section 1.1 in \cite{DG20ALEA}. The claim follows from a direct application of the results in Section 1.3 of \cite{DG20ALEA}. In the notation of the latter paper, write
$$
	X_f^{(n)}[Z]
	= \frac{1}{n} \int_{\Wcal_n} g_f[Z(\cdot+x)]dx,
$$
where $Z(\cdot + x) := (Z( y + x ), \ y\in \R^D)$ and $g_f[z]: = f(z(0))$, $z\in L^\infty(\R^D;\Zcal)$. Lemma \ref{Lem:1LocCond} below shows that $g_f$ satisfies the locality condition on p.138 of \cite{DG20ALEA}, in that
$$
	\underset{z\in L^\infty(\R^D;\Zcal)}{\textnormal{sup}}\ 
	\partial^{\textnormal{osc}}_{B_{\ell + 1}(x)}
	g_f[z]\le (1+C_Z+2\|f\|_{L^\infty(\Zcal)})
	e^{-\frac{1}{1+C_Z+2\|f\|_{L^\infty(\Zcal)}}(|x|-\ell)_+}
$$
for all $x\in\R^D$ and $\ell\ge1$. Recalling the shape-regularity condition \eqref{Eq:SpatialWn}, Proposition 1.7(iii) of \cite{DG20ALEA} then yields, for all $r>0$ and all $n\in\N$,
 \begin{align*}
	\Pr(X_f^{(n)}[Z]\ge r)
	\le \exp\left\{-\frac{n\min\{r,r^2\}}{1+C_Z+2\|f\|_{L^\infty(\Zcal)}}\right\}.
\end{align*}
Similarly, it also holds
 \begin{align*}
	\Pr(X_f^{(n)}[Z]\le - r)
	\le \exp\left\{-\frac{n\min\{r,r^2\}}{1+C_Z+2\|f\|_{L^\infty(\Zcal)}}\right\},
\end{align*}
which combined with the previous display proves the claim.
\end{proof}

	The next lemma provides the locality condition (cf.~p.138 of \cite{DG20ALEA}) used in the proof of Lemma \ref{Lem:PoissConcIneq}.

\begin{lemma}\label{Lem:1LocCond}
For $f\in L^\infty(\Zcal)$, let $g_f:L^\infty(\R^D;\Zcal)\to\R$ be given by
$g_f[z]=f(z(0))$. Then, for all $C>0$, all $x\in\R^D$ and and all $\ell\ge0$,
$$
	\underset{z\in L^\infty(\R^D;\Zcal) }{\textnormal{sup}}\ 
	\partial^{\textnormal{osc}}_{B_{\ell + 1}(x)}
	g_f[z]\le (1+C+2\|f\|_{L^\infty(\Zcal)})
	e^{-\frac{1}{1+C+2\|f\|_{L^\infty(\Zcal)}}(|x|-\ell)_+},
$$
where the oscillation $\partial^{\textnormal{osc}}_{B_{\ell + 1}(x)}g_f[\cdot]$ is defined as in \eqref{Eq:OscDef}.
\end{lemma}

%

\begin{proof}
For fixed $z\in L^\infty(\R^D,\Zcal)$, and any $x\in\R^D$, $\ell\ge0$, we bound
\begin{equation}
\label{Eq:QuantOfInt}
\begin{split}
	\textnormal{sup}&\{g_f[z_1], \ z_1:\R^D\to\Zcal\ 
	\textnormal{measurable} : z_1 = z\
	\textnormal{on}\ \R^D\backslash B_{\ell + 1}(x)\}\\
	&\ \ \quad - \textnormal{inf}\{g_f[z_2], 
	\ z_2:\R^D\to\Zcal\ \textnormal{measurable}
	: z_2 = z\ \textnormal{on}\ \R^D\backslash B_{\ell + 1}(x)\}.
\end{split}
\end{equation}
Note that if $|x|>\ell + 1$ then $0\in\R^D\backslash B_{\ell + 1}(x)$ and therefore \eqref{Eq:QuantOfInt} equals
$$
	f(z(0))-f(z(0))= 0.
$$
On the other hand, if $|x|\le \ell + 1$, \eqref{Eq:QuantOfInt} is trivially bounded by  $2\|f\|_{L^\infty(\Zcal)}$, showing that
$$
	|\partial^{\textnormal{osc}}_{B_{\ell + 1}(x)}g_f[z]|
	\le  2\|f\|_{L^\infty(\Zcal)}1_{B_{\ell + 1}}(x).
$$
Thus, if $|x|>\ell +1$ then $(|x| - \ell)_+>1$ and
$$
	\partial^{\textnormal{osc}}_{B_{\ell + 1}(x)}g_f[z] = 0 
	< (2\|f\|_{L^\infty(\Zcal)} + 1 + C)
	e^{-\frac{1}{1+C+2\|f\|_{L^\infty}}(|x|-\ell)_+}.
$$
If $\ell < |x|\le \ell+1$ then $(|x| - \ell)_+\in (0,1]$ and
\begin{align*}
	|\partial^{\textnormal{osc}}_{B_{\ell + 1}(x)}g_f[z]|
	&\le 2\|f\|_{L^\infty(\Zcal)}
	\le 
	(C+1+2\|f\|_{L^\infty(\Zcal)})
	e^{-\frac{1}{C+1+2\|f\|_{L^\infty(\Zcal)}}(|x|-\ell)_+},
\end{align*}
having used that $u\le (1+u)e^{-\frac{1}{1+u}}$ for all $u\ge0$.
 Finally, if $|x|\le \ell$ then $(|x| - \ell)_+=0$ and
$$
	|\partial^{\textnormal{osc}}_{Z,B_{\ell + 1}(x)}g_f(Z) | 
	\le 2\|f\|_{L^\infty}
	\le C+1+2\|f\|_{L^\infty(\Zcal)}.
$$
\end{proof}

%
%
%
%
%

\section{Proofs for Section \ref{sec:localrates}} \label{sec:prooflocal}

Recalling the notations $\rho(\eps)\equiv\rho(B_\eps)$, $\rho^*$, $y_\eps$ and $\eps_l^0$ from Section \ref{sec:localrates}, let $\rho_0(\eps)$, $\rho_0^*$, $y_\eps^0$ be similarly defined for $\rho=\rho_0$. Throughout this section we  will write, for any $1\le l\le L_n$,
\begin{equation}
\label{Eq:AuxNot}
	\rho_0^l(z_0)
	:=\frac{\rho_0(\epsltrue)}{\rho_0^* \mu_n(B_{\epsltrue})}
	=\prod\limits_{l'\le l}y^0_{\epsltrue},
\end{equation}
and, in slight abuse of abuse of notation $\rho_0^\mathcal L (z_0) := \prod_{l \in \mathcal L}y^0_{\epsltrue}$ for any $\mathcal L \subseteq \{0, \cdots, L_n\}$. Note that
$$
	\frac{\rho_0(\epsltrue)}{\mu_n(B_{\epsltrue})} =  \rho_0^* \rho_0^l(z_0).
$$

%
%
%

\subsection{Proof of Theorem \ref{Theo:LocalRates}}

Let $\Scal :=\left\{l:y_\epsltrue\neq 1\right\}$, which is a random set under the prior distribution, and denote by 
$$
	\mathcal L(\gamma) 
	:=\left \{1\le l \le  L_n:  |y_{\epsltrue}^0 - 1| 
	> \gamma \frac{ \sqrt{ \log n} }{ \sqrt{ n \rho_0(\epsltrue)}} \right\},
	\qquad \gamma>0,
$$
the set of true coefficients $y_{\epsltrue}^0$ that are `significantly different from $1$'.
Moreover,  for any $C>0$, let $L_n^*(C)\in\N$  be such that
\begin{equation*}
	2^{L_n^*(C)} \in   
	(C, 2C] \left(\frac{\log n}{n}\right)^{-\frac{d}{2\alpha_0+d}}. 
\end{equation*}
We first show that $|\rho_0(z_0) - \rho_0^* \rho_0^{L_n^*(C)} |\lesssim \epsilon_n$  and that for all $\gamma $, there exists $C_\gamma$ such that $	\mathcal L(\gamma)\subset \{1, \cdots,L_n^*(C_\gamma)\}$.   
We then show in Lemmas \ref{lem:1}  and \ref{lem:2} that, with large probability under the posterior distribution, the set $\Scal$ of coefficients $y_{\epsltrue}$ that are different from 1 is contained in  $\mathcal L(\underline\gamma)$ for some sufficiently small $\underline \gamma>0$, and that  it contains $\mathcal L(\bar\gamma)$ for  some $\bar\gamma >\underline \gamma$ large enough, so that the last step of the proof is to control $\Pi ( |\rho^S - \rho_0^S| > M_n' \epsilon_n| D^{(n)} )$ for any $S$ such that  $\mathcal L(\bar \gamma) \subset S \subset \mathcal L(\underline\gamma)$ and where
 $$
	\rho^{\Scal} = \prod_{l\in S} y_{\eps_0^l}  = \rho(z_0)/\rho^*. 
$$

We first bound $\rho_0(z_0) - \rho_0^* \rho_0^{L_n^*(C)} $   and that for all $\gamma $, there exists $C_\gamma$ such that $	\mathcal L(\gamma)\subset \{1, \cdots,L_n^*(C_\gamma)\}$
Due to the H\"{o}lder continuity and the assumption on the diameter of the sets $B_\eps$ (cf.~Conditions \ref{Cond:Diam} and \ref{Cond:Holder}) we have for all $l\geq l_0$ such that $2^{-l_0}C_d \le  \delta_0$, 
\begin{equation}\label{trunc:approx}\begin{split}
	\left|  y^0_{\epsltrue } - 1 \right| & =
 	\left| \frac{ \rho_0( B_{\epsltrue}) }{  \rho_0(B_{\eps_0^{l-1}})}  
  	\times\frac{ \mu_n(B_{\eps_0^{l-1}}) }{  \rho_0( B_{\epsltrue})}  - 1\right| 
  	\leq \frac{ C_0\text{diam} (B_{\eps_0^{l-1}})^{\beta_0}}{ \rho_0(z_0) } 
	\leq  \frac{ 2C_0 C_d^{\beta_0}}{ \rho_0(z_0)} 2^{-l \beta_0/d};\\
	\left|\rho_0(z_0)-\frac{\rho_0(\epsltrue)}{\mu_n(B_{\epsltrue})}\right|
	&= \left|\rho_0(z_0)- \rho_0^* \rho_0^{l}\right| 
	\le  C_0 C_d^{\beta_0} 2^{-\beta_0 l/d}.
\end{split}
\end{equation}

so that if $l\geq L_n^*(C)$, using that $\rho_0(B_{\epsltrue}) \geq c_02^{-l}/C_d$,
\begin{align*}
	\left|  y^0_{\epsltrue } - 1 \right|  
	&\leq  \frac{ 2C_0 C_d^{\beta_0}}{ \rho_0(z_0)} 2^{-\beta_0 l/d} \\
	&\leq \frac{ 2C_0 C_d^{\beta_0+1/2}  2^{-(\beta_0+d/2) l/d} }{ c_0^{3/2} } \sqrt{\rho_0(B_{\epsltrue})}  \\
  	& \leq  \frac{ 2C_0 C_d^{\beta_0+1/2} }{ c_0^{3/2} } \sqrt{\rho_0(B_{\epsltrue})}  
	C^{-(\beta_0+d/2)/d}  \sqrt{ \frac{ \log n}{ n} } 
	\leq  \gamma \frac{ \sqrt{ \log n} }{ \sqrt{ n \rho_0(\epsltrue)}}
\end{align*}
as soon as $C \geq \bar C_0 \gamma^{ \frac{ -2 d }{ 2\beta_0 + d}}:= C_\gamma $ with $\bar C_0^{2d/(2\beta_0+d)} =  2C_0 C_d^{\beta_0+1/2} /c_0^{3/2}$. Therefore  $\Lcal(\gamma)\subseteq\{1,\dots, $ $L_n^*(C_\gamma) \}$. 
We also have 
\begin{equation*}
	\left|\rho_0(z_0)-\rho_0^* \rho_0^{L_n^*(C)} \right|
	\lesssim  C^{-\beta_0/d} \left(\frac{\log n}n\right)^{-\frac{\beta_0}{2\beta_0+d}}
	= C^{-\beta_0/d} \epsilon_n.
\end{equation*}
Let $\gamma >0$; writing shorthand $L_n^*(\gamma) ;=  L_n^*(C_\gamma)$ we bound
\begin{equation}\label{decomp:rho1}
\begin{split}
	|\rho_0(z_0) - \rho&(z_0)|  \le |\rho_0(z_0) - \rho_0^*  \rho_0^{L_n^*(\gamma)}|  
	+  |\rho_0^*  - \rho^* | \rho^{\Scal}(z_0)  
 + \rho_0^* | 
	\rho_0^{L_n^*(\gamma) }(z_0)-\rho^{\Scal}(z_0)|  \\
	& \lesssim  \gamma^{2\beta_0/(2\beta_0+d)}  \epsilon_n  
	+  \rho_0^* | \rho_0^{L_n^*(\gamma) }(z_0)-\rho_0^{\Scal}(z_0)| 
	+ \rho_0^* | \rho_0^{\Scal}(z_0)-\rho^S(z_0)|+ |\rho_0^* -\rho^* |\rho^S(z_0).
	\end{split}
\end{equation}

We now bound each term in  the right hand side of \eqref{decomp:rho1}. We start with the fourth term. 
Recall from after \eqref{post} that  the posterior density of $\rho^*|D^{(n)}$ satisfies
$$
	\pi(\rho^* | D^{(n)}) \propto \pi_{\rho}(\rho^* ) (\rho^*)^{N_n}e^{-\rho^* n}.
$$
This coincides with posterior density arising from a Poisson likelihood with parameter $\rho^* n$ and the positive and continuous prior density $\pi_\rho$. Such posterior concentrates at the parametric rate $1/\sqrt{n}$ around $\rho_0^*$; thus, for any sequence $M_n\to\infty$,
\begin{equation} 
\label{post:rhostar}
	E_{\rho_0}^{(n)}\Pi( \rho^*:|\rho^* - \rho_0^* | > M_n /\sqrt{n} | D^{(n)}) \to 0,\qquad n\to\infty.
\end{equation}
It follows that the last term in the right hand side of \eqref{decomp:rho1} is bounded by $ M_n /\sqrt{n} = o(\epsilon_n)$ , with posterior probability tending to one.

The second term is bounded using Lemma \ref{lem:1} below. Indeed, let $\Omega:=\cap_{l=1}^{L_n}\Omega_\epsltrue(\kappa)$ with $\Omega_\epsltrue(\kappa)$ the event defined in Lemma \ref{lem:0}, satisfying
\begin{equation*}
	P^{(n)}_{\rho_0}(\Omega^c)
	\le 2L_n n^{-\frac12\kappa^2}\to0, \qquad n\to\infty, \qquad \forall \kappa >0.
\end{equation*}
Let $\gamma >0$ and $L_n^*(\gamma)$ be such that $ \mathcal L( \gamma) \subseteq \{1, \cdots , L_n^*(\gamma) \}$, and let $C\geq C_{\gamma}$. Then 
\begin{equation}
\begin{split} \label{diff:LnLgamma}
	|\log&\rho_0^{L_n^*(C) }(z_0)-\log\rho_0^{\Lcal(\gamma)}(z_0)|\\
	&\le \sum_{l: l\notin \Lcal(\gamma),l\le  L_n^*(C) }
	|\log(y^0_{\epsltrue})| \\
	&\le  \gamma\sqrt{\frac{\log n}n}\sum_{l: l\notin \Lcal(\gamma), l \le L_n^*(C) }
	\left(\rho_0(B\epsltrue)\right)^{-1/2} 
	\le 2 \gamma\left(\frac{2^{L_n^*(C)}\log n }{ c_0 n }\right)^{1/2}
	\le  \gamma \frac{ 2 \sqrt{C} }{ \sqrt{c_0}}
	\left(\frac{\log n}{n}\right)^{-\frac{\beta_0}{2\beta_0+d}}.
\end{split}
\end{equation}
Therefore, on the event $\mathcal L(\bar \gamma) \subseteq \Scal \subseteq \mathcal L(\underline \gamma)$, with $L_n^*  = L_n^* (\underline \gamma)$, 
\begin{align*}
 	| \rho_0^{L_n^* }(z_0)-\rho_0^{\Scal}(z_0)| 
	\le   \rho_0^{L_n^* }(z_0)\vee \rho_0^{\Scal}(z_0) 
	\left| \log \rho_0^{L_n^* }(z_0) -  \log \rho_0^{\mathcal L(\bar \gamma)}(z_0)\right| 
	\lesssim \frac{\bar \gamma }{ \underline \gamma^{\beta_0/(2\beta_0+d)}} \epsilon_n
\end{align*}
the last inequality following from \eqref{diff:LnLgamma}. From lemma \ref{lem:1}, there exists $\bar \gamma>0$ such that on the set $\Omega$ introduced above we have
\begin{align*}
	\Pi(\Scal^c\cap\mathcal L(\bar\gamma) \neq \emptyset |D^{(n)})\to0,
\end{align*}
while by Lemma \ref{lem:2} there exist $\underline \gamma < \bar \gamma$ such that on $\Omega$, 
\begin{align*}
	\Pi(\Scal\cap\mathcal L(\underline\gamma)^c \neq \emptyset |D^{(n)})\to0.
\end{align*}
Therefore, with $L_n^*  = L_n^* (\underline \gamma)$, for any $M_n \to\infty$, 
\begin{equation}
\label{term2}
\Pi(\rho: | \rho_0^{L_n^* }(z_0)-\rho_0^{\Scal}(z_0)|  > M_n \epsilon_n/4|D^{(n)}) 
	\to 0, \qquad n\to\infty .
\end{equation}

	We conclude bounding the third term in \eqref{decomp:rho1}
More precisely we first  bound for $M_n$ going to infinity arbitrarily slowly and $M_1>0$ to be determined later, 
\begin{align*}
\Pi( &|\rho_0^{\Scal}(z_0)	-\rho^{\Scal}(z_0)>M_n\epsilon_n/4|D^{(n)}) \\
& \leq 
\Pi\left( |\rho_0^{\Scal}(z_0)	-\rho^{\Scal}(z_0)>M_n\epsilon_n/4, \ \sum_{l\in \Scal} |  y_{\varepsilon_0^l}- y_{\varepsilon_0^l}^0| \leq M_1 M_n\epsilon_n/4 \Big |D^{(n)}\right) \\
 & \quad + \Pi\left( \sum_{l\in \Scal} |  y_{\varepsilon_0^l}- y_{\varepsilon_0^l}^0\Big| >M_1 M_n\epsilon_n/4  |D^{(n)}\right)
 \end{align*}
 To control the first term of the right hand side, note that 
 \begin{align*}
 |\rho^{\Scal} - \rho_0^{\Scal}| & \leq \rho_0^{\Scal}\sum_{S' \subset \Scal, S'\neq \emptyset}\prod_{l'\in S'} \frac{ |y_{\varepsilon_0^l}  -y_{\varepsilon_0^l}^0|}{y_{\varepsilon_0^l}^0} \\
 & \leq   \rho_0^{\Scal}\sum_{l \in \Scal } \frac{ |y_{\varepsilon_0^l}  -y_{\varepsilon_0^l}^0|}{y_{\varepsilon_0^l}^0}  +\rho_0^{\Scal}\sum_{S' \subset \Scal, |S'|\geq 2}\prod_{l'\in S'} \frac{ |y_{\varepsilon_0^l}  -y_{\varepsilon_0^l}^0|}{y_{\varepsilon_0^l}^0} .
  \end{align*}
  Since $y_{\varepsilon_0^l}^0\geq c_0$ and 
  $$\sum_{l\in \Scal} |  y_{\varepsilon_0^l}- y_{\varepsilon_0^l}^0| \leq\frac{ M_1 M_n\epsilon_n}{4} := c_0u_n$$
  then 
  \begin{align*}
  \sum_{S' \subset \Scal, |S'|\geq 2}\prod_{l'\in S'} \frac{ |y_{\varepsilon_0^l}  -y_{\varepsilon_0^l}^0|}{y_{\varepsilon_0^l}^0}  & \leq  \sum_{r=2}^{|\Scal|}\sum_{S' \subset \Scal, |S'|=r} \left(\frac{ M_1M_n\epsilon_n}{ 4c_0} \right)^r\\
   &   =  \sum_{r=2}^{\Scal|}\frac{ |\Scal | ! }{ r ! (|\Scal| -r)!} u_n^r  = (1 + u_n)^{|\Scal|} - 1 - |\Scal|u_n= O(|\Scal|^2 u_n^2 ).
  \end{align*}
Moreover $|\Scal|^2u_n \lesssim  \log n^2 \epsilon_n=o(1)$, so that 
  \begin{align}\label{diff:rhoS}
 |\rho^{\Scal} - \rho_0^{\Scal}| & \leq \frac{\rho_0^{\Scal}}{ c_0} \sum_{S' \subset \Scal, S'\neq \emptyset}\prod_{l'\in S'}  |y_{\varepsilon_0^l}  -y_{\varepsilon_0^l}^0| + o(\epsilon_n) 
  \end{align}
  which implies that if $M_1 \leq  \frac{ c_0}{2\rho_0^{\Scal}}$ 
 $$ \Pi( |\rho_0^{\Scal}(z_0)	-\rho^{\Scal}(z_0)>M_n\epsilon_n/4,\  \sum_{l\in \Scal} |  y_{\varepsilon_0^l}- y_{\varepsilon_0^l}^0| \leq M_1 M_n\epsilon_n/4  |D^{(n)})=0.$$
Let 
$$
	\hat{y}_{\epsltrue} 
	:=  \frac{ N_{\epsltrue} + \alpha_l\alpha_n(\epsltrue) }
	{ \alpha_n(\epsltrue)(N_{P(\varepsilon)}+\alpha_l) }, 
$$
then, on the event 
$$
	\Omega_{0,n}
	:= \left\{ \sum_{l \in \Lcal(\underline \gamma)} 
	|  y_{\epsltrue}^0-\hat{y}_{\epsltrue}|\le  M_n' \epsilon_n/2 \right\} 
	\cap \Omega,
$$
 writing $\alpha_n := \alpha_n(\epsltrue)$, $M_n' = M_1M_n/4$ we have for $n$ large enough and since $M_n \epsilon_n = o(1)$, 
\begin{align*}
	&\Pi( | \rho_0^{\Scal}(z_0)-\rho^{\Scal}(z_0)|  > M_n\epsilon_n/4|D^{(n)})  \le   
	\Pi\left( \Scal \subset \Lcal(\underline \gamma), \ \sum_{l \in \Scal} 
	| y_{\epsltrue}-y_{\epsltrue}^0|> M_n'\epsilon_n\Bigg|N^{(n)}\right)  +o_{P^{(n)}_{\rho_0}}(1)\\
	& \le   
	\Pi\left( \Scal \subset \Lcal(\underline \gamma), \ \sum_{l \in \Scal} 
	| y_{\epsltrue}-\hat y_{\epsltrue}|> M_n'\epsilon_n/2\Bigg|N^{(n)}\right)  +o_{P^{(n)}_{\rho_0}}(1)\\
	&\ \le  
	\sum_{l \in \Lcal(\underline \gamma)}
	\frac{2 E^\Pi\left[ 1_{\{l \in \Scal\}} | y_{\epsltrue}-\hat y_{\epsltrue}| \Big| N^{(n)} \right] }
	{ M_n' \epsilon_n } +o_{P^{(n)}_{\rho_0}}(1)\\
	&\ \le  
	\sum_{l \in \Lcal(\underline \gamma)} \frac{2 \int_0^1 | \bar y/\alpha_n -\hat y_{\epsltrue}|
	\bar y^{N_{\epsltrue}+\alpha_l\alpha_n -1} (1- \bar y)^{N_{P(\varepsilon_0^l)}
	-N_{\epsltrue}+\alpha_l(1-\alpha_n)-1}
	 d\bar{y}\Gamma(N_{P(\varepsilon_0^l)} +  \alpha_l )}{ \Gamma( N_{\epsltrue}+\alpha_l\alpha_n) 
	 \Gamma(N_{P(\varepsilon_0^l)}-N_{\epsltrue}+\alpha_l(1-\alpha_n)) M_n' \epsilon_n } 
	 +o_{P^{(n)}_{\rho_0}}(1)\\
	 &\ = 
	 \sum_{l \in \Lcal(\underline \gamma)}\alpha_n^{-1} \frac{2 \int_0^1 
	 | \bar y -\alpha_n\hat y_{\epsltrue}|
	 \bar y^{N_{\epsltrue}+\alpha_l\alpha_n -1} (1- \bar y)^{N_{P(\varepsilon_0^l)}-N_{\epsltrue}
	 +\alpha_l(1-\alpha_n)-1} d\bar{y}\Gamma(N_{P(\varepsilon_0^l)} +  \alpha_l )}
	 { \Gamma( N_{\epsltrue}+\alpha_l\alpha_n) 
	  \Gamma(N_{P(\varepsilon_0^l)}-N_{\epsltrue}+\alpha_l(1-\alpha_n)) M_n' \epsilon_n } 
	  +o_{P^{(n)}_{\rho_0}}(1),
\end{align*}
where we have used the expression of the posterior distribution of $\bar y_\eps|D^{(n)}$ in \eqref{post:polyaT} . Then, 
\begin{align*}
	&\frac{2}{ M_n' \epsilon_n } \sum_{l \in \Lcal(\underline \gamma)}
	\frac{ \alpha_n^{-1}}{ \sqrt{ N_{P(\varepsilon_0^l)}+\alpha_l }} 
	\sqrt{ \frac{ N_{\epsltrue} + \alpha_l\alpha_n }{ N_{P(\varepsilon_0^l)}+\alpha_l } 
	\left(1 -\frac{ N_{\epsltrue} + \alpha_l\alpha_n }{ N_{P(\varepsilon_0^l)}+\alpha_l }  \right)} 
	+o_{P^{(n)}_{\rho_0}}(1)\\
	&\quad \le  \frac{1}{ M_n' \epsilon_n } \sum_{l \in \Lcal(\underline \gamma)}
	\frac{ 1}{ \sqrt{ n \rho_0(P(\epsltrue))+\alpha_l - 2\kappa 
	\sqrt{n\log n \rho_0(P(\epsltrue))} }} +o_{P^{(n)}_{\rho_0}}(1)\\
	 &\quad \lesssim \frac{1}{ M_n' \epsilon_n } 
	 \sum_{l \in \Lcal(\underline \gamma)}\frac{ 1}{ \sqrt{ n2^{-l}+ \alpha_l  }}  
	 +o_{P^{(n)}_{\rho_0}}(1)
  	\lesssim \frac{1}{ M_n' \sqrt{n}\epsilon_n } 
  	2^{L_n^*(C_{\underline \gamma})/2} = O(1/M_n )+o_{P^{(n)}_{\rho_0}}(1)
	=o_{P^{(n)}_{\rho_0}}(1),
 \end{align*}
since $ \Lcal(\underline \gamma) \subset \{ l \le  L_n^*(C_{\underline \gamma})\}$ and 
$$
	\sqrt{n\log n \rho_0(P(\epsltrue))} = o( n \rho_0(P(\epsltrue)) ) .
$$
Finally, on $\Omega_{0,n}$,  
\begin{equation*}
\begin{split}
	\sum_{l \in \Lcal(\underline \gamma)} |  y_{\epsltrue}^0-\hat{y}_{\epsltrue}| 
	& \le  \sum_{l \in \Lcal(\underline \gamma)} \frac{ 2 \kappa }{ \alpha_n(\epsltrue)} 
	\times \frac{ \sqrt{\log n} }{ \sqrt{ n \rho_0 (P(\epsltrue)) } } 
	\le  \frac{ 2 \kappa \sqrt{ \log n} }{ c_0\sqrt{n} }\sum_{l \le  L_n^*(C_{\underline{\gamma}})}  2^{l/2}
	 \lesssim  \kappa \epsilon_n, 
\end{split}
\end{equation*}
so that for any $M_n\to\infty$, using Lemma \ref{lem:0}, $P_{\rho_0}^{(n)}( \Omega_{0,n}^c) \to 0$ and 
\begin{align*}
	\Pi\left(\rho: |\rho_0(z_0) - \rho(z_0)|> M_n \epsilon_n|D^{(n)}\right) 
	& \to 0
\end{align*}
in $P_{\rho_0}^{(n)}$-probability, concluding the proof.\qed

%
%
%

\subsection{Auxiliary Results for the proof of Theorem \ref{Theo:LocalRates}}

	The first two auxiliary results provide an upper bound for the probability of the event $\Omega$ appearing in the proof of Theorem \ref{Theo:LocalRates}. Recall that for $\varepsilon\in\Ecal_l$, $A(\varepsilon)$ denotes its twin bin.
	
\begin{lemma}\label{lem:0}
For $\kappa>0$ and $N_\varepsilon$ defined as after \eqref{likeli}, consider the event
\begin{equation}
\begin{split}
	&\Oomega\\
	&\ :=\left\{ |N_\eps-n\rho_0(B_\eps)|\le  \kappa\sqrt{n\rho_0(B_\eps)\log n }, \ 
	|N_\aeps-n\rho_0(B_\aeps)|\le  \kappa\sqrt{n\rho_0(B_\aeps)\log n }\right\}.
\end{split}
\end{equation}
Then, for all sufficiently large $n$,
\begin{align*}
	P^{(n)}_{\rho_0}(\Oomega^c)\le  4e^{-\frac12\kappa^2\log n}.
\end{align*}
\end{lemma}

\begin{proof}
Note that if a random variable $N$ has a Poisson distribution with a parameter $\lambda>0$, then by Markov inequality, for any $t>0$,
\begin{align*}
	\Pr(N-\lambda>t\sqrt{\lambda})
	\le  e^{-t^2/2-t\sqrt{\lambda}+t\sqrt{\lambda}-\lambda}, \quad \Pr(- N+\lambda>t\sqrt{\lambda})
	\le  e^{-t^2/2},
\end{align*}
Therefore, provided that $\lambda$ is large enough, and $0 < t < \sqrt{\lambda }$
\begin{align*}
	\Pr\left(\frac{t}{\sqrt{\lambda}}|N-\lambda|>t\right) \le  2e^{-t^2/2},
\end{align*}
The claim then follows upon noting that $N_\eps$ has a Poisson distribution with the parameter $\lambda=n\rho_0(\eps)$.
\end{proof}

	An immediate application of Lemma \ref{lem:0} above and the union bound yields the following result.

\begin{corollary}\label{Cor:Omega}
Let $\Omega:=\cap_{l=1}^{L_n}\Omega_\epsltrue(\kappa)$, $\kappa>0$. Then, for all $\kappa>0$
\begin{align*}
	P^{(n)}_{\rho_0}(\Omega^c)\le  2L_n n^{-\frac12\kappa^2}.
\end{align*}
\end{corollary}

\begin{remark}[Towards uniform pointwise rates]\label{Rem:Unif}
Note that had we tried to obtain a result uniform in $z_0$, 
we could have followed \cite[Section 5]{RR23} and considered a net of $[0,1]^d$ with radius $\zeta (n/\log n)^{-1/(2\beta+d)}$, say $(z_i)_{i \leq I_n}$, $I_n \lesssim (n/\log n)^{d/(2\beta+d)}$. Then, for each such $z_i$, we would take a set $\Omega(z_i)$ as defined above and $\bar \Omega = \cap_{i=1}^{I_n} \Omega(z_i) $. This would lead to a constraint on $\kappa $ of the form $\kappa^2 >  2d /( 2 \beta+ d)$, which would then imply that in assumption (i) of Theorem \ref{Theo:LocalRates} we need $t>2d /( 2 \beta+ d)$ (as opposed to any $t>0$). Thus, if for all $z \in [0,1]^d$ we write $\beta(z)$ for the local H\"older smoothness at $z$ and set $\epsilon_n(z) := (n/\log n)^{-\beta(z)/(2\beta(z)+d)}$, under the stronger assumption  $t>2d /( 2 \beta+ d)$ it holds that
$$
	\Pi\left( \rho:\sup_{z\in [0,1]^d} \epsilon_n(z)^{-1}|\rho(z)- \rho_0(z)| > M_n \Bigg|D^{(n)} 
	\right) \to 0, 
$$
in $P_{\rho_0}^{(n)}$-probability as $n\to\infty$ for all $M_n$ going to infinity.  A result of this type was proved for instance in Theorem 1 in \cite{RR23}, in the case of univariate nonparametric regression. A direct consequence of this is that also
$$
	\Pi(\rho:\|\rho- \rho_0\|_{L^\infty([0,1]^d)} > M_n\epsilon_n |D^{(n)} ) =o_{P_{\rho_0}^{(n)}}(1), 
	\quad \epsilon_n := (n/\log n)^{-\beta/(2\beta+d)} \geq \epsilon_n(z),
$$
where $\beta\le\beta(z)$ is the global H\"older smoothness of $\rho_0$ from Condition \ref{Cond:Holder}.
\end{remark}

	The next two key lemmas show that the posterior distribution consistently identifies, in the large sample size limit, the set of true coefficients $y^0_{\varepsilon^l_0}$ that are significantly different from $1$. Recall the definition of the sets $\Scal$ and $\Lcal(\gamma), \ \gamma>0$, from the beginning of the proof of Theorem \ref{Theo:LocalRates}.

\begin{lemma}\label{lem:1}
Consider a P\'olya tree prior $\Pi $ constructed as in Theorem \ref{Theo:LocalRates}. Then, there
exists $\bar \gamma>0$ such that, on the event $\Omega$ defined in Corollary \ref{Cor:Omega}, 
\begin{align*}
	\Pi(\Scal^c\cap\mathcal L(\bar\gamma) \neq \emptyset |D^{(n)})\to0.
\end{align*}
\end{lemma}

\begin{proof}

Recalling the expression of the likelihood in \eqref{likeli}, decompose $( \eps \in \Ecal_l)$ into a set of distinct pairs $(\eps, A(\eps))$, with $\eps, A(\eps)$ the two children of $P(\eps)$ and denote by $\bar{\Ecal}_l$ the set of the obtained pairs. For fixed $\eps\in\mathcal L(\bar\gamma)$, the posterior density $\pi(y_\eps|D^{(n)})$ is then proportional to 
\begin{align*}
	 (\alpha_n(\eps)y_\eps)^{N_\eps}&(1-\alpha_n(\eps)y_\eps)^{N_\aeps}
	\\
	&\times \left(q_{\eps}\delta_1
	+ (1-q_{\eps})\frac{(\alpha_n(\eps) y_{\eps})^{\alpha_\eps \alpha_n(\varepsilon)-1}
	(1-\alpha_n(\eps)y_{\eps})^{\alpha_\eps(1-\alpha_n(\eps))-1}}{B(\alpha_\eps\alpha_n(\eps),\alpha_\eps(1-\alpha_n(\eps)))}\right).
\end{align*}
Thus,
\begin{align*}
	\pi&(y_\eps=1 |D^{(n)})\\
	&= q_\eps \alpha_n(\eps)^{N_\eps}(1-\alpha_n(\eps))^{N_\aeps}
	\Bigg(q_\eps\alpha_n(\eps)^{N_\eps}(1-\alpha_n(\eps))^{N_\aeps}
	+(1-q_\eps)\alpha_n(\eps)^{-1}\\
	&\quad\times B\big(N_\eps+\alpha_\eps\alpha_n(\eps),
	N_\aeps+\alpha_\eps(1-\alpha_n(\eps))\big)
	/B\big(\alpha_\eps\alpha_n(\eps),\alpha_\eps(1-\alpha_n(\eps))\big) \Bigg)^{-1}
	\\
	&=\left(1+\alpha_n(\eps)^{-N_\eps}
	(1-\alpha_n(\eps))^{-N_\aeps}\frac{1-q_\eps}{\alpha_n(\eps)q_\eps}
	\times\frac{B\big(N_\eps+\alpha_\eps\alpha_n(\eps),N_\aeps+\alpha_\eps(1-\alpha_n(\eps))\big)}
	{B\big(\alpha_\eps\alpha_n(\eps),\alpha_\eps(1-\alpha_n(\eps))\big)}\right)^{-1}\\
	&= \left( 1 +\frac{1-q_\eps}{q_\eps}\Gamma_\eps\right)^{-1},
\end{align*}
where, denoting by $\tilde N_\eps:= \nes$ and by $\tilde N_\aeps:=N_\aeps + \alpha_\eps(1-\alpha_n(\eps))$,
\begin{align*}
	\Gamma_\eps 
	&:=\frac{ \alpha_n(\eps)^{-N_\eps}(1-\alpha_n(\eps))^{-N_\aeps} }{\alpha_n(\eps)} 
	\times\frac{\Gamma(\tilde N_\eps)\Gamma(\tilde N_\aeps)}{\Gamma(N_\eps+N_\aeps+\alpha_\eps)}
	\times\frac{\Gamma(\alpha_\eps)}{\Gamma(\alpha_\eps\alpha_n(\eps))
	\Gamma(\alpha_\eps(1-\alpha_n(\eps))}\\
	&=\frac{ \alpha_n(\eps)^{\alpha_\eps\alpha_n(\eps)}(1-\alpha_n(\eps))^{\alpha_\eps(1-\alpha_n(\eps))} }
	{\alpha_n(\eps)}
	\times
	\frac{\Gamma(\alpha_\eps)}{\Gamma(\alpha_\eps\alpha_n(\eps))
	\Gamma(\alpha_\eps(1-\alpha_n(\eps))}
	\times\sqrt{\frac{2\pi \tilde N_\eps\tilde N_\aeps }{\nnes}}\\
	&\quad\
	\times \left(1+O\left(\frac{1}{N_\eps \wedge N_\aeps}\right)\right)
	\times \exp\Big\{(\tilde N_\eps-1)\log(\tilde N_\eps-1)+(\tilde N_\aeps  -1)\log(\tilde N_\aeps -1) \\
	&\quad - \tilde N_\eps \log \alpha_n(\eps) - \tilde N_\aeps\log (1-\alpha_n(\eps) )
	-(\nnes-1)\log(\nnes-1)\Big\},
\end{align*} 
 having used Stirling's formula. Also, letting
\begin{align*}
	p_N(\eps):=\frac{\tilde N_\eps}{\tilde N_\eps+\tilde N_\aeps},
\end{align*}
then,
\begin{equation}
\label{eq:a}
	\Gamma_\eps
	=  H_\eps
	\sqrt{\frac{1 }{p_N(\eps)  \tilde N_\aeps}}
	\exp\left((\tilde N_\eps+\tilde N_\aeps)\text{KL}(p_N(\eps),\alpha_n(\eps))\right)
	\left(1+o(1)\right),
\end{equation}
where 
$$ 
	H_\eps=  
	\frac{ \alpha_n(\eps)^{\alpha_\eps\alpha_n(\eps)}(1-\alpha_n(\eps))^{\alpha_\eps(1-\alpha_n(\eps))} }
	{\alpha_n(\eps)}
	\times
	\frac{\Gamma(\alpha_\eps)}{\Gamma(\alpha_\eps\alpha_n(\eps))
	\Gamma(\alpha_\eps(1-\alpha_n(\eps))} = \frac{\alpha_\eps}{\alpha_n(\eps) } ( 1 + O(1/\alpha_\eps ))
$$
and 
\begin{align*}
	\text{KL}(p_N(\eps),\alpha_n(\eps))
	=p_N(\eps)\log\left(\frac{ p_N(\eps)}{\alpha_n(\eps)}\right)+(1-p_N(\eps))
	\log\left(\frac{ 1-p_N(\eps)}{1-\alpha_n(\eps)} \right).
\end{align*}
Thus, if $\alpha_\eps = O(1)$, then 
\begin{equation}
\label{eq:a}
	\Gamma_\eps
	\simeq 
	\alpha_\eps (1-\alpha_n(\eps))\sqrt{\frac{1 }{p_N(\eps)  \tilde N_\aeps}}
	\exp\left((\tilde N_\eps+\tilde N_\aeps)\text{KL}(p_N(\eps),\alpha_n(\eps))\right).
\end{equation}
Further, by convexity, for $p,q \in (0,1)$, $\text{KL}(p,q)\geq (p-q)^2/2$, and if $p-q = o(1)$, then 
\begin{equation*}
	\text{KL}(p,q)=\frac{ (p-q)^2}{ 2q(1-q)} \Big(1+o(1)\Big).
\end{equation*}
Note that
\begin{equation*}
	\pne
	=\left(\alpha_n y_\eps^0+\frac{\alpha_n \alpha_\eps}{n\rho_0(P(\eps))}
	+\frac{\Delta_\eps}{n\rho_0(P(\eps))}\right)
	\times \left(1+\frac{\alpha_\eps}{n\rho_0(P(\eps))}
	+\frac{\Delta_{P(\eps)}}{n\rho_0(P(\eps))}\right)^{-1},
\end{equation*}
where we set
\begin{align*}
	\Delta_\eps&:=N_\eps-n\rho_0(\eps)=N_\eps-ny_\eps^0\rho_0(\peps);
	\qquad \Delta_{P(\eps)}:=N_\eps+N_\aeps-n\rho_0(\peps).
\end{align*}
Above, we have used that $\rho_0(\eps)=\rho_0(\peps)y_\eps^0\alpha_n(\eps)$. Note that on the set $\Oomega\cap \Omega_{P(\eps)}(\kappa)$ we have
\begin{align*}
	\frac{|\Delta_\eps|}{n\rho_0(\eps)}
	\le  \kappa \sqrt{\frac{\log n}{n\rho_0(\eps)}}; 
	\qquad \frac{|\Delta_{P(\eps)}|}{n\rho_0(P(\eps))}\le 
	 \kappa \sqrt{\frac{\log n}{n\rho_0(\peps)}}.
\end{align*}
Thus,
\begin{align*}
	\text{KL}(p_N(\eps),\alpha_n(\eps)) 
	&\geq \frac{(p_N(\eps)-\alpha_n(\eps))^2}{2} \\
	&= \frac{\alpha_n^2(\eps) \left( y_\eps^0- 1 +\frac{\Delta_\eps -\Delta_{P(\eps)} }
	{n\rho_0(P(\eps))\alpha_n(\eps)}\right)^2}{2
 	\left(1+\frac{\alpha_\eps}{n\rho_0(P(\eps))}+\frac{\Delta_{P(\eps)}}
	{n\rho_0(P(\eps))}\right)^{2} } 
 	\geq 
 	\frac{ \alpha_n^2(\eps) (\bar \gamma - 2\kappa/\alpha_n(\eps))^2 \log n }{ 2n\rho_0(P(\eps))
	\left(1 +\frac{\alpha_\eps}{n\rho_0(P(\eps))} 
	+ \frac{\kappa \sqrt{\log n}}{\sqrt{n\rho_0(P(\eps))} }\right)^{2} },
 \end{align*}
so that on $\Oomega\cap \Omega_{P(\epsilon)}$, provided that $\kappa < \alpha_n(\eps)\bar \gamma/4$, we can lower bound the argument of the exponential in $\Gamma_\eps$ in \eqref{eq:a} by
\begin{align*}
	\frac{[ n \rho_0(P(\eps)) +  \alpha_\eps- \kappa \sqrt{n\log n\rho_0(P(\eps)) }]
	\alpha_n^2(\eps)\bar\gamma^2\log n }{ 2 n\rho_0(P(\eps))\left(1 +\frac{\alpha_\eps}{n\rho_0(P(\eps))} 
	+ \frac{\kappa \sqrt{\log n}}{\sqrt{n\rho_0(P(\eps))} }\right)^{2}}
	\geq \frac{ \alpha_n^2(\eps)\bar \gamma^2 \log n }{3 }.
 \end{align*}
Plugging this into \eqref{eq:a} and using Stirling formula together with thet fact that $\alpha_n(\eps) \in [c_1, 1-c_1]$ leads to the upper bound
\begin{align*}
	\pi(y_\eps=1/2|D^{(n)})
	&\le  C\frac{q_\eps}{(1-q_\eps)\alpha_\eps } e^{-\bar\gamma^2c_1^2 \log n/3}
	\sqrt{\tilde N_\eps\wedge\tilde N_\aeps}\Big(1+o(1)\Big),
\end{align*}
holding, for some $C>0$ independent of $\eps$ and $n$,  for any $\eps\in\mathcal L(\bar\gamma)$. Note that 
\begin{align*}
	\pi(S^c\cap \mathcal L(\bar\gamma)|D^{(n)})
	\le \sum_{\eps\in\mathcal L(\bar\gamma)}\pi(\eps\in S^c|D^{(n)}).
\end{align*}
Also, on $\Oomega$,
\begin{align*}
	{\tilde N_\eps\wedge\tilde N_\aeps}
	\le  n\rho_0(P(\eps))+\kappa\sqrt{n\rho_0(P(\eps))\log n}+\alpha_\eps.
\end{align*}
Therefore,
\begin{align*}
	\Pi(S^c\cap \mathcal L(\bar\gamma)|D^{(n)})
	\lesssim  n^{-\bar\gamma^2c_1^2/3}\sum_{\eps\in \mathcal L(\bar\gamma)} 
	 \frac{ q_\eps}{\alpha_\eps(1-q_\eps)}\left[ \sqrt{n \rho_0(\peps) } + \sqrt{\alpha_\eps}\right] = o(1) 
\end{align*}
as long as $\bar \gamma$ is large enough and
$$
	\alpha_\eps\gtrsim n^{-H}, \qquad 1-q_\eps \gtrsim n^{-H}, \qquad l\le  L_n
$$
for some $H>0$.
\end{proof} 

\begin{lemma} \label{lem:2}
Consider a P\'olya tree prior $\Pi $ constructed as in Theorem \ref{Theo:LocalRates}. Then, there exists $\kappa>0$ such that on the event $\Omega$ (cf.~Lemma \ref{lem:0}),
\begin{align*}
	\Pi(S\cap \mathcal L(\underline\gamma)^c|D^{(n)})\to0.
\end{align*}
\end{lemma}

\begin{proof}
Using throughout the notation introduced in the proof of Lemma \ref{lem:1}, recall that
\begin{align*}
	\pi(y_\eps \neq 1|D^{(n)})
	=\frac{ (1-q_\eps)\Gamma_\eps /q_\eps}
	{1+\frac{1-q_\eps}{q_\eps}\Gamma_\eps},
\end{align*}
and that 
\begin{align*}
	\Gamma_\eps
	\simeq
	\alpha_\eps \sqrt{\frac{1 }{ \tilde N_\aeps \vee \tilde N_\eps }}
	\exp\left((\tilde N_\eps+\tilde N_\aeps)
	\text{KL}(p_N(\eps),\alpha_n(\eps))\right).
\end{align*}
By Taylor's expansion,
\begin{align*}
	\text{KL}(p_N(\eps),\alpha_n(\eps))
	\le  
	\frac{ (p_N(\eps)-\alpha_n(\eps))^2}{ 2\alpha_n(\eps) 
	(1 - \alpha_n(\eps) )} \left(1+O(|p_N(\eps)-\alpha_n(\eps)|)\right).
\end{align*}
Also, since $\rho_0(\eps)-\rho_0(\aeps)=\rho_0(P(\eps))(y^0_\eps-(1-y^0_\eps))$, we have
\begin{align*}
	\frac{ (p_N(\eps)-\alpha_n(\eps))^2}{ 2\alpha_n(\eps) (1 - \alpha_n(\eps) )}
	 (\tilde N_\eps+\tilde N_\aeps)
	&= \frac{ (N_\eps-\alpha_n(\eps)N_{P(\eps)})^2}
	{ 2\alpha_n(\eps) (1 - \alpha_n(\eps) )(\tilde N_\eps+\tilde N_\aeps) } \\
	&=\frac{n\rho_0(P(\eps))\left( \alpha_n(\eps)(y^0_\eps-1)
	+\frac{\Delta_\eps-\alpha_n \Delta_{P(\eps)}}{n\rho_0(P(\eps))}\right)^2}
	{2\alpha_n(\eps) (1 - \alpha_n(\eps) ) \left(1 +\frac{\alpha_\eps + \Delta_{P(\eps)}} { n\rho_0(P(\eps)) } \right)}.
\end{align*}
On $\mathcal L(\underline \gamma)^c$, we have
\begin{align*}
	|y^0_\eps-1|& \le \underline \gamma\sqrt{\frac{\log n}{n\rho_0(\eps)}},
\end{align*}
as well as
\begin{align*}
	\frac{ \rho_0(\eps) }  { \alpha_n(\eps) \rho_0(P(\eps)) }
	&= y^0_\eps = 1+ o(1) ; 
	\qquad  \frac{ \rho_0(\aeps) }{ (1-\alpha_n(\eps)) \rho_0(P(\eps)) }  = y^0_{\aeps}  = 1 + o(1) .
\end{align*}
Therefore, on $\Omega$,
\begin{align*}
	\frac{|\Delta_\eps-\alpha_n(\eps) \Delta_{P(\eps)} |}{n\rho_0(B_{P(\eps)})}
	& = \frac{ |(1-\alpha_n(\eps))\Delta_\eps - \alpha_n(\eps)  \Delta_{\aeps}| }{n\rho_0(B_{P(\eps)})} \\
	& \leq  \frac{ \kappa \sqrt{\log n} }{\sqrt{n\rho_0(B_\peps)}}
	\left( (1 - \alpha_n(\eps))\sqrt{\alpha_n(\eps) y_\eps^0} 
	+\alpha_n(\eps)\sqrt{(1 - \alpha_n(\eps) ) y_{aesp}^0} \right)( 1 + o(1) ) \\
	& \le   \kappa \sqrt{\alpha_n(\eps)( 1 - \alpha_n(\eps)) } 
	\sqrt{\frac{\log n}{n\rho_0(B_\peps)}} \left( \sqrt{\alpha_n(\eps)} + \sqrt{1 - \alpha_n(\eps)}\right)(1+o(1))  
	=o(1). 
\end{align*}
This implies
\begin{align*}
	1 +\frac{\alpha_\eps + \Delta_{P(\eps)}} { n\rho_0(P(\eps)) }  
	\geq  1 +\frac{\alpha_\eps } { n\rho_0(P(\eps)) } +o(1).
\end{align*}
Thus, we have, choosing $\underline \gamma $ small enough, 
\begin{align*}
	(\tilde N_\eps+\tilde N_\aeps)\text{KL}(p_N(\eps),\alpha_n(\eps))
	&\le   \frac{1}{2} 
	\left(( \frac{ \underline \gamma }{ \sqrt{1 - \alpha_n(\eps)} } 
	+ \kappa(    \sqrt{\alpha_n(\eps)} + \sqrt{1 - \alpha_n(\eps)} ) \right)^2 \log n \\
 	& \le \frac{\log n }{2}\left( 2\frac{  \underline \gamma }{ 1 - c_1} 
	+\kappa^2 ( 1 + 2\sqrt{\alpha_n(\eps)(1-\alpha_n(\eps)} )\right)(1 + o(1))\\
 	&  \le \frac{\log n }{2}\left( 2\frac{  \underline \gamma }{ 1 - c_1} +2\kappa^2 \right)(1+ o(1)) .
\end{align*}
It  follows that on $\Omega$, with $\underline \gamma $ arbitrarily small, for some $\tilde c_1>0$,
\begin{align}\label{UB:Gammaeps}
	\Gamma_\eps\le  \frac{\alpha_\eps }
	{\sqrt{c_0}2^{-l/2}  }  e^{ \frac{\log n [\tilde c_1 \underline \gamma + 2 \kappa^2-1] }{2}}. 
\end{align}
This implies that, choosing $\underline  \gamma, \kappa^2$ such that $\tilde c_1 \underline \gamma + 2 \kappa^2\leq 2t $ (with $t$ the quantity in condition (i) in the statement of Theorem \ref{Theo:LocalRates})
\begin{align*}
	\Pi(\Scal\cap \mathcal L^c(\underline\gamma)|D^{(n)})
	\le \sum_{l\in \mathcal L^c(\underline\gamma)}\Pi(\eps_0^l \in \Scal|D^{(n)})
	\le  \frac1{\sqrt{c_0}}n^{\frac{ \tilde c_1\underline \gamma 
	+2\kappa^2-1}{2}}
	\sum_{l\in\mathcal L^c(\underline\gamma)}\frac{\alpha_{\eps_0^l}(1-q_{\eps_0^l})}{q_{\eps_0^l}}
	2^{l/2}.
\end{align*}
The desired result then follows provided that
\begin{align*}
	\sum_{l\in\mathcal L^c(\underline\gamma)}
	\frac{\alpha_{\eps_0^l}(1-q_{\eps_0^l})}{q_{\eps_0^l}} \le  n^{1/2-\tilde t},
\end{align*}
for some $\tilde t>0$ . In particular, the above is implied since by assumption (i) in the statement of Theorem \ref{Theo:LocalRates} there exist $t,c_2>0$ such that
$$
	\alpha_\eps (1-q_\eps) \le  2^{-lt} , \qquad q_\eps> c_2 , 
	\qquad \forall \eps \in \mathcal E_l.
$$
\end{proof}

%
%
%

\subsection{Proof of Proposition \ref{prop:ass:Z:diam}}


%

We have
\begin{align*}
	\text{Var} \left( \int_{ \Wcal_n} 1_{\{Z(x)\in B_\eps\}}d\mu_n(x)   \right) 
	&= 
	\frac{ \int_{\Wcal_n\times \Wcal_n}\text{Cov}( 1_{B_\eps}(Z(0)),1_{ B_\eps}
	(Z(x_2-x_1)))dx_1dx_2 }{n^2} \\
	& \le \frac{ \int_{\overline \Wcal_n}  \text{Cov}( 1_{B_\eps}(Z(0)), 
	1_{ B_\eps}(Z(x)))dx }{n},
\end{align*}
where $\overline \Wcal_n := \{ x\in\R^D: x+y \in \Wcal_n \ \text{for some}\ y\in\Wcal_n\}$. Note that $|\bar \Wcal_n| \lesssim n$. Thus, we obtain
\begin{align*}
	\text{Var} \left( \int_{ \Wcal_n} 1_{\{Z(x)\in B_\eps\}}d\mu_n(x) \right) 
	&\le \nu(B_\eps) 
	\frac{ \int_{\bar \Wcal_n}\text{Corr}( 1_{B_\eps}(Z(0)), 1_{ B_\eps}(Z(x)))dx }{n} 
 \lesssim \frac{  \nu(B_\eps)  }{n},
\end{align*}
having used assumption \eqref{mixing}.\qed

%
%
%

\subsection{Proof of Theorem \ref{thm:L1:polya}} \label{sec:L1polya}

Let $\bar  \epsilon_n = (n/\log n^2)^{-\beta /(2 \beta+d)}$. We start deriving posterior contraction rates in the empirical $L^1$-distance, verifying the assumptions of  Theorem \ref{Theo:GenRandCov}, starting with the small ball lower bound \eqref{Eq:SmallBall}. Let $L_{1,n}\in\N$ be such that $2^{L_{1,n}} \in [ L_1n \bar \epsilon_n^2 , 2L_1 n \bar \epsilon_n^2]L_1/(\log n)^q$ where $q=1$ in case (i)  and $q=2$ in case (ii) and $L_1$ is some positive constant. Recalling \eqref{Eq:AuxNot} and that if $\rho\sim \Pi$, then $\rho(z) = \rho(  B_{\varepsilon_{L_{1,n}}(z)} ) / \mu( B_{\varepsilon_{L_{1,n}}(z)} )$ for all $z \in [0,1]^d$, define

$$ 
	\qquad  \rho_{0,L_{1,n}}(z):= \rho_0^* \rho_{0}^{L_{1,n}}(z) 
	= \frac{ \rho_0(  B_{\varepsilon_{L_{1,n}}(z)} ) }{ \mu_n( B_{\varepsilon_{L_{1,n}}(z)} )},
$$ 
and  set 
$$
	\rho^*_{0,\nu} := \sum_{ \varepsilon \in \mathcal E_{L_{1,n}} }
	\frac{ \rho_0(  B_{\varepsilon_{L_{1,n}}(z)} )\nu( B_{\varepsilon_{L_{1,n}}(z)} )}
	{ \mu_n( B_{\varepsilon_{L_{1,n}}(z)})}.
$$ 
Similarly define  $\rho_{\nu}^*$, so that $\bar \rho = \rho/\rho^*_{\nu}$. Note that  $\rho_{0,L_{1,n}} $ is a piecewise constant approximation of $\rho_0$.  Using \eqref{trunc:approx} and the triangle inequality,  we have that 
$$ 
	\|  \rho - \rho_0 \|_\infty \leq \| \rho - \rho_{0,L_{1,n}}\|_\infty 
	+  \| \rho_0 - \rho_{0,L_{1,n}}\|_\infty \leq \| \rho - \rho_{0,L_{1,n}}\|_\infty + C_1 2^{-L_{1,n} \beta/d}.
$$
Consider $y_\eps$ and $\rho^*$ such that
\begin{equation*}
	\sum_{l \leq L_{1,n}}\sum_{ \varepsilon \in \mathcal E_l}| y_{\varepsilon} - y_{\varepsilon}^0| 
	\leq \frac{ \bar \epsilon_n }{ 4 C_0\rho_0^* }; 
	\qquad |\rho^* - \rho_0^*|\leq \frac{ \bar \epsilon_n }{ 4C_0 \|  \bar \rho_0\|_\infty};
\end{equation*}
then using \eqref{diff:rhoS} with $S =  \cup_{l \leq L_{1,n}} \mathcal E_l$,
$$
	\| \rho - \rho_{0,L_{1,n}}\|_\infty \leq 2|\rho^* - \rho_0^*| \|  \bar \rho_0\|_\infty  + \rho_0^* 
	\sum_{l \leq L_{1,n}}\sum_{ \varepsilon \in \mathcal E_l}| y_{\varepsilon} - y_{\varepsilon}^0| 
	\leq \bar \epsilon_n.
$$
This implies in particular that writing $\rho_{0L_{1,n}; \nu}^*= \int_{\mathcal Z} \rho_{0,L_{1,n}}(z)d\nu(z)$, then 
$ |\rho_\nu^* - \rho_{0L_{1,n}; \nu}^*| \leq \bar \epsilon_n $ and combining with $\rho_0 \geq c_0$, we obtain
\begin{align*} 
	\| \bar \rho - \bar \rho_{0,L_{1,n}}\|_\infty 
	& = \left\| \frac{ \rho}{\rho_\nu^*} - \frac{ \rho_0}{\rho_{0L_{1,n}; \nu}^* }\right\|_\infty 
	\leq \frac{ 2|\rho_\nu^* - \rho_{0,\nu}^*| \| \rho_0\|_\infty }{ c_0^2}
	+ \frac{ 1 }{ c_0} \| \rho - \rho_{0,L_{1,n}}\|_\infty  
	\leq  \frac{ 2 \bar \epsilon_n}{ c_0^2 } (\| \rho_0\|_\infty+ c_0).
\end{align*}
Finally, setting $\epsilon_n = \bar \epsilon_n a_1$ for some $a_1>0$ and writing $q_l $ for $q_{\varepsilon}$ when $\varepsilon \in \mathcal E_l$ to unburden the notation,
\begin{align*}
 	\Pi( B_{n,2}(\rho_0) ) 
	&\geq 
	\Pi_S(  S = \cup_{l \leq L_{1,n}} \mathcal E_l ) \Pi\left( \max_{l \leq L_{1,n}} 
	\max_{\varepsilon \in \mathcal E_l} |   y_{\varepsilon} - y_{\varepsilon}^0| 
	\leq \bar \epsilon_n 2^{-L_{1,n}} \right) 
	\Pi( |\rho^* - \rho_0^*| \leq \bar \epsilon_n ) \\
 	& \geq
  	\prod_{l\leq L_{1,n}} (1-q_l)^{2^l}\prod_{l=L_{1,n}+1}^{L_n} q_l^{2^l}  e^{ - C 2^{L_{1,n}} \log n }\\
  	&
  	\geq 2^{ - 2t_1 L_{1,n}2^{L_{1,n}}  }e^{ -2^{L_n( 1-t_2) + 1 }} e^{ - C 2^{L_{1,n}} \log n } 
	\geq e^{ - C n \bar \epsilon_n^2 - 2 (n/\log n)^{(1-t_2)_+}}
\end{align*}
for $C>0$. Hence, Condition \eqref{Eq:SmallBall} is verified as soon as $t_2 \geq 2\beta/(2\beta+d)$, as then
$$
	\Pi( B_{n,2}(\rho_0) )\geq e^{ - 2C n \bar \epsilon_n^2}.
$$

	Moving onto the construction of the sieves \eqref{Eq:Sieves}, note that $\Pi( | \rho^* - \rho_0^*| > M_n /\sqrt{n} | D^{(n)}) = o_{P^{(n)}_{\rho_0}}(1) $ so that we can restrict $\rho^*$ to be bounded by $2 \rho_0^*$. Then for $L_{2,n}\in\N$ such that $ 2^{L_{2,n} } \in  n\bar \epsilon_n^2 (L_2, 2L_2 )/\log n$ for some $L_2>0$,
let 
$$
	\mathcal R_n := \{ \rho  = \rho^* \rho^S : \rho^* \in (\rho_0^*/2, 2 \rho_0^*),\ |S| \leq 2^{L_{2,n} }  \}.
$$
Then with $S_l = S \cap \mathcal E_l$, 
\begin{align*}
	\Pi( \mathcal R_n^c \cap \{\rho  = \rho^* \rho^S : \rho^* \in (\rho_0^*/2, 2 \rho_0^*)\} ) 
	& \leq 
	\Pi\Big( \sum_{ l\geq L_{2,n} - \log L_n/\log 2 } |S_l| > 2^{L_{2,n} -1}\Big ) \\
	& \leq \Pi\Big( \sum_{ l\geq L_{2,n} - \log L_n/\log 2 } \tilde S_l> 2^{L_{2,n} -1} \Big)
\end{align*}
where $\tilde S_l\overset{\text{ind}}{\sim} \text{Binomial}(2^l , 2^{-lt_2})$. Chernoff inequality implies that for all $s >0$, with $L_{2,n}' :=  L_{2,n} - \log L_n/\log 2$, 
\begin{align*}
 	\Pi\Big( \sum_{ l\geq L_{2,n}' } |\tilde S_l| > 2^{L_{2,n} -1} \Big) 
	\leq e^{ -s2^{L_{2,n} -1}} e^{ \sum_{l \geq L_{2,n}'} 2^{l(1-t_2)}(e^s-1) }.
\end{align*}
Provided that $t_2\geq  1$, we have  choosing $s = (t_2-1) L_{2,n} \log 2 - \log L_n$,
 \begin{align*}
 	\Pi\Big( \sum_{ l\geq L_{2,n}' } |\tilde S_l| > 2^{L_{2,n} -1} \Big) 
	\leq e^{ -\frac{ (t_2-1) L_{2,n}\log 2^{L_{2,n}+1}}{ 4} } 
	\leq e^{- C n\bar \epsilon_n^2}, 
\end{align*}
as soon as $2^{L_{2,n} } \geq L_2 n \bar \epsilon_n^2/\log n$ with $L_2$ large enough.  If $t_2 < 1$, 
\begin{align*}
 	\Pi\Big( \sum_{ l\geq L_{2,n}' } |\tilde S_l| > 2^{L_{2,n} -1} \Big) 
	\leq e^{ -s2^{L_{2,n} -1}} e^{ 2^{L_n(1-t_2)+1} (e^s -1) }. 
\end{align*}
Moreover, $2^{L_n(1-t_2)} \leq (n/\log n)^{1-t_2} $  and $t_2 \geq 2\beta / (2\beta +d)$ so that 
$$ 
	2^{L_n(1-t_2)}\leq n^{d/(2+d)} (\log n)^{-d/(2+d)} = o ( 2^{L_{2,n}}),
	\quad 2^{L_{2,n}} \geq L_2 n^{d/(2\beta+d)} (\log n)^{2\beta/(2\beta+d)}/2,
$$
so choosing $s = 4$, we verify the sieves condition \eqref{Eq:Sieves},
\begin{align*}
 	\Pi\Big( \sum_{ l\geq L_{2,n}' } |\tilde S_l| > 2^{L_{2,n} -1}\Big ) 
	\leq e^{ - L_{2,n}  2^{L_{2,n} } }\leq e^{ -Cn\bar \epsilon_n^2 }.
\end{align*}
by choosing $L_2$ large enough.

	Finally, we derive the complexity bound \eqref{Eq:MetricEntropy}. Let $\rho_1,\rho_2\in\Rcal_n$ such that $\sum_{\eps \in S} |y_{1\eps} - y_{2\eps} | \leq n^{-2}$ with $\rho_j = \rho_j^*\prod_{\eps \in S} y_{j\eps}$.  Using \eqref{diff:rhoS},
 $$ 
 	\|\rho_1^S - \rho_2^S \|_\infty  \leq 2\|\rho_1^S\|_\infty \sum_{\eps \in S} |y_{1\eps} - y_{2\eps} |,
$$
 and since 
 $$ 
 	\|\rho_1^S\|_\infty 
	\leq \sup_{z\in[0,1]^d} \prod_{\eps \in S,\  z \in B_\eps} \alpha_n(\eps)^{-1}  
	\leq \frac{ 1 }{ \mu_n( B_{\eps_{L_n}(z)}) } \leq n,
$$
we have $ \|\rho_1^S - \rho_2^S \|_\infty  = o( \bar \epsilon_n) $. Thus, given that $y_\eps \leq 1/\alpha_n(\eps)\leq 1/c_1$, there follows 
 $$ 
 	N( \zeta \bar \epsilon_n , \mathcal R_n , \|\cdot \|_\infty ) 
	\lesssim  2^{L_{2,n}} \log n \lesssim L_2 n \bar \epsilon_n^2. 
$$

	An application of Theorem \ref{Theo:GenRandCov} then proves the first claim of Theorem \ref{thm:L1:polya}.
For the posterior contraction rates in $L^1([0,1]^d,\nu)$-norm, note that 
$$
	\|\rho - \rho_0 \|_{L^1([0,1]^d,\nu)}
	= \sum_{\eps \in \mathcal E_{L_n}}\int_{B_\eps}\left| \frac{\rho(B_\eps)}{ \mu_n(B_\eps)}  - \rho_0(z) \right| 
	d\nu(z).
$$
Let $\rho_{0, L_{1,n}}$ be the above piecewise constant approximation of $\rho_0$.  We have $\|\rho_0 - \rho_{0, L_{1,n}}\|_\infty = O(2^{-\beta L_{1,n}/d})$,  and
\begin{align*}
	\|\lambda^{(n)}_{\rho_0}-\lambda_{\rho_{0, L_{1,n}}} \|_{L^1([0,1]^d,\nu)} 
	& =  \sum_{\eps \in \mathcal E_{L_n}}\int_{\mathcal W_n} 1_{\{ Z(x) \in B_\eps\}} 
	|\rho(z(x)) - \rho_{0, L_{1,n}}(z(x))| d\mu_n(x) \\
	&= \sum_{\eps \in \mathcal E_{L_n}}\int_{B_\eps}|\rho(z) - \rho_{0,L_{1,n}}(z)| \mu_n(B_\eps)\\ 
	&= \sum_{\eps \in \mathcal E_{L_n}}\int_{B_\eps}
	|\rho(z) - \rho_{0,L_{1,n}}(z)| \frac{ \mu_n(B_\eps) }{ \nu( B_\eps)}\nu( B_\eps))\\
	& \gtrsim \sum_{\eps \in \mathcal E_{L_n}}\int_{B_\eps}|\rho(z) - \rho_{0,L_{1,n}}(z)| \nu( B_\eps) 
	= \| \rho - \rho_{0, L_{1,n}}\|_{L^1([0,1]^d,\nu)}
\end{align*}
where the last inequality comes from Condition \eqref{cond:split}, together with the upper bound on $\nu$. It follows that $\|\rho - \rho_0 \|_{L^1([0,1]^d,\nu)} \lesssim \|\lambda^{(n)}_{\rho_0}-\lambda^{(n)}_{\rho_{0,L_{1,n}}} \|_1$
which jointly with the first claim terminates the proof. \qed

\end{document}